\documentclass[12pt,a4paper]{amsart}
\usepackage{mathtools,amssymb,amsmath,amsopn,amsthm,graphicx,MnSymbol,centernot,stmaryrd,array}
\usepackage{tikz-cd}
\usepackage{enumitem}
\usepackage{amsaddr}
\usepackage{etoolbox}
\patchcmd{\subsection}{-.5em}{.5em}{}{}
\usetikzlibrary{matrix}
\begin{document}

\newtheorem{definition}{Definition}[section]
\newtheorem{definitions}[definition]{Definitions}
\newtheorem{deflem}[definition]{Definition and Lemma}
\newtheorem{lemma}[definition]{Lemma}
\newtheorem{proposition}[definition]{Proposition}
\newtheorem{theorem}[definition]{Theorem}
\newtheorem{corollary}[definition]{Corollary}
\newtheorem{algo}[definition]{Algorithm}
\theoremstyle{remark}
\newtheorem{rmk}[definition]{Remark}
\theoremstyle{remark}
\newtheorem{remarks}[definition]{Remarks}
\theoremstyle{remark}
\newtheorem{notation}[definition]{Notation}
\theoremstyle{remark}
\newtheorem{example}[definition]{Example}
\theoremstyle{remark}
\newtheorem{examples}[definition]{Examples}
\theoremstyle{remark}
\newtheorem{dgram}[definition]{Diagram}
\theoremstyle{remark}
\newtheorem{fact}[definition]{Fact}
\theoremstyle{remark}
\newtheorem{illust}[definition]{Illustration}
\theoremstyle{remark}
\newtheorem{que}[definition]{Question}
\theoremstyle{definition}
\newtheorem{conj}[definition]{Conjecture}
\newtheorem{scho}[definition]{Scholium}
\newtheorem{por}[definition]{Porism}
\DeclarePairedDelimiter\floor{\lfloor}{\rfloor}

\renewenvironment{proof}{\noindent {\bf{Proof.}}}{\hspace*{3mm}{$\Box$}{\vspace{9pt}}}
\author[Sardar, Kuber]{Shantanu Sardar and Amit Kuber}
\address{Department of Mathematics and Statistics\\Indian Institute of Technology, Kanpur\\ Uttar Pradesh, India}
\email{shantanusardar17@gmail.com, askuber@iitk.ac.in}
\title[On the computation of hammocks for domestic string algebras]{{On the computation of order types of hammocks for domestic string algebras}}

\begin{abstract}
For the representation-theoretic study of domestic string algebras, Schr\"{o}er introduced a version of hammocks that are bounded discrete linear orders. He introduced a finite combinatorial gadget called the bridge quiver, which we modified in the prequel of this paper to get a variation called the arch bridge quiver. Here we use it as a tool to provide an algorithm to compute the order type of an arbitrary closed interval in such hammocks. Moreover, we characterize the class of order types of these hammocks as the bounded discrete ones amongst the class of finitely presented linear orders--the smallest class of linear orders containing finite linear orders as well as $\omega$, and that is closed under isomorphisms, order reversal, finite order sums and lexicographic products.
\end{abstract}

\maketitle

\newcommand\A{\mathcal{A}}
\newcommand\B{\mathfrak{B}}
\newcommand\C{\mathcal{C}}
\newcommand\Pp{\mathcal{P}}
\newcommand\D{\mathcal{D}}
\newcommand\Hamm{\hat{H}}
\newcommand\hh{\mathfrak{h}}
\newcommand\HH{\mathcal{H}}
\newcommand\RR{\mathcal{R}}
\newcommand\Red[1]{\mathrm{R}_{#1}}
\newcommand\HRed[1]{\mathrm{HR}_{#1}}
\newcommand\LL{\mathcal{L}}
\newcommand\M{\mathcal{M}}
\newcommand\Q{\mathcal{Q}}
\newcommand\SD{\mathcal{SD}}
\newcommand\MD{\mathcal{MD}}
\newcommand\SMD{\mathcal{SMD}}
\newcommand\T{\mathcal{T}}
\newcommand\TT{\mathfrak T}
\newcommand\ii{\mathcal I}
\newcommand\UU{\mathcal{U}}
\newcommand\VV{\mathcal{V}}
\newcommand\ZZ{\mathcal{Z}}
\newcommand{\N}{\mathbb{N}} 
\newcommand{\R}{\mathbb{R}}
\newcommand{\Z}{\mathbb{Z}}
\newcommand{\bb}{\mathfrak b}
\newcommand{\qq}{\mathfrak q}
\newcommand{\ch}{\circ_H}
\newcommand{\cg}{\circ_G}
\newcommand{\bua}[1]{\mathfrak b^{\alpha}(#1)}
\newcommand{\falpha}{{\mathfrak{f}\alpha}}
\newcommand{\fgamma}{\gamma^{\mathfrak f}}
\newcommand{\fbeta}{{\mathfrak{f}\beta}}
\newcommand{\bub}[1]{\mathfrak b^{\beta}(#1)}
\newcommand{\bla}[1]{\mathfrak b_{\alpha}(#1)}
\newcommand{\blb}[1]{\mathfrak b_{\beta}(#1)}
\newcommand{\lmin}{\lambda^{\mathrm{min}}}
\newcommand{\lmax}{\lambda^{\mathrm{max}}}
\newcommand{\xmin}{\xi^{\mathrm{min}}}
\newcommand{\xmax}{\xi^{\mathrm{max}}}
\newcommand{\lbmin}{\bar\lambda^{\mathrm{min}}}
\newcommand{\lbmax}{\bar\lambda^{\mathrm{max}}}
\newcommand{\ff}{\mathfrak f}
\newcommand{\cc}{\mathfrak c}
\newcommand{\dd}{\mathfrak d}
\newcommand{\sqsf}{\sqsubset^\ff}
\newcommand{\rr}{\mathfrak r}
\newcommand{\pp}{\mathfrak p}
\newcommand{\uu}{\mathfrak u}
\newcommand{\vv}{\mathfrak v}
\newcommand{\ww}{\mathfrak w}
\newcommand{\xx}{\mathfrak x}
\newcommand{\yy}{\mathfrak y}
\newcommand{\zz}{\mathfrak z}
\newcommand{\MM}{\mathfrak M}
\newcommand{\mm}{\mathfrak m}
\newcommand{\sbq}{\mathfrak s}
\newcommand{\tbq}{\mathfrak t}
\newcommand{\Spec}{\mathbf{Spec}}
\newcommand{\Br}{\mathbf{Br}}
\newcommand{\sk}[1]{\{#1\}}
\newcommand{\Prime}{\mathbf{Pr}}
\newcommand{\Parent}{\mathbf{Parent}}
\newcommand{\Uncle}{\mathbf{Uncle}}
\newcommand{\Cousin}{\mathbf{Cousin}}
\newcommand{\Nephew}{\mathbf{Nephew}}
\newcommand{\Sibling}{\mathbf{Sibling}}
\newcommand{\uc}{\mathrm{uc}}
\newcommand{\MCP}{\mathrm{MCP}}
\newcommand{\MSCP}{\mathrm{MSCP}}
\newcommand{\TTT}{\widetilde{\T}}
\newcommand{\la}{l}
\newcommand{\ra}{r}
\newcommand{\lb}{\bar{l}}
\newcommand{\rb}{\bar{r}}
\newcommand{\tBa}{\varepsilon^{\mathrm{Ba}}}
\newcommand{\brac}[2]{\langle #1,#2\rangle}
\newcommand{\braket}[3]{\langle #1\mid #2:#3\rangle}
\newcommand{\fin}{fin}
\newcommand{\inff}{inf}
\newcommand{\Zg}{\mathrm{Zg}(\Lambda)}
\newcommand{\Zgs}{\mathrm{Zg_{str}}(\Lambda)}
\newcommand{\STR}[1]{\mathrm{Str}(#1)}
\newcommand{\dmod}{\mbox{-}\operatorname{mod}}
\newcommand{\Ba}[1]{\mathrm{Ba}(#1)}
\newcommand{\HQ}{\mathcal{HQ}^\mathrm{Ba}}
\newcommand{\bHQ}{\overline{\mathcal{HQ}}^\mathrm{Ba}}
\newcommand\Af{\mathcal{A}^{\ff}}
\newcommand\AAf{\bar{\mathcal{A}}^{\ff}}
\newcommand\Hf{\mathcal{H}^{\ff}}
\newcommand\Rf{\mathcal{R}^{\ff}}
\newcommand\Tf{\T^{\ff}}
\newcommand\Uf{\mathcal{U}^{\ff}}
\newcommand\Sf{\mathcal{S}^{\ff}}
\newcommand\Vf{\mathcal{V}^{\ff}}
\newcommand\Zf{\mathcal{Z}^{\ff}}
\newcommand\bVf{\overline{\mathcal{V}}^{\ff}}
\newcommand\bTf{\overline{\mathcal{T}}^{\ff}}
\newcommand{\fmin}{\xi^{\mathrm{fmin}}}
\newcommand{\fmax}{\xi^{\mathrm{fmax}}}
\newcommand{\xif}{\xi^\ff}

\newcommand{\LOfp}{\mathrm{LO}_{\mathrm{fp}}}
\newcommand{\dLOfpb}{\mathrm{dLO}_{\mathrm{fp}}^{\mathrm{b}}}
\newcommand\Tl{\mathbf{T}}
\newcommand\Tla{\mathbf{T}_{\la}}
\newcommand\Tlb{\mathbf{T}_{\lb}}
\newcommand\Ml{\mathbf{M}}
\newcommand\Mla{\mathbf{M}_{\la}}
\newcommand\Mlb{\mathbf{M}_{\lb}}
\newcommand\OT{\mathcal{O}}
\newcommand\LO{\mathbf{LO}}
\newcommand\rad{\mathrm{rad}_\Lambda} 
\newcommand{\rk}[1]{\mathrm{rk}(#1)}
\newcommand{\wid}[1]{\mathrm{wd}(#1)}

\section{Introduction}
This paper is the sequel of \cite{SK} where the authors embarked on the journey towards the computation of the order types of certain linear orders known as hammocks for domestic string algebras--that paper introduces a new finite combinatorial tool called the arch bridge quiver and investigates its properties. Recall that a string algebra $\Lambda$ over an algebraically closed field $\mathcal K$ is presented as a certain quotient $\mathcal KQ/\langle\rho\rangle$ of the path algebra $\mathcal KQ$ of a quiver $Q$, where $\rho$ is a set of monomial relations. The vertices of the Auslander-Reiten (A-R) quiver were classified essentially in \cite{GP} in terms of certain walks on the quiver known as strings and bands, whereas its arrows were classified by Ringel and Butler \cite{BR}. 

A string algebra is domestic if it has only finitely many bands. For two bands $\bb_1$ and $\bb_2$, say $\bb_1$ and $\bb_2$ \emph{commute} if there are cyclic permutations $\bb'_1,\bb''_1$ of $\bb_1$ and $\bb'_2,\bb''_2$ of $\bb_2$ such that $\bb'_1\bb'_2$ and $\bb''_2\bb''_1$ are strings. If two bands commute in a string algebra then it is necessarily non-domestic by \cite[Corollary~3.4.1]{GKS}.

For a vertex $v$ of the quiver $Q$ the simplest version of hammock, denoted $H_l(v)$, is the collection of all strings starting at the vertex. It can be equipped with an order $<_l$ such that $\xx<_l\yy$ in $H_l(v)$ guarantees the existence of a canonical graph map $M(\xx)\to M(\yy)$ between the corresponding string modules. Schr\"{o}er \cite[\S~2.5]{SchroerThesis} showed that $(H_l(v),<_l)$ is a bounded discrete linear order. Prest and Schr\"{o}er showed that it has finite m-dimension (essentially \cite[Theorem~1.3]{PS}). For more flexibility we consider the hammock $H_l(\xx_0)$ for an arbitrary string $\xx_0$, which is the collection of all strings which admit $\xx_0$ as a left substring--this too is a bounded discrete linear order with respect to the induced order $<_l$, say with $\mm_1(\xx_0)$ and $\mm_{-1}(\xx_0)$ as the maximal and minimal elements respectively. For technical reasons it is necessary to write the hammock as a union $H_l(\xx_0)=H_l^1(\xx_0)\cup H_l^{-1}(\xx_0)$, where $H_l^1(\xx_0)$ is the interval $[\xx_0,\mm_1(\xx_0)]$ and $H_l^{-1}(\xx_0)=[\mm_{-1}(\xx_0),\xx_0]$. Given $\xx<_l\yy$ in $H_l(\xx_0)$, the goal of this paper is to compute the order type of the interval $[\xx,\yy]$ via the assignment of a label, which we call a `term', to the path from $\xx$ to $\yy$. What we call a term in this paper is referred to as a `real term' in \cite[\S~2.4]{GKS}. 

Some classes of finite Hausdorff rank linear orders were studied in \cite{AKG}. The class $\LOfp$ \cite[\S~5]{AKG} is the smallest class of linear orders containing $1$ that is closed under $+$, $\omega\times\mbox{-}$ and $\omega^*\times\mbox{-}$ operations. Its subclass consisting of bounded discrete linear orders is denoted by $\dLOfpb$. The main results of this paper (Corollary \ref{hammocksaredlofpb} and Proposition \ref{termreverse}) show that $\dLOfpb$ is precisely the class of order types of hammocks for domestic string algebras.  

Fix a domestic string algebra $\Lambda$, a string $\xx_0$ and $i\in\{1,-1\}$. Motivated by the notion of the \emph{bridge quiver} introduced by Schr\"{o}er for the study of hammocks for domestic string algebras, the authors introduced the notion of the \emph{extended arch bridge quiver} in \cite{SK}, denoted $\HQ_i(\xx_0)$, that lies between the extended bridge quiver and the extended weak bridge quiver. For $i\in\{1,-1\}$, the paths in $\HQ_i(\xx_0)$ generate all strings in $H_l^i(\xx_0)$ in the precise sense described in \cite[Lemma~3.3.4]{GKS}. 

In this paper, we introduce the notion of a `decorated tree', denoted $\Tf_i(\xx_0)$, that linearizes the extended arch bridge quiver $\HQ_i(\xx_0)$ from the viewpoint of the string $\xx_0$ placed at the root; here the adjective decorated refers to the assignment of either a $+$ or $-$ sign to each non-root vertex and, for each vertex, a linear ordering of the children with the same sign. The data present as part of the decoration of the tree allows us to inductively associate terms to individual vertices of $\Tf_i(\xx_0)$ so that the term associated with the root, denoted $\mu_i(\xx_0;\mm_i(\xx_0))$, labels $H_l^i(\xx_0)$. The following flowchart concisely depicts the steps of the algorithm. 
$$(\Lambda,i,x_0)\mapsto\HQ_i(x_0)\mapsto\Tf_i(x_0)\mapsto\mu_i(\xx_0;\mm_i(\xx_0))\mapsto\mathcal O(\mu_i(\xx_0;\mm_i(\xx_0)))\cong H_l^i(x_0).$$

To further explain a rough idea behind the proof we need some terminology and notation. Recall from \cite{SK} that for a string $\xx$ of positive length, we say $\theta(\xx)=1$ if its first syllable is an inverse syllable, otherwise we say $\theta(\xx)=-1$.
\begin{definition}
Given $\xx(\neq\xx_0)\in H_l(\xx_0)$, we say that $P(\xx):=(\xx_0;\xx_1,\hdots,\xx_n)$ is the \emph{standard partition} of $\xx$ if
\begin{enumerate}
    \item we have $\xx=\xx_n\hdots\xx_1\xx_0$ where $|\xx_k|>0$ for each $1\leq k\leq n$;
    \item for each $1\leq k<n$ the string $\xx_k\hdots\xx_1\xx_0$ forks;
    \item the only proper left substrings of $\xx$ in $H_l(\xx_0)$ which fork are the ones described above.
\end{enumerate}
We say that the \emph{signature type} of $\xx$, denoted $\Theta(\xx_0;\xx)$, is $(\theta(\xx_1),\hdots,\theta(\xx_n))$. In case $|\xx_0|=0$, we denote $\Theta(\xx_0;\xx)$ by $\Theta(\xx)$.
\end{definition}

\begin{rmk}\label{uniquesignaturetype}
If $\xx,\xx'\in H_l(\xx_0)$ fork then $\Theta(\xx_0;\xx)\neq\Theta(\xx_0;\xx')$.
\end{rmk}

For $i\in\{1,-1\},\xx\in H_l(\xx_0),\xx\neq\xx_0$, say that $\Theta(\xx_0;\xx)=C(i)$ if it is of the form $(i,\hdots,i)$. Say that $\Theta(\xx_0;\xx)=F(i)$ if it is of the form $(i,-i,\hdots,-i)$, where $-i$ may not appear; here $C$ and $F$ stand for the words constant and flipped respectively.

The simplest possible terms are associated with the flipped signature types, and hence lot of the effort of the main proof goes into constructing strings whose signature types are so. We use these as starting points for building strings with more complex signature types. The interaction between signature types, left $\N$-strings and terms is an essential feature of this proof.

Continuing from \cite{SK} we provide a set of (counter)examples of algebras with peculiar properties which could be useful elsewhere. Even though there is lot of literature and explicit proofs for gentle algebras, we believe that these two papers are the first systematic attempt to study the combinatorics of strings for all domestic string algebras.

The rest of the paper is organized as follows. The concept of terms together with their associated order types is introduced in \S\ref{motivaterms}, the example of a particular domestic string algebra, namely $\Lambda_2$, is discussed in detail. Even though the bridge quiver of a domestic string algebra does not contain any directed cycle, its underlying undirected graph could still contain cycles. After fixing a string $\xx_0$ and parity $i$, we linearize/separate all paths in the extended arch bridge quiver $\HQ_{i}(\xx_0)$ with the help of the concept of decorated trees introduced in \S\ref{secTree-basics}. The results in \S\ref{longshortpair}, \S\ref{longelements} and \S\ref{upsiloncharacter} help to classify strings in $H^l_i(\xx_0)$ into various classes. Recall that there are only finitely many H-reduced strings in a domestic string algebra \cite[Corollary~5.8]{SK}. In \S\ref{Hforkingstring}, we characterize all H-reduced forking strings using the vertices of the decorated tree and then use them to build all H-reduced strings in \S\ref{secbuild}. The signature types of these H-reduced strings are computed in \S\ref{signtype}. Finally the algorithm (Algorithm \ref{termformation}) for computation of the term associated with a closed interval in the hammock, and hence that of the order type, is completed in \S\ref{brquivtoterm}.

The reader is advised to refer to \cite{SK} for any terminology/notation not explained in this paper.

\subsection*{Acknowledgements}
The first author thanks the \emph{Council of Scientific and Industrial Research (CSIR)} India - Research Grant No. 09/092(0951)/2016-EMR-I for the financial support.

\subsection*{Data Availability Statement}
Data sharing not applicable to this article as no datasets were generated or analysed during the current study.

\section{Labelling paths in discrete linear orders}\label{motivaterms}
The class $\dLOfpb$ consisting of bounded discrete linear orders is of particular interest because of the next result.

\begin{lemma}\label{HBDLO}(cf. \cite[\S 2.5]{Sch})
If $\xx_0$ is a string then $(H_l(\xx_0),<_l)$ is a toset. Its maximal element is $\mm_1(\xx_0)$ and the minimal element is $\mm_{-1}(\xx_0)$, where for each $j\in\{1,-1\}$ $\mm_j(\xx_0)$ is the maximal length string in $H_l(\xx_0)$ such that $\delta(\xx_0;\mm_j(\xx_0))=j$.
\end{lemma}
 
Consider any $\vv\in H_l(\xx_0)$. Recall from \cite[\S 2.4]{GKS} that $\la(\vv)$ is the immediate successor of $\vv$ with respect to $<_l$, if exists, that satisfies $|\vv|<|\la(\vv)|$. We can inductively define $\la^{n+1}(\vv):=\la(\la^n(\vv))$, if exists. If $\la^n(\vv)$ exists for all $n\geq0$ then the limit of the increasing sequence $\brac 1 \la(\vv):=\lim_{n\to\infty}\la^n(\vv)$ is a left $\N$-string. On the other hand, if there is a maximum $N\geq1$ such that $\la^N(\vv)$ exists but $\la^{N+1}(\vv)$ does not exist then we define $[1,\la](\vv):=\la^N(\vv)$. Note that $\brac 1\la(\vv)\notin H_l(\xx_0)$ but $[1,\la](\vv)\in H_l(\xx_0)$.

Recall that any left $\N$-string in a domestic string algebra is an almost periodic string, i.e., is of the form $^\infty\bb\uu$ for some primitive cyclic word $\bb$ and some finite string $\uu$ such that the composition $\bb\uu$ is defined. Let $\Hamm_l(\xx_0)$ denote the completion of $H_l(\xx_0)$ by all left $\N$-strings $^\infty\bb\uu$ such that $\bb\uu\in H_l(\xx_0)$. The definition of the ordering $<_l$ can be easily extended to include infinite left $\N$-strings so that $\Hamm_l(\xx_0)$ is also a bounded total order with same extremal elements as that of $H_l(\xx_0)$.

Note that any left $\N$-string in $\Hamm_l(\xx_0)$ is ``far away'' from its endpoints. To be more precise if $\brac 1\la(\vv)$ exists then the interval $[\brac 1\la(\vv),\mm_1(\xx_0)]$ in $\Hamm_l(\xx_0)$ contains infinitely many elements of $H_l(\xx_0)$, where $\brac 1\la(\vv)<_l\mm_1(\xx_0)$, where $\mm_1(\xx_0)$ is a finite string. Thus the left $\N$-string $\brac 1\la(\vv)$ can also be approached from the right, with respect to $<_l$, as a limit of some decreasing sequence in $H_l(\xx_0)$. 

Now we recall the concept of fundamental solution of the following equation from \cite[\S 2.4]{GKS}:
\begin{equation}\label{lalbeqn}
    \brac 1\lb(\xx)=\brac 1\la(\vv).
\end{equation}
If one solution to the above equation exists then infinitely many solutions exist for if $\yy'$ is a solution then so is $\lb(\yy')$ and the \emph{fundamental solution} is a solution of minimal length. If $\vv$ is a left substring of $\yy$, then we rewrite Equation \eqref{lalbeqn} as $\yy=\brac\lb\la(\vv)$, otherwise we write $\vv=\brac\la\lb(\yy)$. If $\yy=\brac\lb\la(\vv)$, then we also say that $\yy=\brac\lb\la(\la^n(\vv))$ for any $n\in\N$.

Similarly if $\brac 1\la(\vv)$ does not exist but $\la(\vv)$ exists then we may ask if the equation
\begin{equation}\label{lalbeqntorsion}
    [1,\lb](\xx)=[1,\la](\vv)
\end{equation}
has a solution. If one solution exists then finitely many such solutions exist for if $\yy'$ is a solution and $[1,\lb](\vv)\neq\lb(\yy')$ then $\lb(\yy')$ is also a solution. As explained above we can define a \emph{fundamental solution} of Equation \eqref{lalbeqntorsion} and introduce the notations $[\la,\lb]$ or $[\lb,\la]$ depending on the comparison of lengths.

In the above two paragraphs we introduced new expressions using two kinds of bracket operations, namely $\brac\lb\la$ and $[\lb,\la]$. We say that such expressions are `terms'--the precise definition is given below.

If $\yy=\brac\lb\la(\vv)$ then the term $\brac\lb\la$ labels the path from $\vv$ to $\yy$ in $\Hamm_l(\xx_0)$. Let us describe the path from $\vv$ to any element $\uu$ in the interval $[\vv,\yy]$. For $\uu$ there exists some $n\geq 1$ such that either $\uu=\la^n(\vv)$ or $\uu=\lb^n(\yy)$. In the former case, we know the label of the required path, namely $\la^n$, whereas in the latter case, we use concatenation operation on terms, written juxtaposition, to label the path as $\lb^n\brac\lb\la$. Here we acknowledge that the only path to reach such $\uu$ from $\vv$ has to pass through $\yy$ which is captured using the bracket term. In other words, the bracket term $\brac\lb\la$ describes a `ghost path' which passes through each $\la^n(\vv)$ to reach the limit $\brac 1\la(\vv)$, and then directly jumps to $\yy$ avoiding all elements of the interval $[\brac 1\la(\vv),\yy]$--such elements can only be accessed through reverse paths from $\yy$.

Once we have concatenation operation on terms as well as the bracket operation, it is natural to ask if more complex terms could be constructed by their combination. For example, we may ask if the expression $\brac\lb\la^2$ labels a path between two points in some domestic string algebra. In case the answer is affirmative, we can also ask the same question for the expression $\brac 1{\brac\lb\la}:=\lim_{n\to\infty}\brac\lb\la^n$. Further we can also ask if there are solutions to equations of the form
\begin{eqnarray*}
\brac 1\la(\xx)&=&\brac 1{\brac\lb\la}(\vv),\\
\brac 1{\brac\la\lb}(\xx)&=&\brac 1{\brac\lb\la}(\vv)\\
\brac 1{[\la,\lb]}(\xx)&=&\brac 1{[\lb,\la]}(\vv).
\end{eqnarray*}
In summary, we need to close the collection of terms under concatenation and both bracket operations, whenever such a term labels a path between two elements in some hammock in some domestic string algebra.

\begin{example}\label{Ex}
Consider the algebra $\Lambda_2$ presented as  
\begin{figure}[h]
    \centering
    \begin{tikzcd}
    v_1 \arrow[r, "a", bend left] \arrow[r, "b"', bend  right] & v_2 \arrow[r, "c"] & v_3 \arrow[r, "d", bend left] \arrow[r, "e"', bend right] & v_4
    \end{tikzcd}
    \caption{$\Lambda_2$ with $\rho = \{cb, dc\}$}
    \label{Lambda2}
\end{figure}
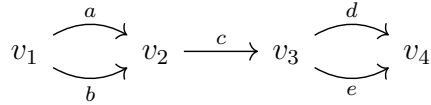
Figure \ref{ARcompletion} shows all connected components of the A-R quiver of $\Lambda_2$ along with some infinite dimensional string modules. The exceptional component is a copy of $\mathbf{Z}A^\infty_\infty$-component with a hole in the middle in and rest of the components are tubes.
\begin{figure}[h]
    \begin{center}
    \includegraphics[width=1.0\textwidth]{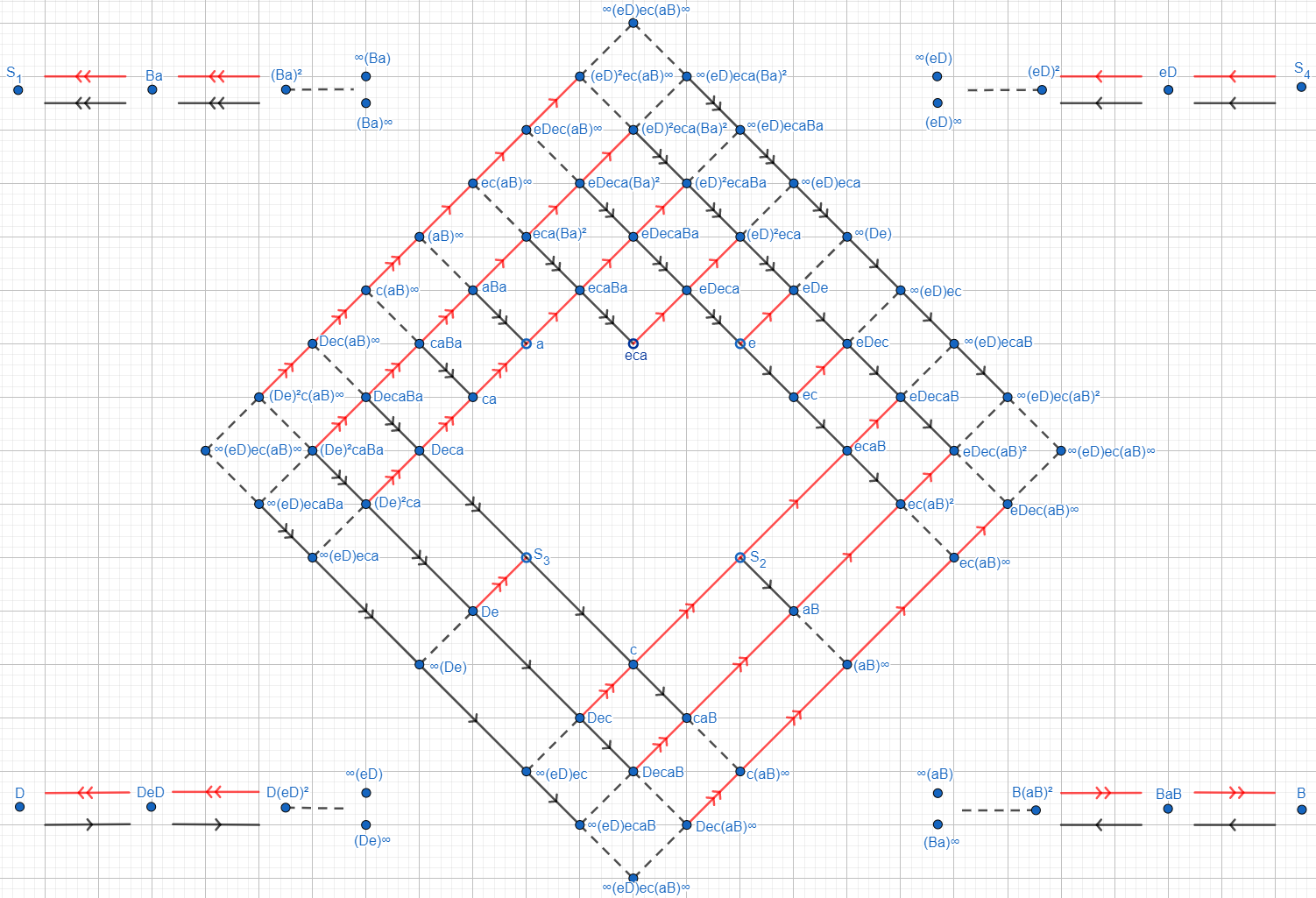}
\end{center}
    \caption{AR components for string modules in $\Lambda_2$}
    \label{ARcompletion}
\end{figure}
Note that red $\hookrightarrow$ and $\twoheadrightarrow$ arrows denote the operators $\la$ and $\lb$ respectively whereas black $\hookrightarrow$ and $\twoheadrightarrow$ arrows denote the operators $\ra$ and $\rb$ respectively. The module $S_j$ represents the simple modules at vertex $v_j$.

Now $\brac{1}{\la}(eca)=\ ^\infty(eD)eca=\brac{1}{\lb}(a)$. Furthermore it is easy to verify that $\brac1{\brac\lb\la}(eca)=\ ^\infty(Ba)=\brac1\lb(1_{(v_1,i)})$, and hence $\brac{\brac\lb\la}\lb(1_{(v_1,i)})=eca$. In Examples \ref{maxtermcompute} and \ref{intermediatetermcompute} we will compute the terms, i.e., labels for the paths from $S_1$ to $eca$ and $ca$ respectively without the involvement of any infinite string.

After connecting all possible rays and corays (in the sense of \cite[\S~5]{R}) consisting of string modules in (a part of) the A-R quiver of $\Lambda_2$ we get Figure \ref{Hammocks}. Here the vertical red colored lines (resp. the horizontal black colored lines) represent hammocks $H_l(v)$ (resp. $H_r(v)$) for different vertices $v$. In fact we also show the infinite dimensional string modules, and thus show $\Hamm_l(v)$ and $\Hamm_r(v)$. For example, the rightmost vertical line represent $\Hamm_l(v_1)$ and the lowermost horizontal line represent $\Hamm_r(v_4)$.
\begin{figure}[ht]
    \begin{center}
    \includegraphics[width=.9\linewidth]{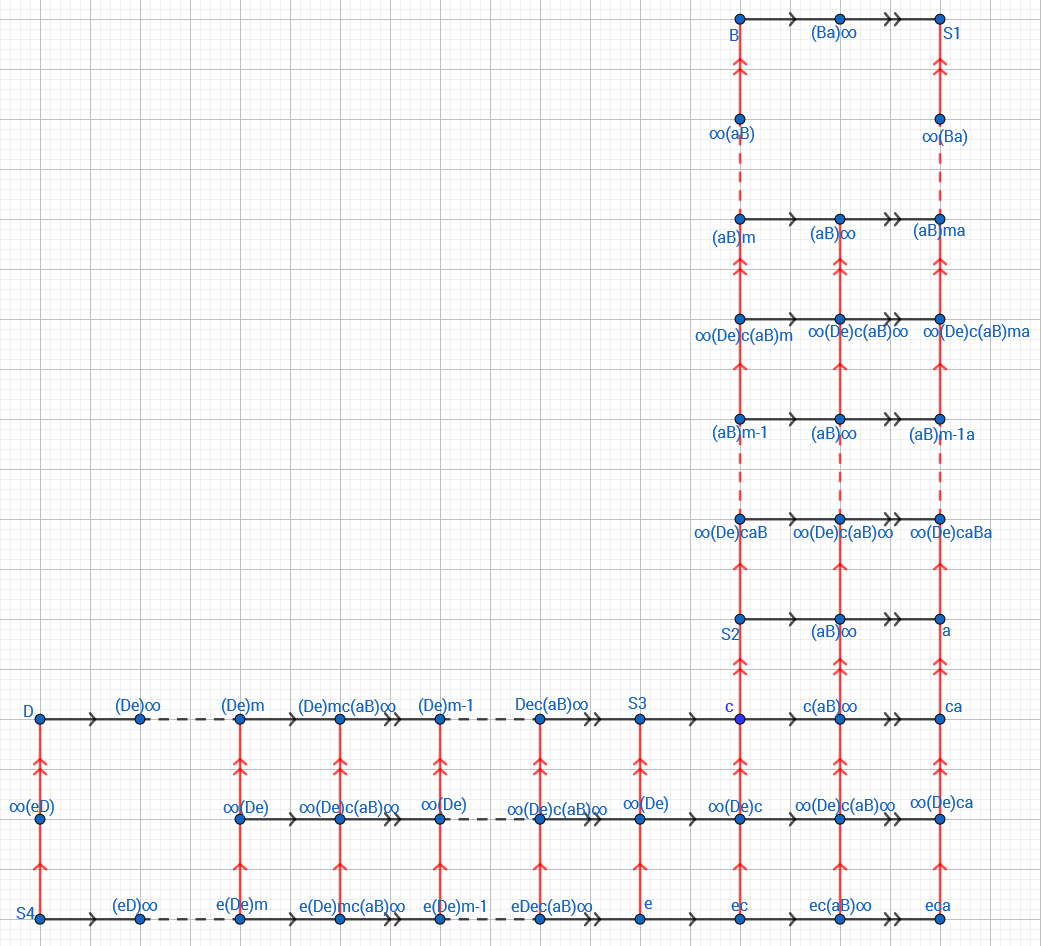}
    \end{center}
    \caption{Hammocks for $\Lambda_2$}
    \label{Hammocks}
\end{figure}
\end{example}

Now we are ready to formally define terms. Starting with symbols $\la,\lb$ we simultaneously inductively define the class of \emph{simple $\la$-terms}, denoted $\Tla$ and the class of \emph{simple $\lb$-terms}, denoted $\Tlb$ as follows:
\begin{itemize}
    \item $[\lb,\la]\in\Tla$ and $[\la,\lb]\in\Tlb$;
    \item if $\tau\in\Tla$ and $\mu\in\Tlb$ then $\brac\mu\tau\in\Tla$ and $\brac\tau\mu\in\Tlb$;
    \item if $\tau,\tau'\in\Tla$ and $\mu,\mu'\in\Tlb$ then $\tau'\tau\in\Tla$ and $\mu'\mu\in\Tlb$.
\end{itemize}
Set $\Tl:=\Tla\cup\Tlb$ and call an element of $\Tl$ a \emph{simple term}.

\begin{rmk}\label{lbarlbijecorespond}
Both $\Tla$ and $\Tlb$ are in bijective correspondence with $\mathbf{A}3\mathbf{ST}$ of \cite[Definition~7.1]{AKG} in view of \cite[Theorem~7.2]{AKG}.
\end{rmk}

Define a map $\ii:\Tl\to\Tl$ as follows:
\begin{itemize}
    \item $\ii(\la):=\lb$ and $\ii(\lb):=\la$;
    \item $\ii(\brac\mu\tau):=\brac{\ii(\tau)}{\ii(\mu)}$;
    \item $\ii([\lb,\la]):=[\la,\lb]$ and $\ii([\la,\lb]):=[\lb,\la]$;
    \item $\ii(\nu'\nu):=\ii(\nu')\ii(\nu)$.
\end{itemize}

\begin{rmk}\label{lbarlcorrelation}
The map $\ii:\Tl\to\Tl$ is an involution such that for each $\nu\in\Tl$, $\nu\in\Tla$ if and only if $\ii(\nu)\in\Tlb$.
\end{rmk}

Suppose $\tau,\tau_1,\tau_2\in\Tla$ and $\mu,\mu_1,\mu_2\in\Tlb$. Let $\approx_\la$ be the equivalence relation on $\Tla$ generated by $$\brac\mu\tau\tau\approx_\la\ii(\mu)\brac\mu\tau\approx_\la\brac\mu\tau\approx_\la\brac\mu{\tau^m}\approx_\la\brac{\mu^m}\tau, \ \brac\mu{[\lb,\la]}\approx_{\la}\brac\mu\la,\ \brac{[\la,\lb]}\tau\approx_{\la}\brac\lb\tau$$ and $\approx_{\lb}$ be the equivalence relation on $\Tlb$ generated by $$\brac\tau\mu\mu\approx_{\lb}\ii(\tau)\brac\tau\mu\approx_{\lb}\brac\tau\mu\approx_{\lb}\brac\tau{\mu^m}\approx_{\lb}\brac{\tau^m}\mu,\ \brac{[\lb,\la]}\mu\approx_{\lb}\brac\la\mu,\ \brac\tau{[\la,\lb]}\approx_{\lb}\brac\tau\lb$$ for each $m\geq2$.

To make these into congruence relations, whenever $\tau_1\approx_{\la}\tau_2$ and $\mu_1\approx_{\lb}\mu_2$ we further need the following identifications: $$\tau_2\tau\approx_{\la}\tau_1\tau,\ \tau\tau_2\approx_{\la}\tau\tau_1,\ \mu_2\mu\approx_{\lb}\mu_1\mu,\ \mu\mu_2\approx_{\lb}\mu\mu_1,$$ $$\brac{\mu_1}{\tau_1}\approx_{\la}\brac{\mu_2}{\tau_2},\ \brac{\tau_1}{\mu_1}\approx_{\lb}\brac{\tau_2}{\mu_2}.$$

\begin{rmk}\label{lLcoincidence}
The relation $\approx_{\la}$ on $\Tla$ coincides with $\approx_L$ of \cite[Definition~6.3]{AKG} on $\mathbf{A}3\mathbf{ST}$.
\end{rmk}

Say that $\nu\in\Tl$ is \emph{irreducible} if either $\nu\in\{\la,\lb\}$ or $\nu$ is of the form $\brac\mu\tau$.

Given $\tau\in\Tla$ and $\mu\in\Tlb$, say that 
\begin{itemize}
    \item $\mu\tau$ is a \emph{valid concatenation} if $\ii(\tau)\mu\approx_{\lb}\ii(\tau)$;
    \item $\tau\mu$ is a \emph{valid concatenation} if $\ii(\mu)\tau\approx_{\la}\ii(\mu)$.
\end{itemize}

Say that a finite concatenation $\tau_n\hdots\tau_2\tau_1$ of simple terms is a \emph{mixed $\la$-term} (resp. \emph{mixed $\lb$-term}) if
\begin{itemize}
    \item $\tau_k\in\Tla$ if and only if $k$ is odd (resp. even);
    \item $\tau_{k+1}\tau_k$ is a valid concatenation for each $1\leq k<n$.
\end{itemize}
Let $\Mla$ and $\Mlb$ denote the classes of all mixed $\la$ and $\lb$-terms respectively. Set $\Ml:=\Mla\cup\Mlb$ and say that a \emph{term} is an element of $\Ml$.

The congruence relations $\approx_{\la}$ and $\approx_{\lb}$ are naturally extended to $\Mla$ and $\Mlb$ respectively.

Associate to each $\nu\in\Tl$ its order type, denoted $\OT(\nu)$, as follows:
\begin{itemize}
    \item $\OT(\la)=\OT(\lb)=\OT([\lb,\la])=\OT([\la,\lb]):=\mathbf{1}$;
    \item $\OT(\brac\mu\tau):=\omega\times\OT(\tau)+\omega^*\times\OT(\mu)$;
    \item $\OT(\tau'\tau):=\OT(\tau)+\OT(\tau')$.
\end{itemize}
Clearly $\OT$ takes values in $\dLOfpb$.

\begin{rmk}\label{suffixinvalidconcatenation}
Suppose $\tau\in\Tla$ and $\mu\in\Tlb$. If $\mu\tau$ is a valid concatenation then $\omega^*\times\OT(\mu)$ is a proper suffix of $\OT(\tau)$.
\end{rmk}

In view of the above remark, the map $\OT$ extends to $\Ml$ by setting $$\OT(\tau_n\hdots\tau_2\tau_1):=\OT(\tau_1).$$

\begin{rmk}
If $\nu\in\Tl$ is irreducible then $\wid{\OT(\nu)}\leq2$, where $\wid{L}$ denotes the width of $L\in\LOfp$ as defined in \cite[Definition~8.1]{AKG}.
\end{rmk}

In view of Remarks \ref{lbarlbijecorespond} and \ref{lLcoincidence} the next result follows from \cite[Corollary~9.4]{AKG}.
\begin{theorem}
Suppose $\tau,\tau'\in\Mla$ and $\mu,\mu'\in\Mlb$. Then
\begin{itemize}
    \item $\OT(\tau)\cong\OT(\tau')$ if and only if $\tau\approx_{\la}\tau'$;
    \item $\OT(\mu)\cong\OT(\mu')$ if and only if $\mu\approx_{\lb}\mu'$.
\end{itemize}
\end{theorem}

Before we close the section, let us define a natural notation, $\ast$, on terms.
\begin{enumerate}
    \item $\la\ast1:=\la$ and $\la\ast(-1):=\lb$.
    \item $\brac\tau\eta\ast1:=\brac\tau\eta$ and $\brac\tau\eta\ast(-1):=\brac\eta\tau$.
    \item $[\tau,\eta]\ast1:=[\tau,\eta]$ and $[\tau,\eta]\ast(-1):=[\eta,\tau]$.
\end{enumerate}

\section{Long and short pairs}\label{longshortpair}
Recall the notion of an H-pair from \cite[\S~5]{SK}.
\begin{rmk}\label{locatHreduction}
Suppose $(\yy_1,\yy_2)$ is an H-pair. Then $N(\bb(\yy_1,\yy_2),\yy_2\ch\yy_1)=1$ if and only if there is a partition ${\yy_2\bb\yy_1}=\yy'_2\bb'\yy'_1$, where $\bb'$ is a cyclic permutation of $\bb$ and the first syllable of $\yy'_2$ is not a syllable of $\bb$, such that $\bb'\yy'_1\nequiv_H\yy'_1$.
\end{rmk}

Suppose $\bb$ is a band, $\yy$ is a string and $i\in\{1,-1\}$. Say that the pair $(\bb,\yy)$ is an \emph{$i$-left valid pair} (\emph{$i$-LVP} for short) if $\bb_\yy\yy$ is a string for some cyclic permutation $\bb_\yy$ of $\bb$ such that $\theta(\bb_\yy)=i$. Dually say that the pair $(\bb,\yy)$ is a \emph{$i$-right valid pair} (\emph{$i$-RVP} for short) if $(\bb^{-1},\yy^{-1})$ is a $(-i)$-LVP.

\begin{example}
Consider the algebra $\Lambda^{(iii)}$ from \cite[Example~4.24]{SK}. The only bands there are $\bb:=cbaDEF$ and $\bb^{-1}$. Choose $\yy:=D$. Since $\bb$ has two different cyclic permutations $\bb_1:=\bb$ and $\bb_2:=DEFcba$ with $\theta(\bb_1)=1$ and $\theta(\bb_2)=-1$ such that $\bb_1\yy$ and $\bb_2\yy$ are strings, the pair $(\bb,\yy)$ is both $1$-LVP and $-1$-LVP.   
\end{example}

\begin{rmk}
\cite[Proposition~4.23]{SK} guarantees the uniqueness of the cyclic permutation $\bb_\yy$ for any $i$-LVP $(\bb,\yy)$.
\end{rmk}

We drop the reference to $i$ in $i$-LVP if either it is clear from the context or if $i$ does not play any role in the statement.

\begin{proposition}\label{validpairbandcopy}
Suppose $(\bb,\yy)$ is an LVP and $\bb\in\B(\yy)$ then $\yy=\bb_\yy\yy'$ for some $\yy'$.
\end{proposition}
\begin{proof}
Since $N(\bb,\yy)>0$, $\yy$ can be written as $\yy=\yy_2\bb'\yy_1$ for some cyclic permutation $\bb'$ of $\bb$. If $\yy_2$ is not a right substring of a finite power of $\bb_\yy$ then there exists a non-trivial cycle in the bridge quiver, i.e., a meta-band passing through $\bb$ by \cite[Lemma~3.3.4]{GKS}, which contradicts the domesticity of the string algebra by \cite[Proposition~3.4.2]{GKS}.
\end{proof}

\begin{definition}
Say that an LVP $(\bb,\yy)$ is a \emph{short pair} if $\yy\equiv_H\bb_\yy\yy$; otherwise say that $(\bb,\yy)$ is a \emph{long pair}.
\end{definition}

\begin{rmk}
An LVP $(\bb,\yy)$ with $|\yy|>0$ is long if and only if $(\bb,\zz_r(\yy))$ is long.
\end{rmk}

\begin{rmk}
Suppose $(\bb,\yy)$ is an LVP. If $(\bb,\yy)$ is short then $\xx\yy\mapsto\xx\bb_\yy\yy:H_l(\yy)\to H_l(\bb_\yy\yy)$ is an isomorphism.
\end{rmk}

For an LVP $(\bb,\yy)$, let $\bar\ww_l(\bb,\yy)$ be the longest common right substring (possibly of length $0$) of $\yy$ and $\bb_\yy\yy$, and $\ww_l(\bb,\yy):=\rho_r(\bar\ww_l(\bb,\yy))$. Dually for an RVP $(\bb,\yy)$, let $\bar\ww_r(\bb,\yy)$ be the longest common left substring (possibly of length $0$) of $\yy$ and $\yy\bb_\yy$, and $\ww_r(\bb,\yy):=\rho_l(\bar\ww_r(\bb,\yy))$.

Say that an LVP $(\bb,\yy)$ is a \emph{left tight valid pair} (\emph{LTVP}, for short) if $|\bar\ww_l(\bb,\yy)|=0$; otherwise say it is a \emph{left loose valid pair} (\emph{LLVP}, for short). Dually we define RTVP and RLVP.
\begin{rmk}\label{shortduebandsubstring}
Suppose $(\bb,\yy)$ is an LLVP with $\delta(\bar\ww_l(\bb,\yy))=0$ then $\rho_r(\yy)=\ww_l(\bb,\yy)=\rho_r(\bb_\yy\yy)$, and hence $(\bb,\yy)$ is a short LVP by \cite[Proposition~5.3]{SK}.
\end{rmk}

\begin{proposition}\label{obviousshort}
Suppose $(\bb,\yy)$ is an LVP such that $\bb\in\B(\yy)$. Then $(\bb,\yy)$ is a short LVP.
\end{proposition}
\begin{proof}
Since $\bb\in\B(\yy)$, we have $\yy=\bb_\yy\yy'$ for some string $\yy'$ by Proposition \ref{validpairbandcopy}. Hence $\bb_\yy$ is a left substring of $\bar\ww_l(\bb,\yy)$. The conclusion follows by Remark \ref{shortduebandsubstring}.
\end{proof}  

For a long LVP $(\bb,\yy)$ the above proposition gives that $N(\bb,\yy)=0$.

\begin{definition}
For an LVP $(\bb,\yy)$ if $\ww_l(\bb,\yy)$ is a proper right substring of $\rho_r(\bb_\yy)$ (resp. of $\rho_r(\yy)$) then say that $(\bb,\yy)$ is $b$-long (resp. $s$-long). If $(\bb,\yy)$ is an $s$-long pair then we say that it is $s^\partial$-long where $\partial=\delta(\rho_r(\yy))$. Similarly if $(\bb,\yy)$ is an $b$-long pair then we say that it is $b^\partial$-long where $\partial=\delta(\rho_r(\bb_\yy))$. 
\end{definition}
\begin{proposition}\label{longcriterion}
Suppose $(\bb,\yy)$ is an LVP.
\begin{itemize}
\item The pair $(\bb,\yy)$ is $s^i$-long if and only if $\xx\yy\mapsto\xx\bb_\yy\yy:H^i_l(\yy)\to H^i_l(\bb_\yy\yy)$ is injective but not surjective.
\item The pair $(\bb,\yy)$ is $b^i$-long if and only if $\xx\bb_\yy\yy\mapsto\xx\yy:H^i_l(\bb_\yy\yy)\to H^i_l(\yy)$ is injective but not surjective.
\end{itemize}
\end{proposition}
\begin{proof}
We only prove the first statement; the proof of the second is similar. Recall that $i:=\delta(\rho_r(\yy))$.

Suppose $\xx\yy$ is a string. Since $(\bb,\yy)$ is $s$-long, $\ww_l(\bb,\yy)$ is a proper right substring of $\rho_r(\yy)$ which states that $\theta(\gamma')=-i$ where $\gamma'\in\bb_\yy$ is the syllable such that $\ww_l(\bb,\yy)\gamma'$ is a string. Since $\xx\ww_l(\bb,\yy)$ is a string and $\delta(\ww_l(\bb,\yy)\gamma')=0$, $\xx\bb_\yy\yy$ is also a string. This shows that $\xx\yy\mapsto\xx\bb_\yy\yy$ is indeed a well-defined map from $H^i_l(\yy)\to H^i_l(\bb_\yy\yy)$. This map is clearly injective.

Let $\xx$ be the minimal string such that $\xx\rho_r(\yy)\in\rho\cup\rho^{-1}$. Then clearly the word $\xx\yy$ is not a string but the argument as in the above paragraph ensures that $\xx\bb_\yy\yy$ is a string. Therefore the map is not surjective.

For the converse, non-surjectivity of the map guarantees the existence of a string $\xx$ of minimal positive length with $\theta(\xx)=i$ such that $\xx\bb_\yy\yy$ is a string whereas the word $\xx\yy$ is not. Then there is some right substring $\yy'$ of $\yy$ such that $\xx\yy'\in\rho\cup\rho^{-1}$. The definition of $\rho_r(\yy)$ ensures that $|\yy'|\leq|\rho_r(\yy)|$. Since $\xx\bb_\yy\yy'$ is a string we see that $|\ww_l(\bb,\yy)|<|\yy'|$ thus completing the proof that the pair $(\bb,\yy)$ is $s$-long.
\end{proof}

\begin{example}\label{doublelongpair}
Consider the algebra $\Gamma$ from Figure \ref{doublelong}. Here the bands are $\bb:=dFCE$ and $\bb^{-1}$. For $\yy:=a$ the LVP $(\bb,\yy)$ is $s$-long as well as $b$-long.

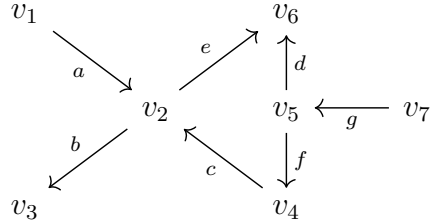
\begin{figure}[h]
\centering
  \begin{tikzcd}
v_1 \arrow[rd, "a"'] &                                      & v_6                                &                    \\
                     & v_2 \arrow[ld, "b"'] \arrow[ru, "e"] & v_5 \arrow[d, "f"] \arrow[u, "d"'] & v_7 \arrow[l, "g"] \\
v_3                  &                                      & v_4 \arrow[lu, "c"]                &                   
\end{tikzcd}
    \caption{$\Gamma$ with $\rho=\{ba,ea,bc,dg,ecfg\}$}
    \label{doublelong}
\end{figure}
\end{example}



Say that a long LVP $(\bb,\yy)$ is \emph{double long} if it is both $s$-long or $b$-long; otherwise say that it is \emph{single long}.

\section{Decorated trees}\label{secTree-basics}
Fix $\xx_0\in\STR\Lambda$ and $i\in\{1,-1\}$. We wish to associate a rooted tree, $\Tf_i(\xx_0)$, with some additional structure described in the definition below to the triple $(\Lambda,i,\xx_0)$. 
\begin{definition}
A \emph{decorated tree} is a tuple $$\TT:=(V,A,\star,V_1,V_{-1},<_1,<_{-1},T,NT,\iota\in\{1,-1\}),$$ where
\begin{itemize}
    \item $(V=\{\star\}\sqcup V_1\sqcup V_{-1},A)$ is a finite height rooted tree with root $\star$;
    \item all children of $\star$ lie in $V_{-\iota}$;
    \item $T\sqcup NT$ is the set of all leaves;
    \item for each $j\in\{1,-1\}$, the relation $<_j\subseteq V_j\times V_j$ is a strict partial order where $\uu<_j\uu'$ only if $\uu$ and $\uu'$ have the same parent;
    \item for each $\uu\in V$ and for each $j\in\{1,-1\}$, the relation $<_j$ restricted to the children of $\uu$ in $V_j$ is a strict linear order.
\end{itemize}
\end{definition}
Here $T$ and $NT$ stands for `torsion' and `non-torsion' respectively. Recall that a string $\yy$ is a \emph{(left) torsion string} if $\alpha\yy$ is not a string for any syllable $\alpha$. 

Consider the set $\A^f_i(\xx_0)$ of all arrows $\uu$ appearing in a path in $\HQ_i(\xx_0)$ starting from $\xx_0$. Recall from \cite[Remark~6.2]{SK} that $\A^f_i(\xx_0)$ is a finite set.

\begin{algo}\label{decoratree}
Construct a rooted tree $\Tf_i(\xx_0)$ as follows.
\begin{itemize}
    \item Add the root vertex labelled $\xx_0$. Add one child of the root for each half $i$-arch bridge or a zero $i$-arch bridge from $\xx_0$.
    \item For each newly added vertex $\uu$, if $\tbq(\uu)$ is a band then add one child of $\uu$ for each element $\uu'\in\A^f_i(\xx_0)$ satisfying $\sbq(\uu')=\tbq(\uu)$.
    \item Repeat the above step until all elements of $\A^f_i(\xx_0)$ are exhausted.
\end{itemize}
\end{algo}

Since $\Lambda$ is domestic $\Tf_i(\xx_0)$ is finite. Let $\VV^f_i(\xx_0)$ denote the set of vertices of the tree $\Tf_i(\xx_0)$. Let $\UU^f_i(\xx_0),\HH^f_i(\xx_0),\ZZ^f_i(\xx_0),\RR^f_i(\xx_0)$ denote those elements of $\VV^f_i(\xx_0)$ that are arch-bridges, half i-arch bridges, torsion zero i-arch bridges and torsion reverse arch bridges respectively. These sets together with the root give a partition of $\VV^f_i(\xx_0)$.

Let $\xi^f_i(\xx_0)$ denote the set of all children of $\xx_0$. For each non-root $\uu\in\VV^f_i(\xx_0)$ let $\xi^f(\uu)$ denote the set of children of $\uu$. Partition $\xi^f(\uu)$ as $\xi^f_1(\uu)\sqcup\xi^f_{-1}(\uu)$ by defining $\xi^f_j(\uu):=\{\uu'\in\xi^f(\uu)\mid\theta(\beta(\uu'))=j\}$. For a non-root vertex $\uu$ let $\pi(\uu)$ denote the parent of $\uu$. The \emph{height} of $\uu\in\VV^f_i(\xx_0)$ is the non-negative integer $h(\uu)$ such that $\pi^{h(\uu)}(\uu)=\xx_0$.

Define a function $\phi:\VV^f_i(\xx_0)\to\{1,0,-1\}$ by $\phi(\uu):=\begin{cases}0&\mbox{if }h(\uu)=0;\\-i&\mbox{if }h(\uu)=1;\\\theta(\beta(\uu))&\mbox{if }h(\uu)>1.\end{cases}$

The tree $\Tf_i(\xx_0)$ ``spreads out'' the set $\A^f_i(\xx_0)$ by placing an ``observer'' at $\xx_0$.

For $\uu\in\VV^f_i(\xx_0)$ let $\Pp(\uu):=(\xx_0,\pi^{h(\uu)-1}(\uu),\pi^{h(\uu)-2}(\uu),\hdots,\pi(\uu),\uu)$, and $$\ff(\xx_0;\uu):=\begin{cases}\xx_0&\mbox{if }\uu=\xx_0;\\\hh(\xx_0;\Pp(\uu))&\mbox{if }\uu\in\ZZ^f_i(\xx_0)\cup\RR^f_i(\xx_0);\\^\infty\tbq(\uu)\hh(\xx_0;\Pp(\uu))&\mbox{if }\uu\in\HH^f_i(\xx_0)\cup\UU^f_i(\xx_0).\end{cases}$$

For distinct non-root vertices $\uu,\uu'\in\Vf_i(\xx_0)$, let $\ff(\xx_0;\uu,\uu')$ denote the maximal common left substring of $\ff(\xx_0;\uu)$ and $\ff(\xx_0;\uu')$. Note that $\ff(\xx_0;\uu,\uu')=\ff(\xx_0;\uu',\uu)$. Let $\ff(\xx_0;\uu\mid\uu')$ be the string satisfying $$\ff(\xx_0;\uu)=\ff(\xx_0;\uu\mid\uu')\ff(\xx_0;\uu,\uu').$$

If $\tbq(\uu)$ is a band, let $\cc(\xx_0;\uu)$ denote the longest left substring of $\ff(\xx_0;\uu)$ (or, equivalently of $\hh(\xx_0;\Pp(\uu))$) in $H_l^i(\xx_0)$ whose last syllable is not a syllable of $\tbq(\uu)$; otherwise set $\cc(\xx_0;\uu):=\ff(\xx_0;\uu)$.

If $\tbq(\uu)$ is a band, let $\bb^\falpha(\uu)$ denote the cyclic permutation of $\tbq(\uu)$ for which $\bb^\falpha(\uu)\cc(\xx_0;\uu)$ is a string. Let $\fgamma(\uu)$ denote the first syllable of $\bb^\falpha(\uu)$.

\begin{proposition}\label{Hcontainparentc}
Suppose $\uu\in\UU^f_i(\xx_0)\cup\RR^f_i(\xx_0)$. Then $\cc(\xx_0;\pi(\uu))$ is a substring of $\hh(\xx_0;\Pp(\uu))$.
\end{proposition}
\begin{proof}
By definition of $\hh(\xx_0;\text{-})$ (\cite[Definition~6.3]{SK}) we get $\hh(\xx_0;\Pp(\uu))=\uu\ch\hh(\xx_0;\Pp(\pi(\uu)))$. By definition, $\cc(\xx_0;\pi(\uu))$ is a left substring of $\hh(\xx_0;\Pp(\pi(\uu)))$.

Consider the following cases:
\begin{itemize}
\item If $N(\xx_0;\sbq(\uu),\hh(\xx_0;\Pp(\uu)))=1$ then $\hh(\xx_0;\Pp(\pi(\uu)))$ is a left substring of $\hh(\xx_0;\Pp(\uu))$.
\item If $N(\xx_0;\sbq(\uu),\hh(\xx_0;\Pp(\uu)))=0$ then $\hh(\xx_0;\Pp(\uu))= \HRed{\sbq(\uu)}(\sk{\uu\sbq(\uu)\hh(\xx_0;\Pp(\pi(\uu)))})$. There are two subcases.
\begin{itemize}
    \item If $\sk{\uu\sbq(\uu)\hh(\xx_0;\Pp(\pi(\uu)))}=\uu\sbq(\uu)\hh(\xx_0;\Pp(\pi(\uu)))$ then $\hh(\xx_0;\Pp(\uu))=\uu\hh(\xx_0;\Pp(\pi(\uu)))$.
    \item If $\sk{\uu\sbq(\uu)\hh(\xx_0;\Pp(\pi(\uu)))}=\uu\hh(\xx_0;\Pp(\pi(\uu)))=\ww\bb^\falpha(\pi(\uu))\cc(\xx_0;\pi(\uu))$ for some string $\ww$ then $\hh(\xx_0;\Pp(\uu))=\ww\cc(\xx_0;\pi(\uu))$.
\end{itemize}
\end{itemize}
Clearly in every cases $\cc(\xx_0;\pi(\uu))$ is a left substring of $\hh(\xx_0;\Pp(\uu))$.
\end{proof}

\begin{corollary}\label{Comparableparentchildc}
Suppose $\uu\in\UU^f_i(\xx_0)\cup\Rf_i(\xx_0)$. Then the strings $\cc(\xx_0;\uu)$ and $\cc(\xx_0;\pi(\uu))$ are comparable since both are left substrings of $\hh(\xx_0;\Pp(\uu))$.
\end{corollary}

In view of the above corollary, we define a function $\upsilon:\Vf_i(\xx_0)\setminus\{\xx_0\}\to\{1,0,-1\}$ as $$\upsilon(\uu):=\begin{cases}1&\mbox{if }\cc(\xx_0;\pi(\uu))\mbox{ is a proper left substring of }\cc(\xx_0;\uu);\\-1&\mbox{if }\cc(\xx_0;\uu)\mbox{ is a proper left substring of }\cc(\xx_0;\pi(\uu));\\0&\mbox{if }\cc(\xx_0;\uu)=\cc(\xx_0;\pi(\uu)).\end{cases}$$

\begin{examples}\label{Ccomparison}
For $\uu\in\UU^f_i(\xx_0)\cup\Rf_i(\xx_0)$, the strings $\cc(\xx_0;\uu)$ and $\cc(\xx_0;\pi(\uu))$ are comparable in all possible ways as the following examples demonstrate.
\begin{itemize}
\item[($\upsilon=-1$)] Choosing $\xx_0:=1_{(v_1,i)}$ for the algebra $\Lambda''$ from \cite[Figure~4]{SK} the extended arch bridge quiver is shown in \cite[Figure~12]{SK}. If $\uu=dF$ and $\pi(\uu)=edcbA$ then $\cc(\xx_0;\uu)=A$ is a proper left substring of $\cc(\xx_0;\pi(\uu))=cbA$.

\item[($\upsilon=1$)] In the same example, if $\uu=IbhG$ and $\pi(\uu)=dF$ then $\cc(\xx_0;\pi(\uu))=A$ is a proper left substring of $\cc(\xx_0;\uu)=IbA$.

\item[($\upsilon=0$)] Choosing $\xx_0:=1_{(v_1,i)}$ for the algebra $\Lambda^{(iv)}$ in \cite[Figure~8]{SK}, if $\uu=dC$ and $\pi(\uu)=bA$ then $\cc(\xx_0;\uu)=A=\cc(\xx_0;\pi(\uu))$.
\end{itemize}
\end{examples}

\begin{proposition}\label{Cequivnormality}
Suppose $\uu\in\Vf_i(\xx_0)$ and $h(\uu)\geq1$. Then $\uu$ is normal if and only if $\ff(\xx_0;\uu,\pi(\uu))$ is a proper left substring of $\cc(\xx_0;\uu)$.
\end{proposition}
\begin{proof}
Recall that both $\ff(\xx_0;\uu,\pi(\uu))$ and $\cc(\xx_0;\uu)$ are substrings of $\ff(\xx_0;\uu)$ and hence they are comparable.

First assume that $h(\uu)\geq2$; a similar argument works when $h(\uu)=1$.

If $\uu$ is normal then $\cc(\xx_0;\uu)=\uu^o\ff(\xx_0;\uu,\pi(\uu))$. Since $|\uu^o|>0$, we get the required conclusion.

Conversely if $\ff(\xx_0;\uu,\pi(\uu))$ is a proper left substring  $\cc(\xx_0;\uu)$ then $\beta(\uu)$ is a syllable of $\cc(\xx_0;\uu)$, and hence not a syllable of $\tbq(\uu)$. Therefore $\uu$ is normal.
\end{proof}

\begin{proposition}\label{sbandappearinccontainH}
Suppose $\uu\in\Uf_i(\xx_0)\cup\Rf_i(\xx_0)$ and  $N(\xx_0;\sbq(\uu),\cc(\xx_0;\uu))=1$. Then $\hh(\xx_0;\Pp(\pi(\uu)))$ is a substring of $\cc(\xx_0;\uu)$.
\end{proposition}
\begin{proof}
Since $\cc(\xx_0;\uu)$ is a left substring of $\hh(\xx_0;\Pp(\uu))$ and $N(\xx_0;\sbq(\uu),\cc(\xx_0;\uu))=1$ we get $N(\xx_0;\sbq(\uu),\hh(\xx_0;\Pp(\uu)))=1$. Hence using $\hh(\xx_0;\Pp(\uu))=\uu\ch\hh(\xx_0;\Pp(\pi(\uu)))$, we get that $\hh(\xx_0;\Pp(\uu))=\sk{\uu\sbq(\uu)\hh(\xx_0;\Pp(\pi(\uu)))}$. Then $\hh(\xx_0;\Pp(\pi(\uu)))$ is a left substring of $\hh(\xx_0;\Pp(\uu))$. Hence $\hh(\xx_0;\Pp(\pi(\uu)))$ and $\cc(\xx_0;\uu)$ are comparable.

Since $N(\xx_0;\sbq(\uu),\hh(\xx_0;\Pp(\pi(\uu))))=0$ but $N(\xx_0;\sbq(\uu),\cc(\xx_0;\uu))=1$, we conclude that $\hh(\xx_0;\Pp(\pi(\uu)))$ is a left substring of $\cc(\xx_0;\uu)$.
\end{proof}

For $\uu\in\Vf_i(\xx_0)$ if $\sbq(\uu)$ is a band, Proposition \ref{Hcontainparentc} allows us to define a string $\ww(\xx_0;\uu)$ by $\ff(\xx_0;\uu,\pi(\uu))=\ww(\xx_0;\uu)\cc(\xx_0;\pi(\uu))$. Let $\bb^\fbeta$ denote the cyclic permutation of $\sbq(\uu)$ for which $\bb^\fbeta\ww(\xx_0;\uu)$ is a string.

The following is an important observation from the proof of Proposition \ref{Hcontainparentc}.
\begin{corollary}\label{Hcontainextc}
Suppose $\uu\in\UU^f_i(\xx_0)\cup\Rf_i(\xx_0)$. Then $\ww(\xx_0;\uu)\cc(\xx_0;\pi(\uu))$ is a left substring of $\hh(\xx_0;\Pp(\uu))$.
\end{corollary}

Our next goal is to show that two different vertices of the tree have incomparable values of $\ff(\xx_0;\mbox{-})$. Here is a supporting result.
\begin{proposition}\label{notextfofrevarchbridge}
If $\uu\in\RR^f_i(\xx_0)$ then $\alpha\ff(\xx_0;\uu)$ is not a string for any syllable $\alpha$.
\end{proposition}

\begin{proof}
Suppose $\alpha\ff(\xx_0;\uu)$ is a string for some syllable $\alpha$. Since $\uu\in\RR^f_i(\xx_0)$, $\uu$ is a maximal reverse torsion arch bridge, and hence $\alpha\uu\sbq(\uu)$ is not a string. Hence $\bb^\fbeta\ww(\xx_0;\uu)\cc(\xx_0;\uu)\nequiv_H\ww(\xx_0;\uu)\cc(\xx_0;\uu)$ as $\alpha\uu^o\bb^\fbeta\ww(\xx_0;\uu)\cc(\xx_0;\uu)$ is not a string but $\alpha\uu^o\ww(\xx_0;\uu)\cc(\xx_0;\uu)$.

On the other hand since $\alpha\uu\sbq(\uu)$ is not a string, we get $N(\xx_0;\sbq(\uu),\ff(\xx_0;\uu))=0$. Thus by the definition of $\hh(\xx_0;\Pp)$, $\HRed{\sbq(\uu)}(\uu^o\bb^\fbeta\ww(\xx_0;\uu)\cc(\xx_0;\uu))$ exists, a contradiction to the above paragraph. Hence our assumption is wrong.
\end{proof}

\begin{proposition}\label{fisinjective}
Suppose $\uu,\uu'\in\Vf_i(\xx_0)$ are non-root vertices. If $\ff(\xx_0;\uu')$ is a left substring of $\ff(\xx_0;\uu)$ then $\uu=\uu'$.
\end{proposition}

\begin{proof}
Suppose $\ff(\xx_0;\uu')$ is a left substring of $\ff(\xx_0;\uu)$.

If the former is an $\N^*$-string then so is latter and hence they are equal.

On the other hand, if $\ff(\xx_0;\uu')$ is a finite string then $\uu'\in\RR^f_i(\xx_0)\cup\ZZ^f_i(\xx_0)$.

If $\uu'\in\ZZ^f_i(\xx_0)$ then $\ff(\xx_0;\uu')=\uu'\xx_0$. By definition, $\uu'$ is a maximal weak torsion zero bridge, and hence $\ff(\xx_0;\uu')$ cannot be extended to the left. Thus $\ff(\xx_0;\uu')=\ff(\xx_0;\uu)$.

If $\uu'\in\RR^f_i(\xx_0)$ then Proposition \ref{notextfofrevarchbridge} guarantees the same conclusion.

In each case we have that $\ff(\xx_0;\uu)=\ff(\xx_0;\uu')$ and that both $\uu$ and $\uu'$ are either equal to $\xx_0$ or in $\ZZ^f_i(\xx_0)\cup\RR^f_i(\xx_0)$ or in $\HH^f_i(\xx_0)\cup\UU^f_i(\xx_0)$. In each case it is clear that $\hh(\xx_0;\Pp(\uu))=\hh(\xx_0;\Pp(\uu'))$. Now \cite[Theorem~7.5]{SK} gives that $\Pp(\uu)=\Pp(\uu')$, and hence $\uu=\uu'$ in $\Tf_i(\xx_0)$. 
\end{proof}

\begin{corollary}
If $\uu,\uu'\in\Vf_i(\xx_0)$ are distinct non-root vertices then $|\ff(\xx_0;\uu\mid\uu')|>0$ and hence $\theta(\ff(\xx_0;\uu'\mid\uu))=-\theta(\ff(\xx_0;\uu\mid\uu'))$.
\end{corollary}

Now we define total ordering(s) on the set of children of a vertex of $\Tf_i(\xx_0)$. Let $\uu_1,\uu_2\in\xif_i(\xx_0)$. Say $\uu_1\sqsf_i\uu_2$ if $\theta(\ff(\xx_0;\uu_2\mid\uu_1))=i$. Let $\fmin_i(\xx_0)$ and $\fmax_i(\xx_0)$ denote the minimal and maximal elements of this order respectively. The immediate successor and immediate predecessor of $\uu_1$ in $(\xif_i(\xx_0),\sqsf_i)$, if exist, are denoted by $\uu_1^{f+}$ and $\uu_1^{f-}$ respectively.

Let $\uu\in\Vf_i(\xx_0)\setminus\{\xx_0\}$ and $\uu_1,\uu_2\in\xif_j(\uu)$. Say $\uu_1\sqsf_j\uu_2$ if $\theta(\ff(\xx_0;\uu_2\mid\uu_1))=-j$. Let $\fmin_j(\uu)$ and $\fmax_j(\uu)$ denote the minimal and maximal elements of this order respectively. The immediate successor and immediate predecessor of $\uu_1\in\xif_j(\uu)$ in $(\xif_j(\uu),\sqsf_j)$, if exist, are denoted by $\uu_1^{f+}$ and $\uu_1^{f-}$ respectively. Choosing leaves in $\Zf_i(\xx_0)\cup\Rf_i(\xx_0)$ as torsion while $\Hf_i(\xx_0)\cup\Uf_i(\xx_0)$ as non-torsion we have completely described the decorated tree structure of $\Tf_i(\xx_0)$.

We shall use words like parent, grandparent, child, uncle, granduncle etc. to describe relations between two vertices of $\Tf_i(\xx_0)$.

\begin{example}
Choosing $\xx_0:=1_{(v_1,i)}$ for the algebra $\Lambda''$ from \cite[Figure~4]{SK} the extended arch bridge quiver is shown in \cite[Figure~12]{SK}. Figure \ref{TreeforLambda''} shows the decorated tree $\Tf_1(\xx_0)$ where the arrow with target $\uu$ in the tree is labeled $j$ if $\uu\in\xif_j(\pi(\uu))$. Here $edcbA,IbA\in\xif_1(\xx_0)$ satisfy $edcbA\sqsf_1 IbA$.

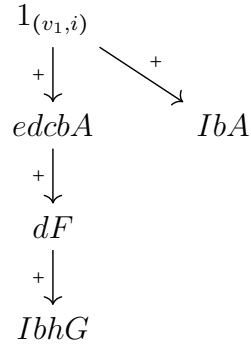
\begin{figure}[h]
\centering
 \begin{tikzcd}
{{1_{(v_1,i)}}} \arrow[d, "+"'] \arrow[rd, "+"] &     \\
edcbA \arrow[d, "+"']                           & IbA \\
dF \arrow[d, "+"']                              &     \\
IbhG                                            &    
\end{tikzcd}
    \caption{$\Tf_1(1_{(v_1,i)})$ for $\Lambda''$}
    \label{TreeforLambda''}
\end{figure}

\end{example}


\section{Long elements of $\Vf_i(\xx_0)$}\label{longelements}
Say that $\uu\in\UU^f_i(\xx_0)\cup\Rf_i(\xx_0)$ is \emph{long} if $N(\xx_0;\sbq(\uu),\ff(\xx_0;\uu))=1$; otherwise say it is \emph{short}. If we refer to $\uu$ as long or short then it is implicitly assumed that $\uu\in\UU^f_i(\xx_0)\cup\RR^f_i(\xx_0)$.

If $\uu\in\Vf_i(\xx_0)$ is long then let $\tilde\ww(\xx_0;\uu)$ be defined by $\ww(\xx_0;\uu)=\bb^\fbeta(\uu)\tilde\ww(\xx_0;\uu)$. On the other hand if $\uu$ is short then we set $\tilde\ww(\xx_0;\uu):=\ww(\xx_0;\uu)$. Note that $|\tilde\ww(\xx_0;\uu)|<|\sbq(\uu)|$.

\begin{rmk}\label{longshortchildflocation}
Suppose $\uu\in\Uf_i(\xx_0)\cup\Hf_i(\xx_0)$ and $\uu_1,\uu_2\in\xif(\uu)$. If $\uu_1$ is long and $\uu_2$ is short then $\ff(\xx_0;\uu_1,\uu_2)=\ww(\xx_0;\uu_2)\cc(\xx_0;\uu)$. Moreover if $\theta(\beta(\uu_1))=\theta(\beta(\uu_2))=\partial$ then $\uu_2\sqsf_\partial\uu_1$.
\end{rmk}

\begin{rmk}\label{fcomparisonsiblings}
Suppose $\uu\in\Vf_i(\xx_0)\setminus\xx_0$. If $\uu_1,\uu_2\in\xif_\partial(\uu)$ with $\uu_1\sqsf_\partial\uu_2$ then $\ff(\xx_0;\uu,\uu_1)$ is a left substring of $\ff(\xx_0;\uu,\uu_2)$.  
\end{rmk}

\begin{rmk}\label{fincreasingsandwich}
Suppose $\uu\in\Uf_i(\xx_0)\cup\Hf_i(\xx_0)$ and $\uu_1,\uu_2,\uu_3\in\xif_\partial(\uu)$. If $\uu_1\sqsf_\partial\uu_2\sqsf_\partial\uu_3$ then $\ff(\xx_0;\uu_1,\uu_3)$ is a left substring of $\ff(\xx_0;\uu_2,\uu_3)$. 
\end{rmk}

\begin{rmk}\label{countingsbqnormal}
If $\uu\in\Uf_i(\xx_0)$ is normal then $\uu^o$ is a substring of $\cc(\xx_0;\uu)$ which is a substring of $\hh(\xx_0;\Pp(\uu))$. Hence $N(\xx_0;\sbq(\uu),\ff(\xx_0;\uu))=N(\xx_0;\sbq(\uu),\hh(\xx_0;\Pp(\uu)))=N(\xx_0;\sbq(\uu),\cc(\xx_0;\uu))$.
\end{rmk}

\begin{proposition}\label{sourcebandlongappearinH}
Suppose $\uu\in\UU^f_i(\xx_0)\cup\Rf_i(\xx_0)$. Then $\uu$ is long if and only if $N(\xx_0;\sbq(\uu),\hh(\xx_0;\Pp(\uu)))=1$.
\end{proposition}
\begin{proof}
Since $\hh(\xx_0;\Pp(\uu))$ is a left substring of $\ff(\xx_0;\uu)$, we only need to prove the forward implication. If $\uu\in\Rf_i(\xx_0)$ then $\ff(\xx_0;\uu)=\hh(\xx_0;\Pp(\uu))$, and hence the conclusion is obvious. Thus we assume that $\uu\in\Uf_i(\xx_0)$. Let $\gamma$ denote the first syllable of $\tbq(\uu)$ respectively. 

Suppose $\uu$ is long. Then $N(\xx_0;\sbq(\uu),\ff(\xx_0;\uu))=1$. For contradiction, assume that $N(\xx_0;\sbq(\uu),\hh(\xx_0;\Pp(\uu)))=0$. Then there is a shortest string $\ww$ of positive length such that $N(\xx_0;\sbq(\uu),\ww\hh(\xx_0;\Pp(\uu)))=1$ and $\ww\hh(\xx_0;\Pp(\uu))$ is a left substring of $\ff(\xx_0;\uu)$. Hence $\ww$ is a left substring of a finite power of $\tbq(\uu)$. In view of Remark \ref{countingsbqnormal} $\uu$ is abnormal. Minimality of $\ww$ guarantees that a right substring of $\ww$ of positive length is also a substring of a cyclic permutation of $\sbq(\uu)$. Further \cite[Propositions~4.2,4.3]{SK} gives that $\gamma$ is not a syllable of $\sbq(\uu)$. Thus $\ww$ is not a substring of $\sbq(\uu)$. Therefore $N(\xx_0;\sbq(\uu),\ff(\xx_0;\uu))=1$ and $N(\xx_0;\sbq(\uu),\hh(\xx_0;\Pp(\uu)))=0$ together give that a cyclic permutation of $\sbq(\uu)$ is a substring of $\ww$. Since $\gamma$ is not a syllable of $\sbq(\uu)$, minimality of $\ww$ ensures that $\ww$ is a substring of $\tbq(\uu)$. Hence a cyclic permutation of $\sbq(\uu)$ is a substring of $\tbq(\uu)$, a contradiction to the domesticity of the algebra by \cite[Corollary~3.4.1]{GKS}.
\end{proof}

Proposition \ref{sourcebandlongappearinH} together with Remark \ref{countingsbqnormal} ensures that if $\uu$ is normal with $h(\uu)>1$ then $\uu$ is long if and only if $N(\xx_0;\sbq(\uu),\cc(\xx_0;\uu))=1$. We will show in the proof of Proposition \ref{bandincludedinblong} that the same conclusion of holds when $\uu$ is $b$-long. However the following example shows that it fails if $\uu$ is abnormal and $s$-long. 

\begin{example}\label{C-equalopplongEx} Consider the algebra $\Gamma^{(i)}$ in Figure \ref{C-equalopplong}. The only bands here are $\bb_1:= hbdC$, $\bb_2:= dfE$ and their inverses. Choose $\xx_0:=1_{(\vv_1,i)}$ and $\uu:=dC$. Then $\pi(\uu)=hbA$. Thus $\cc(\xx_0;\pi(\uu))=A$ and $\cc(\xx_0;\uu)=ChbA$, and hence $\upsilon(\uu)=1$. Here $\uu$ is $s$-long, $N(\xx_0;\sbq(\uu),\ff(\xx_0;\uu))=1$ but $N(\xx_0;\sbq(\uu),\cc(\xx_0;\uu))=0$. 

\begin{figure}[h]
\centering
\begin{tikzcd}
v_7 \arrow[r, "g"'] & v_5 \arrow[r, "h"']                 & v_6                                \\
v_1                 & v_2 \arrow[u, "b"] \arrow[l, "a"]   & v_3 \arrow[u, "c"'] \arrow[l, "d"] \\
                    & v_4 \arrow[u, "e"] \arrow[ru, "f"'] &   
\end{tikzcd}
    \caption{$\Gamma^{(i)}$ with $\rho=\{ae,ad,be,hg,cf,hbdf\}$}
    \label{C-equalopplong}
\end{figure}
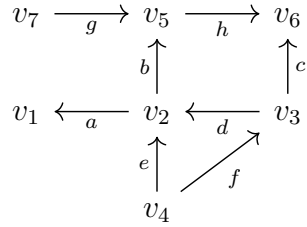
\end{example}



\begin{rmk}
It is interesting to note that if we remove $ae$ from $\rho$ in Figure \ref{C-equalopplong} then in that algebra we get $\upsilon(\uu)=0$ where $\uu$ is as chosen in Example \ref{C-equalopplongEx}.
\end{rmk}

For an arbitrary $\uu$ the best we can achieve is the following.
\begin{corollary}\label{longequivwithextC}
Suppose $\uu\in\UU^f_i(\xx_0)\cup\RR^f_i(\xx_0)$ and $\uu\in\xif_\partial(\pi(\uu))$ for some $\partial\in\{+,-\}$. Then $\uu$ is long if and only if the LVP $(\bb,\yy):=(\sbq(\uu),\tilde\ww(\xx_0;\uu)\cc(\xx_0;\pi(\uu)))$ is long.
\end{corollary}
\begin{proof}
If $\uu$ is long then $N(\xx_0;\sbq(\uu),\hh(\xx_0;\Pp(\uu)))=1$ by Proposition \ref{sourcebandlongappearinH}. Since $\hh(\xx_0;\Pp(\uu))$ is an H-reduced string by \cite[Theorem~6.4]{SK}, $(\sbq(\uu),\tilde\ww(\xx_0;\uu)\cc(\xx_0;\pi(\uu)))$ is long by Corollary \ref{Hcontainextc}.

Conversely if $(\sbq(\uu),\tilde\ww(\xx_0;\uu)\cc(\xx_0;\pi(\uu)))$ is long then $N(\xx_0;\sbq(\uu),\ww(\xx_0;\uu)\cc(\xx_0;\pi(\uu)))=1$ and hence $N(\xx_0;\sbq(\uu),\hh(\xx_0;\Pp(\uu)))=1$ by Corollary \ref{Hcontainextc}. Therefore $\uu$ is long by Proposition \ref{sourcebandlongappearinH}. 
\end{proof}

In view of the above corollary, we freely transfer adjectives and qualifiers of the pair $(\sbq(\uu),\tilde\ww(\xx_0;\uu)\cc(\xx_0;\pi(\uu)))$ to $\uu$.

We note an immediate consequence of the above result.
\begin{corollary}\label{longchildlocation}
Suppose $\uu\in\Hf_i(\xx_0)\cup\Uf_i(\xx_0)$, $\uu_1,\uu_2\in\xif(\uu)$ and $\tilde\ww(\xx_0;\uu_1)$ is a left substring of $\tilde\ww(\xx_0;\uu_2)$. If $\uu_2$ is long then $\uu_1$ is also long.
\end{corollary}

Suppose $\uu\in\Hf_i(\xx_0)\cup\Uf_i(\xx_0)$ and $(\tbq(\uu),\cc(\xx_0;\uu))$ is long. If the pair is $s$-long then let $\bar\gamma^s(\uu)$ denote the last syllable of $\cc(\xx_0;\uu)$, and let $\bar{\bar\gamma}^s(\uu)$ denote the last syllable of $\bb^\falpha(\uu)$. If the pair is $b$-long then let $\bar\gamma^b(\uu)$ denote the last syllable of $\bb^\falpha(\uu)$, and let $\bar{\bar\gamma}^b(\uu)$ denote the last syllable of $\cc(\xx_0;\uu)$ if $|\cc(\xx_0;\uu)|>0$.

\begin{rmk}
If $\uu\in\Hf_i(\xx_0)$ and $|\cc(\xx_0;\uu)|=0$ then $(\tbq(\uu),\cc(\xx_0;\uu))$ is not $s$-long.
\end{rmk}

Suppose $\uu\in\Uf_i(\xx_0)\cup\Rf_i(\xx_0)$. If $\uu$ is $s$-long then $\rr^s_1(\pi(\uu))$ denotes the longest left substring of $\bb^\falpha(\pi(\uu))$ satisfying $\delta(\rr^s_1(\pi(\uu)))=\theta(\bar\gamma^s(\pi(\uu)))$ if $|\rr^s_1(\pi(\uu))|>0$. Let $\tilde\rr^s(\pi(\uu))$ and $\rr^s_2(\pi(\uu))$ satisfy $\rr^s_1(\pi(\uu))\tilde\rr^s(\pi(\uu))=\rho_r(\rr^s_1(\pi(\uu))\cc(\xx_0;\pi(\uu)))$ and $\rr^s_2(\pi(\uu))\rr^s_1(\pi(\uu))\tilde\rr^s(\pi(\uu))\in\rho\cup\rho^{-1}$. Define a syllable $\tilde\beta^s(\pi(\uu))$ and a string $\rr^s(\pi(\uu))$ by $\rr^s_2(\pi(\uu))\rr^s_1(\pi(\uu))=\tilde\beta^s(\pi(\uu))\rr^s(\pi(\uu))$.

On the other hand, if $\uu$ is $b$-long then $\rr^b_1(\pi(\uu))$ denotes the longest left substring of $\bb^\falpha(\pi(\uu))$ satisfying $\delta(\rr^b_1(\pi(\uu)))=\theta(\bar\gamma^b(\pi(\uu)))$ if $|\rr^b_1(\pi(\uu))|>0$. Let $\tilde\rr^b(\pi(\uu))$ and $\rr^b_2(\pi(\uu))$ satisfy $\rr^b_1(\pi(\uu))\tilde\rr^b(\pi(\uu))=\rho_r(\rr^b_1(\pi(\uu))\bb^\falpha(\pi(\uu)))$ and $\rr^b_2(\pi(\uu))\rr^b_1(\pi(\uu))\tilde\rr^b(\pi(\uu))\in\rho\cup\rho^{-1}$. Define a syllable $\tilde\beta^b(\pi(\uu))$ and a string $\rr^b(\pi(\uu))$ by $\rr^b_2(\pi(\uu))\rr^b_1(\pi(\uu))=\tilde\beta^b(\pi(\uu))\rr^b(\pi(\uu))$.

\begin{rmk}\label{r1doublelong}
Suppose $\uu\in\Rf_i(\xx_0)\cup\Uf_i(\xx_0)$. If $(\sbq(\uu),\tilde\ww(\xx_0;\uu)\cc(\xx_0;\pi(\uu)))$ is double long then $\fgamma(\pi(\uu))$ is the first syllable of exactly one of $\rr^s_1(\pi(\uu))$ and $\rr^b_1(\pi(\uu))$, and hence $|\rr^s_1(\pi(\uu))||\rr^b_1(\pi(\uu))|=0$ but $|\rr^s_1(\pi(\uu))|+|\rr^b_1(\pi(\uu))|>0$.
\end{rmk}

In view of the above remark, if $(\sbq(\uu),\tilde\ww(\xx_0;\uu)\cc(\xx_0;\pi(\uu)))$ is a double long pair and $|\rr^s_1(\pi(\uu))|>0$ then we say it is \emph{$(s,b)$-long}, otherwise we say it is \emph{$(b,s)$-long}. Say that a $(b,s)$-long pair is $(b,s)^\partial$-long if it is $b^\partial$-long. Similarly say that an $(s,b)$-long pair is $(s,b)^\partial$-long if it is $s^\partial$-long.

\begin{rmk}\label{WRlongcomparison}
If $\uu$ is $b$-long (resp. $s$-long) then $\tilde{\ww}(\xx_0;\uu)$ is a left substring of $\rr^b_1(\pi(\uu))$ (resp. $\rr^s_1(\pi(\uu))$). Moreover, if $\uu$ is double long then Remark \ref{r1doublelong} guarantees that $|\tilde{\ww}(\xx_0;\uu)|=0$.
\end{rmk}

\begin{rmk}
Suppose $\uu\in\Hf_i(\xx_0)\cup\Uf_i(\xx_0)$ and $\xif_\partial(\uu)\neq\emptyset$.
\begin{itemize}
\item If $(\tbq(\uu),\cc(\xx_0;\uu))$ is $b^\partial$-long then $\rr^b_1(\uu)$ is a proper left substring of $\tilde\ww(\xx_0;\fmin_\partial(\uu))$ if and only if $|\rr^b_2(\uu)|=1$. As a consequence $\rr^b_1(\uu)$ is a proper left substring of $\ww(\xx_0;\fmin_\partial(\uu))$.
\item If $(\tbq(\uu),\cc(\xx_0;\uu))$ is $s^\partial$-long and $\xif_\partial(\uu)\neq\emptyset$ then $\rr^s_1(\uu)$ is a proper left substring of $\ww(\xx_0;\fmin_\partial(\uu))$.
\end{itemize}
\end{rmk}

\begin{example}
In the algebra $\Gamma^{(ii)}$ in Figure \ref{doublelongpositive} the only bands are $\bb_1:=icBD$, $\bb_2:=cfE$ and their inverses. Choosing $\xx_0:=1_{(\vv_1,1)}$ and $\uu:=icBa$, the LVP $(\tbq(\uu),\cc(\xx_0;\uu))=(\bb_1,a)$ is $b^{1}$-long with $\rr^b_1(\uu)=B$ and $\rr^b_2(\uu)=HG$. Here $GBD,cBD\in\xif_{1}(\uu)$ with $\fmin_{1}(\uu)=GBD$ with $\tilde\ww(\xx_0;\fmin_{1}(\uu))=\rr^b_1(\uu)$. 

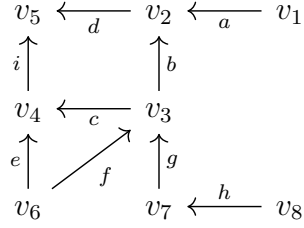
\begin{figure}[h]
\begin{tikzcd}
v_5                                 & v_2 \arrow[l, "d"]                 & v_1 \arrow[l, "a"]  \\
v_4 \arrow[u, "i"]                  & v_3 \arrow[u, "b"'] \arrow[l, "c"] &                     \\
v_6 \arrow[u, "e"] \arrow[ru, "f"'] & v_7 \arrow[u, "g"']                & v_8 \arrow[l, "h"']
\end{tikzcd}
    \caption{$\Gamma^{(ii)}$ with $\rho=\{da,bf,cg,ie,icf,dbgh\}$}
    \label{doublelongpositive}
\end{figure}

If we consider the algebra obtained from the same quiver by replacing the relation $dbgh$ by $dbg$ in $\rho$ then keeping $\xx_0$ and $\uu$ same as above we get $\rr^b_1(\uu)=B$ and $\rr^b_2(\uu)=G$. Here $cBD=\fmin_{1}(\uu)$ with $\tilde\ww(\xx_0;\fmin_{1}(\uu))=cB$. 
\end{example}

\section{Characterizing different values of $\upsilon$}\label{upsiloncharacter}
\begin{proposition}\label{CpropCharacter}
Suppose $\uu\in\Vf_i(\xx_0)$, $h(\uu)>2$ and $\pi(\uu)$ is abnormal. Then $\pi(\uu)^c\cc(\xx_0;\pi(\uu))$ is a string if and only if $\bb^\falpha(\pi(\uu))=\bua{\pi(\uu)}$ if and only if $\upsilon(\pi(\uu))=1$.
\end{proposition}

\begin{proof}
Note that the only possible syllables $\alpha$ such that $\pi(\uu)^c\alpha$ is a string are $\alpha(\pi(\uu))$ and the last syllable of $\bua{\pi(\uu)}$, say $\gamma'$. Moreover $\gamma'$ is a syllable of $\tbq(\pi(\uu))$ but not of $\sbq(\pi(\uu))$ while $\alpha(\pi(\uu))$ is a syllable of $\sbq(\pi(\uu))$ but not of $\tbq(\pi(\uu))$. Hence $\pi(\uu)^c\cc(\xx_0;\pi(\uu))$ is a string if and only if $\alpha(\pi(\uu))$ is the last syllable of $\cc(\xx_0;\pi(\uu))$. Since $\cc(\xx_0;\pi^2(\uu))$ is a proper left substring of $\hh(\xx_0;\pi(\uu))$ by Corollary \ref{Comparableparentchildc}, if the latter holds then $\upsilon(\pi(\uu))=1$.

On the other hand if $\upsilon(\pi(\uu))=1$ then the last syllable of $\cc(\xx_0;\pi(\uu))$ is a syllable of $\sbq(\pi(\uu))$ but not of $\tbq(\pi(\uu))$. Hence for any cyclic permutation $\bb'$ of $\tbq(\pi(\uu))$ if $\bb'\cc(\xx_0;\pi(\uu))$ is a string then $\bb'=\bua{\pi(\uu)}$. Hence $\pi(\uu)^c\cc(\xx_0;\pi(\uu))$ is a string.
\end{proof}

Below we note some immediate consequences.
\begin{corollary}\label{slongdueabnormality}
Suppose $\uu\in\Vf_i(\xx_0)$, $h(\uu)>2$, $\pi(\uu)$ is abnormal and $\upsilon(\pi(\uu))=1$. If $\uu$ is $s$-long then $|\tilde\ww(\xx_0;\uu)|>0$ if and only if $|\pi(\uu)^c|=0$.
\end{corollary}

\begin{corollary}\label{blongdueabnormality}
Suppose $\uu\in\Vf_i(\xx_0)$, $h(\uu)>2$, $\pi(\uu)$ is abnormal and $\upsilon(\pi(\uu))=1$. Then
$\uu$ is $b$-long if and only if $\pi(\uu)^c=\ww\tilde\ww(\xx_0;\uu)$ for some string $\ww$. If these equivalent conditions hold then 
\begin{itemize}
    \item $\rr^b_1(\pi(\uu))=\pi(\uu)^c$.
    \item $\uu$ is $b^{-\theta(\beta(\uu))}$-long.
\end{itemize} 
\end{corollary}

\begin{corollary}\label{exceptionalpositive}
Suppose $\uu\in\Vf_i(\xx_0)$, $h(\uu)>2$ and $\upsilon(\pi(\uu))=1$ with $\cc(\xx_0;\pi(\uu))=\zz\cc(\xx_0;\pi^2(\uu))$. If $\uu$ is long and $\uu_1$ is an uncle of $\uu$ with $\beta(\uu_1)=\beta(\uu)$ then $\uu_1$ could be either $b^\partial$-long or $s^\partial$-long but not both. Further if $\pi(\uu)\sqsf_\partial\uu_1$ then $|\pi(\uu)^c|>0$, $|\tilde{\ww}(\xx_0;\uu)|=0$, $\uu$ is $b^{-\partial}$ long and exactly one of the following happens:
\begin{itemize}
    \item $\uu_1$ is $b^\partial$-long and $\rr^b_1(\pi^2(\uu))=\zz$;
    \item $\uu_1$ is $s^\partial$-long and $\rr^s_1(\pi^2(\uu))=\zz$.
\end{itemize}
\end{corollary}

\begin{example}\label{unclelongpositive}
In the algebra $\Gamma^{(ii)}$ in Figure \ref{doublelongpositive} the only bands are $\bb_1:=icBD$, $\bb_2:=cfE$ and their inverses. Choosing $\xx_0:=1_{(\vv_1,1)}$ and $\uu:=HGfE$ we have $\pi(\uu)=cBD$ and $\pi^2(\uu)=icBa$ with $\upsilon(\uu)=1$. Here $\uu$ is $b^{-1}$-long and its uncle $\uu_1:=GBD$ is $b^1$-long.
\end{example}


\begin{corollary}\label{abnpropshort}
Suppose $\uu\in\Vf_i(\xx_0)$, $h(\uu)>2$ and $\upsilon(\pi(\uu))=1$. Further suppose that $\pi(\uu)$ is abnormal with $|\pi(\uu)^c|>0$. 
\begin{itemize}
\item If $|\tilde{\ww}(\xx_0;\pi(\uu))|>0$ then Proposition \ref{CpropCharacter} guarantees that $\pi(\uu)^c\alpha(\pi(\uu))$ is a right substring of $\tilde{\ww}(\xx_0;\pi(\uu))$. Since $\delta(\pi(\uu)^c\alpha(\pi(\uu)))=0$ we get $\pi(\uu)$ is short. 

\item If $|\tilde{\ww}(\xx_0;\pi(\uu))|=0$ then $\pi(\uu)$ is $s$-long for otherwise we get $\upsilon(\pi(\uu))=0$, a contradiction. But here $\ww(\xx_0;\pi(\uu))=\bb^\falpha{(\pi^2(\uu))}$.
\end{itemize}
\end{corollary}

\begin{proposition}\label{characterzerodoublelong}
Suppose $\uu\in\Vf_i(\xx_0)$, $h(\uu)>2$ and $\upsilon(\pi(\uu))=1$. If $\pi(\uu)$ is abnormal with $|\pi(\uu)^c|=0$ then $\uu$ is not double long.
\end{proposition}
\begin{proof}
Since $|\pi(\uu)^c|=0$, Corollary \ref{blongdueabnormality} ensures that $|\rr^b_1(\pi(\uu))|=0$.

If $\uu$ is $b$-long then $\beta(\uu)$ is the first syllable of $\bub{\pi(\uu)}$. In this case $\beta(\pi(\uu))\alpha(\pi(\uu))$ and $\beta(\uu)\alpha(\pi(\uu))$ are strings, and $\theta(\alpha(\pi(\uu)))=\theta(\beta(\pi(\uu)))$. Thus $\theta(\beta(\uu))=-\theta(\alpha(\pi(\uu)))$ and hence $\uu$ is not $s$-long.

If $\uu$ is $s$-long then from the above discussion it is clear that $|\tilde{\ww}(\xx_0;\uu)|>0$ and $\beta(\pi(\uu))$ is the first syllable of $\tilde{\ww}(\xx_0;\uu)$. But since $\theta(\beta(\pi(\uu)))=-\theta(\bar{\bar\gamma}(\pi(\uu)))$, $\uu$ is not $b$-long. Hence the result.
\end{proof}

\begin{proposition}\label{cincresingcriteia}
Suppose $\uu\in\Vf_i(\xx_0)$, $h(\uu)>2$, $\pi(\uu)$ is abnormal with $|\pi(\uu)^c|>0$ and $\upsilon(\pi(\uu))=1$. If $\cc(\xx_0;\pi(\uu))=\zz\cc(\xx_0;\pi^2(\uu))$ then $\pi(\uu)^c\zz$ is a left substring of $\bb^\falpha(\pi^2(\uu))$.
\end{proposition}

\begin{proof}
Since $\pi(\uu)$ is abnormal, $\zz$ is a left substring of $(\bb^\falpha(\pi^2(\uu)))^2$. If $|\pi(\uu)^c|>0$ then, in view of Remark \ref{locatHreduction}, $(\sbq(\pi(\uu)),\tilde\ww(\xx_0;\pi(\uu))\cc(\xx_0;\pi^2(\uu)))$ is short, and hence $\zz$ is a proper left substring of $\bb^\falpha(\pi^2(\uu))$. Since $|\zz|>0$, $\pi(\uu)^c\zz$ a string. Hence $\pi(\uu)^c\zz$ and $\bb^\falpha(\pi(\uu))$ are comparable strings.

If $\pi(\uu)^c\zz$ is a not left substring of $\bb^\falpha(\pi^2(\uu))$ then $\bb^\falpha(\pi^2(\uu))$ is a proper substring of $\pi(\uu)^c\zz$ and hence $\fgamma(\pi^2(\uu))$ is a syllable of $\pi(\uu)^c$. Since $\zz$ is a proper left substring of $\bb^\falpha(\pi^2(\uu))$, $\fgamma(\pi^2(\uu))\bar\gamma$ is a substring of $\pi(\uu)^c$, where $\bar\gamma$ is the last syllable of $\bb^\falpha(\pi^2(\uu))$. But $\fgamma(\pi^2(\uu))$ is the first syllable of $\zz$ and hence a common syllable of $\zz$ and $\pi(\uu)^c$, a contradiction to $|\zz|>0$. Therefore $\pi(\uu)^c\zz$ is a left substring of $\bb^\falpha(\pi^2(\uu))$.
\end{proof}


 

\begin{example}
In the algebra $\Gamma^{(i)}$ in Figure \ref{C-equalopplong}, choosing $\xx_0:=1_{(\vv_1,i)}$ and $\uu:=Gb$ we have $\pi(\uu):=dC$ and $\pi^2(\uu):=hbA$. Here $\cc(\xx_0;\pi^2(\uu))=A$, $\cc(\xx_0;\pi(\uu))=ChbA$ and hence the string $\zz:=Chb$ satisfies $\bb^\falpha(\pi^2(\uu))=\pi(\uu)^c\zz$.
\end{example}

\begin{rmk}\label{bandnuminabnC}
Suppose $\uu\in\Vf_i(\xx_0)$, $h(\uu)>2$ and $\upsilon(\pi(\uu))<1$ then since $\hh(\xx_0;\Pp(\pi(\uu)))$ is not a substring of $\cc(\xx_0;\uu)$, Proposition \ref{sbandappearinccontainH} gives that $N(\xx_0;\sbq(\pi(\uu)),\cc(\xx_0;\pi(\uu)))=0$.
\end{rmk}

\begin{proposition}\label{CopposeCharacter}
Suppose $\uu\in\Vf_i(\xx_0)$ and $h(\uu)>2$. Then $\upsilon(\pi(\uu))=-1$ if and only if $\pi(\uu)$ is abnormal, $\pi(\uu)^c\cc(\xx_0;\pi^2(\uu))$ is a string but $\pi(\uu)^c\cc(\xx_0;\pi(\uu))$ is not a string.
\end{proposition}
\begin{proof}
Suppose $\cc(\xx_0;\pi^2(\uu))=\zz\cc(\xx_0;\pi(\uu))$ for some string $\zz$ of positive length. Since both $\cc(\xx_0;\pi^2(\uu))$ and $\cc(\xx_0;\pi(\uu))$ are substrings of $\hh(\xx_0;\Pp(\pi(\uu)))$ and $\hh(\xx_0;\Pp(\pi(\uu)))$ is a proper substring of $\bb^\falpha(\pi(\uu))\cc(\xx_0;\pi(\uu))$, $\zz$ is a left substring of $\bb^\falpha(\pi(\uu))$ and hence $\pi(\uu)$ is abnormal.

Since $\cc(\xx_0;\pi(\uu))$ is a proper left substring of $\cc(\xx_0;\pi^2(\uu))$ the last syllable of $\cc(\xx_0;\pi^2(\uu))$ is a syllable of $\tbq(\pi(\uu))$ but not of $\sbq(\pi(\uu))$. Hence for any cyclic permutation $\bb'$ of $\sbq(\pi(\uu))$ if $\bb'\cc(\xx_0;\pi^2(\uu))$ is a string then $\bb'=\bla{\pi(\uu)}$. Hence $\pi(\uu)^c\cc(\xx_0;\pi^2(\uu))$ is a string. Moreover Proposition \ref{CpropCharacter} gives that $\pi(\uu)^c\cc(\xx_0;\pi(\uu))$ is not a string.

Now $\pi(\uu)^c\cc(\xx_0;\pi^2(\uu))$ is a string if and only if $\gamma'(\pi(\uu))$ is the last syllable of $\cc(\xx_0;\pi^2(\uu))$ where $\gamma'(\pi(\uu))$ is the last syllable of $\bua{\pi(\uu)}$. 

For the reverse implication clearly $\upsilon(\pi(\uu))\neq0$. Moreover  $\upsilon(\pi(\uu))\neq1$ by Proposition \ref{CpropCharacter}, since $\pi(\uu)^c\cc(\xx_0;\pi(\uu))$ is not a string. Therefore $\upsilon(\pi(\uu))=-1$. 
\end{proof}

\begin{corollary}\label{CopposeCharactercor}
Suppose $\uu\in\Vf_i(\xx_0)$ and $h(\uu)>2$. If $\cc(\xx_0;\pi^2(\uu))=\zz\cc(\xx_0;\pi(\uu))$ where $|\zz|>0$. Then the following hold:
\begin{enumerate}
    \item $\bb^\falpha(\pi^2(\uu))=\bla{\pi(\uu)}$;
    \item $(\tbq(\pi(\uu)),\cc(\xx_0,\pi(\uu)))$ is $b$-long with $\rr^b_1(\pi(\uu))=\pi(\uu)^c\zz$;
    \item $(\tbq(\pi(\uu)),\rr^b_1(\pi(\uu))\cc(\xx_0,\pi(\uu)))$ is $b$-long but not $s$-long; 
    \item $\uu$ is long iff $\tilde\ww(\xx_0;\uu)$ is a left substring of $\rr^b_1(\pi(\uu))$;
\end{enumerate}
\end{corollary}

\begin{proof}
Proposition \ref{CopposeCharacter} gives that $\pi(\uu)$ is abnormal and $\pi(\uu)^c\cc(\xx_0;\pi^2(\uu))$ is a string. and hence (1) follows immediately. Moreover $\zz$ is a right substring of $\bua{\pi(\uu)}$. Since $\bb^\falpha(\pi^2(\uu))\cc(\xx_0;\pi^2(\uu))$ is a string, $\zz$ is a proper left substring of $\pi(\uu)_\beta$. Therefore $(\tbq(\pi(\uu)),\cc(\xx_0,\pi(\uu)))$ is $b$-long. By definition of $\rr^b_1(\pi(\uu))$, we clearly get $\rr^b_1(\pi(\uu))=\pi(\uu)^c\zz$ thus proving (2).

Since $\zz$ is a proper left substring of $\pi(\uu)_\beta$, $(\tbq(\pi(\uu)),\rr^b_1(\pi(\uu))\cc(\xx_0,\pi(\uu)))$ is $b$-long but not $s$-long thus proving (3). Moreover if $\tilde\ww(\xx_0;\uu)$ is a left substring of $\rr^b_1(\pi(\uu))$ then $\uu$ is long. The converse part of (4) is by Remark \ref{WRlongcomparison}.
\end{proof}

\begin{corollary}\label{exceptionalnegative}
Suppose $\uu\in\Vf_i(\xx_0)$, $h(\uu)>2$ and $\upsilon(\pi(\uu))=-1$. If $\uu$ is long then it is $b^{-\partial}$-long. Moreover if $\cc(\xx_0;\pi^2(\uu))$ is a left substring of $\tilde\ww(\xx_0;\uu)\cc(\xx_0;\pi(\uu))$ and there is an uncle $\uu_1$ of $\uu$ with $\beta(\uu_1)=\beta(\uu)$ such that $\pi(\uu)\sqsf_\partial\uu_1$ then
\begin{itemize}
    \item $|\pi(\uu)^c|>0$ and $\tilde{\ww}(\xx_0;\uu)=\zz$;
    \item $\uu_1$ is $b^\partial$-long and $|\rr^b_1(\pi^2(\uu))|=0$.
\end{itemize}
\end{corollary}

\begin{example}
In the algebra $\Gamma^{(iii)}$ in Figure \ref{doublelongnegative} the only bands are $\bb_1:=icD$, $\bb_2:=cahG$ and their inverses. 

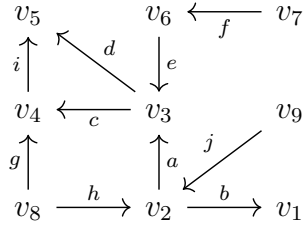
\begin{figure}[h]
    \centering
    \begin{tikzcd}
v_5                               & v_6 \arrow[d, "e"]                  & v_7 \arrow[l, "f"]   \\
v_4 \arrow[u, "i"]                & v_3 \arrow[l, "c"] \arrow[lu, "d"'] & v_9 \arrow[ld, "j"'] \\
v_8 \arrow[u, "g"] \arrow[r, "h"] & v_2 \arrow[r, "b"] \arrow[u, "a"']  & v_1                 
\end{tikzcd}
    \caption{$\Gamma^{(iii)}$ with $\rho=\{bh,aj,da,ce,ig,icah,def\}$}
    \label{doublelongnegative}
\end{figure}

Choosing $\xx_0:=1_{(\vv_1,i)}$ and $\uu:=FEahG$ we have $\pi(\uu)=cD$ and $\pi^2(\uu)=icaB$ with $\upsilon(\uu)=-1$. Here $\uu$ is $b^{-1}$-long and its uncle $\uu_1:=ED$ is $b^{1}$-long. 
\end{example}


\begin{proposition}\label{CequalCharacter}
Suppose $\uu\in\Vf_i(\xx_0)$ and $h(\uu)>2$. Then $\upsilon(\pi(\uu))=0$ if and only if $\pi(\uu)$ is abnormal with $|\pi(\uu)^c|>0$ and there is a partition $\pi(\uu)^c=\xx_2\xx_1$ with $|\xx_1|>0$ and $\xx_1$ is a right substring of $\bb^\falpha(\pi(\uu))$.
\end{proposition}

\begin{proof}
Since $\cc(\xx_0;\pi(\uu))=\cc(\xx_0;\pi^2(\uu))$, $\bb^\falpha(\pi^2(\uu))\cc(\xx_0;\pi(\uu))$ and $\bb^\falpha(\pi(\uu))\cc(\xx_0;\pi(\uu))$ are strings. Let $\alpha,\gamma',\gamma''$ denote the last syllables of $\cc(\xx_0;\pi(\uu)),\bb^\falpha(\pi(\uu)),\bb^\falpha(\pi^2(\uu))$ respectively. Then $\alpha\notin\{\gamma',\gamma'',\fgamma(\pi(\uu)),\fgamma(\pi^2(\uu))\}$.

If $\fgamma(\pi(\uu))\neq\fgamma(\pi^2(\uu))$ then by the definition of a string algebra we must have $\gamma'=\gamma''$. On the other hand if $\fgamma(\pi(\uu))=\fgamma(\pi^2(\uu))$, then since $\fgamma(\pi(\uu))\gamma',\fgamma(\pi^2(\uu))\gamma''$ are strings, we can again conclude $\gamma'=\gamma''$. Therefore $\pi(\uu)$ is abnormal and there exists a partition $\pi(\uu)^c=\xx_2\xx_1$ such that $|\xx_1|>0$.

For the other direction observe that since $|\xx_1|>0$, neither $\pi(\uu)^c\cc(\xx_0;\pi(\uu))$ nor $\pi(\uu)^c\cc(\xx_0;\pi^2(\uu))$ is a string. Recall from Corollary \ref{Comparableparentchildc} that $\cc(\xx_0;\pi^2(\uu))$ and $\cc(\xx_0;\pi(\uu))$ are comparable. Hence Propositions \ref{CpropCharacter} and \ref{CopposeCharacter} yield the necessary result.
\end{proof}

\begin{proposition}\label{abnequalshort} 
If $\uu\in\Vf_i(\xx_0)$, $h(\uu)>2$ and $\upsilon(\pi(\uu))<1$ then $\pi(\uu)$ is short.
\end{proposition}
\begin{proof}
We only prove the result for $\upsilon(\pi(\uu))=0$ and the proof for the another case is similar. 

For a contradiction we assume that $\pi(\uu)$ is long. Since $\upsilon(\pi(\uu))=0$, Proposition \ref{CequalCharacter} guarantees that there is a partition $\pi(\uu)^c=\xx_2\xx_1$ with $|\xx_1|>0$ and $\xx_1$ is a right substring of $\bb^\falpha(\pi(\uu))$. Here $\beta(\pi(\uu))\xx_2$ is a string. Corollary \ref{longequivwithextC} gives that the LVP $(\bb,\yy):=(\sbq(\pi(\uu)),\tilde\ww(\xx_0;\pi(\uu))\cc(\xx_0;\pi^2(\uu)))$ is long and hence $\ww(\xx_0;\pi(\uu))=\xx_2\bb^\falpha{(\pi^2(\uu))}$. Moreover Corollary \ref{Hcontainextc} states that $\ww(\xx_0;\pi(\uu))\cc(\xx_0;\pi^2(\uu))$ is a left substring of $\hh(\xx_0;\Pp(\pi(\uu)))$. Now the maximal left substring $\yy$ of $\bb^\falpha{(\pi^2(\uu))}$ whose last syllable is $\alpha(\pi(\uu))$ satisfies that $\yy\cc(\xx_0;\pi^2(\uu)))$ is a left substring of $\cc(\xx_0;\pi(\uu)))$, a contradiction to $\upsilon(\pi(\uu))=0$. This completes the proof.
\end{proof}

\begin{corollary}\label{exceptionalzero}
Suppose $\uu\in\Vf_i(\xx_0)$, $h(\uu)>2$ and $\upsilon(\pi(\uu))=0$. If $\uu$ is long and there exist an uncle $\uu_1$ of $\uu$ with $\beta(\uu_1)=\beta(\uu)$ then the following hold:
\begin{itemize}
\item $\uu$ is $b^{-\partial}$-long;
\item if $|\tilde\ww(\xx_0;\uu)|>0$ then $\uu_1\sqsf_\partial\pi(\uu)$;
\item $\uu$ is double $(b,s)^{-\partial}$-long only if $|\tilde{\ww}(\xx_0;\uu)|=0$. In this case $\uu_1$ is necessarily $s^\partial$-long, $|\rr^s_1(\pi^2(\uu))|=0$, $|\rr^b_1(\pi(\uu))|>0$ and $\pi(\uu)\sqsf_\partial\uu_1$.
\end{itemize}
\end{corollary}

\begin{rmk}\label{nonpositivelong}
Suppose $\uu\in\Vf_i(\xx_0)$, $h(\uu)>2$ and $\upsilon(\pi(\uu))<1$. If $\uu$ is long then it is either $b^{-\partial}$-long or $(b,s)^{-\partial}$-long.
\end{rmk}

\begin{corollary}\label{abnrfrelation}
Suppose $\uu\in\Vf_i(\xx_0)$, $h(\uu)>2$ and $\pi(\uu)$ is abnormal then
\begin{itemize}
\item the LVP $(\tbq(\pi(\uu)),\cc(\xx_0,\pi(\uu)))$ is $b^{-\theta(\beta(\pi(\uu)))}$-long;
\item $\rr^b_1(\pi(\uu))$ is a right substring of $\ff(\xx_0;\pi(\uu),\pi^2(\uu))$.
\end{itemize}
\end{corollary}
\begin{proof}
Since $\pi(\uu)$ is abnormal Proposition \ref{Cequivnormality} gives that $\ff(\xx_0;\pi(\uu),\pi^2(\uu))=\zz\cc(\xx_0;\pi(\uu))$ for some string $\zz$.

Propositions \ref{CpropCharacter} ($\upsilon(\pi(\uu))=1$), \ref{CequalCharacter} ($\upsilon(\pi(\uu))=0$) and Corollary \ref{CopposeCharactercor} (($\upsilon(\pi(\uu))=-1$)) give that the LVP $(\tbq(\pi(\uu)),\cc(\xx_0,\pi(\uu)))$ is $b^{-\theta(\beta(\pi(\uu)))}$-long.

The following cases describe about $\ww(\xx_0;\pi(\uu))$ depending on the value of $\upsilon(\pi(\uu))$. 

\textbf{Case 1:} $\upsilon(\pi(\uu))=1$.

Suppose $\cc(\xx_0;\pi(\uu))=\zz'\cc(\xx_0;\pi^2(\uu))$ for some positive length string $\zz'$.

If $\pi(\uu)$ is long then Corollary \ref{abnpropshort} guarantees that either $|\pi(\uu)^c|=0$ or $|\pi(\uu)^c|>0$ with $|\tilde{\ww}(\xx_0;\pi(\uu))|=0$, and hence $\zz'=\bb^\falpha(\pi^2(\uu))$.

On the other hand if $\pi(\uu)$ is short then $\zz'$ is a proper left substring of $\bb^\falpha(\pi^2(\uu))$.

Since $\beta(\pi(\uu))\pi(\uu)^c$ is a string we get $\pi(\uu)^c=\zz$ by Proposition \ref{CpropCharacter} and hence $\ww(\xx_0;\pi(\uu))=\pi(\uu)^c\zz'=\rr^b_1(\pi(\uu))\zz'$ by Corollary \ref{blongdueabnormality}. In summary $$\tilde{\ww}(\xx_0;\pi(\uu)):=\begin{cases} \rr^b_1(\pi(\uu))&\mbox{if $\pi(\uu)$ is long;} \\\rr^b_1(\pi(\uu))\zz'&\mbox{otherwise}.\end{cases}$$

\textbf{Case 2:} $\upsilon(\pi(\uu))=0$.

Proposition \ref{abnequalshort} gives that $\pi(\uu)$ is short. Hence by Proposition \ref{CequalCharacter} we get $\ww(\xx_0;\pi(\uu))=\tilde{\ww}(\xx_0;\pi(\uu))=\rr^b_1(\pi(\uu))$.

\textbf{Case 3:} $\upsilon(\pi(\uu))=-1$.

Let $\cc(\xx_0;\pi^2(\uu))=\zz'''\cc(\xx_0;\pi(\uu))$ for some positive length string $\zz'''$. Corollary \ref{CopposeCharactercor} guarantees that $\rr^b_1(\pi(\uu))=\pi(\uu)^c\zz'''$. Since $\beta(\pi(\uu))\pi(\uu)^c$ is a string we get $\tilde{\ww}(\xx_0;\pi(\uu))=\pi(\uu)^c$. Since $\pi(\uu)$ is short by Proposition \ref{abnequalshort}, we have $\ww(\xx_0;\pi(\uu))=\tilde{\ww}(\xx_0;\pi(\uu))=\pi(\uu)^c$. Therefore $\rr^b_1(\pi(\uu))=\ww(\xx_0;\pi(\uu))\zz'''$.

All the three cases above show that $\rr^b_1(\pi(\uu))$ is a right substring of $\ff(\xx_0;\pi(\uu),\pi^2(\uu))$. This completes the proof.
\end{proof}

We close this section with some expository discussion. A single monomial relation in $\rho$ can be responsible for making a child-parent pair in $\Vf_i(\xx_0)$ simultaneously long. When this happens then, for some $\partial\in\{+,-\}$, the parent is $b^\partial$-long and the child is $s^\partial$-long. The definition of $\ch$ makes sure that even if the same relation touches several other bands then it cannot make any other vertex of $\uu$ long. The next proposition characterizes such situation.

\begin{example}
In the algebra $\Gamma^{(iv)}$ from Figure \ref{cumulative} the only bands are $\bb_1:=acD$, $\bb_2:=eHG$. Choosing $\xx_0:=1_{(\vv_1,i)}$ and $\uu:=f$ we have $\pi(\uu):=e$. 

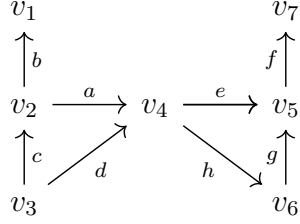
\begin{figure}[h]
    \centering
    \begin{tikzcd}
v_1                                  &                                               & v_7                \\
v_2 \arrow[r, "a"] \arrow[u, "b"']   & v_4 \arrow[r] \arrow[r, "e"] \arrow[rd, "h"'] & v_5 \arrow[u, "f"] \\
v_3 \arrow[u, "c"'] \arrow[ru, "d"'] &                                               & v_6 \arrow[u, "g"]
\end{tikzcd}
    \caption{$\Gamma^{(iv)}$ with $\rho=\{bc,ha,ed,fg,hd,feac\}$}
    \label{cumulative}
\end{figure}

Here $\cc(\xx_0;\pi(\uu))=B$ and $\cc(\xx_0;\uu)=acDaB$. The monomial relation $feac\in\rho$ is responsible to make $\pi(\uu)$ $b$-long and $\uu$ $s$-long. 
\end{example}

\begin{proposition}\label{bandincludedinblong}
Suppose $\uu\in\Vf_i(\xx_0)$, $h(\uu)>2$ and $\pi(\uu)$ is $b$-long. Then $\beta(\uu)\tilde{\ww}(\xx_0;\uu)$ is a substring of $\rr^b_2(\pi(\uu))$ if and only if 
\begin{enumerate}
\item if $\pi(\uu)$ is abnormal then $\upsilon(\pi(\uu))=1$ and $|\pi(\uu)^c|=0$;
\item $\uu$ is $s$-long with $\rr^s_1(\pi(\uu))=\tilde{\ww}(\xx_0;\uu)$ and $$\tilde{\rr}^s(\pi(\uu)):=\begin{cases} \pi(\uu)^o\tilde{\ww}(\xx_0;\pi(\uu))\tilde{\rr}^b(\pi^2(\uu))&\mbox{if $\pi(\uu)$ is normal;} \\\tilde{\ww}(\xx_0;\pi(\uu))\tilde{\rr}^b(\pi^2(\uu))&\mbox{otherwise}.\end{cases}$$
\end{enumerate}
\end{proposition}
\begin{proof}
We only need to prove the forward direction as the backward direction is obvious.

Since $\pi(\uu)$ is $b$-long, if it is abnormal then Proposition \ref{abnequalshort} gives $\upsilon(\pi(\uu))=1$. Further Corollary \ref{abnpropshort} gives that $|\pi(\uu)^c|=0$. To prove $\uu$ is $s$-long it is enough to prove the following claim.

\textbf{Claim:} If $\pi(\uu)$ is $b$-long then $N(\sbq(\pi(\uu)),\cc(\xx_0;\pi(\uu)))=1$.

Since $\pi(\uu)$ is $b$-long we get $N(\xx_0;\sbq(\uu),\ff(\xx_0;\uu))=1$. If $\pi(\uu)$ is normal then Remark \ref{countingsbqnormal} gives that $N(\xx_0;\sbq(\uu),\cc(\xx_0;\uu))=N(\xx_0;\sbq(\uu),\ff(\xx_0;\uu))=1$. If $\pi(\uu)$ is abnormal with $\upsilon(\pi(\uu))=1$ and $|\pi(\uu)^c|=0$ then $\cc(\xx_0;\pi(\uu))=\tilde{\ww}(\xx_0;\pi(\uu))\bb^\falpha{(\pi^2(\uu))}\cc(\xx_0;\pi^2(\uu))$. Hence the claim.
\end{proof}

\section{H-reduced forking strings}\label{Hforkingstring}
Recall that a string $\xx\in H_l(\xx_0)$ is forking if it can be extended on the left by two different syllables.
\begin{proposition}
If $\uu_1\neq\uu_2$ for $\uu_1,\uu_2\in\Vf_i{(\xx_0)}\setminus\{\xx_0\}$ then $\ff(\xx_0;\uu_1,\uu_2)$ is a H-reduced forking string.
\end{proposition}

\begin{proof}
Clearly $\ff(\xx_0;\uu_1,\uu_2)$ is a forking string. Then note that if $\yy$ is a left substring of $\ff(\xx_0;\uu_1)$ and $\HRed{\bb}(\yy)$ exists for some band $\bb$ then $\bb=\tbq(\uu_1)$.

Thus if $\HRed{\bb}(\ff(\xx_0;\uu_1,\uu_2))$ exists for some band $\bb$ then $\bb=\tbq(\uu_1)=\tbq(\uu_2)$. As a consequence we get $\ff(\xx_0;\uu_1)=\ff(\xx_0;\uu_2)$, and hence $\uu_1=\uu_2$ by Proposition \ref{fisinjective}, a contradiction.
\end{proof}

Despite the combinatorial complexity of the collection of strings, by the end of this section, we will ``essentially prove'' the following surprising statement.
\begin{rmk}\label{forkpoints}
If $\yy\in H_l^i(\xx_0)$ is an H-reduced forking string then $\yy$ has one of the following forms for some $\uu\in\Vf_i(\xx_0)$:
\begin{itemize}
    \item $\ff(\xx_0;\uu,\uu^{f+})$;
    \item $\ff(\xx_0;\uu,\pi(\uu))$.
\end{itemize}
\end{rmk}

In this section we guarantee the existence of certain distant relatives of $\uu\in\Vf_i(\xx_0)$ under certain hypotheses and thus describe a particular forking string as a forking string of two immediate siblings. There will be several different versions of this result under different hypotheses; we will prove the uncle-nephew interaction in detail and only indicate the changes in the proofs of the remaining cases. 
\begin{lemma}(Uncle forking lemma--same parity)\label{Forkinglocation}
Suppose $\uu\in\Vf_i(\xx_0)$, $h(\uu)>2$ and $\theta(\beta(\uu))=\theta(\beta(\pi(\uu)))=\partial$ for some $\partial\in\{+,-\}$. Further suppose 
\begin{enumerate}
    \item $\uu$ is long;
    \item $\tilde\ww(\xx_0;\uu)\cc(\xx_0;\pi(\uu))=\tilde\zz\cc(\xx_0;\pi^2(\uu))$ for some string $\tilde\zz$ of positive length;
    \item $\beta(\uu)\tilde\zz\bb^\falpha(\pi^2(\uu))$ is a string.
\end{enumerate}
Then there is an uncle $\uu_1$ of $\uu$ satisfying $\beta(\uu_1)\in\{\beta(\uu),\beta(\pi(\uu))\}$ and $$\ff(\xx_0;\pi(\uu),\uu_1)=\begin{cases}\tilde{\ww}(\xx_0;\uu)\cc(\xx_0;\pi(\uu))&\mbox{if }\uu_1\sqsf_\partial\pi(\uu);\\\tilde{\ww}(\xx_0;\pi(\uu))\cc(\xx_0;\pi^2(\uu))&\mbox{if }\pi(\uu)\sqsf_\partial\uu_1.\end{cases}$$
\end{lemma}

The proof is long and we need to set up more notation.

If $\upsilon(\pi(\uu))>-1$ then let $\zz$ satisfy $\cc(\xx_0;\pi(\uu))=\zz\cc(\cc_0;\pi^2(\uu))$. Note that the word $\bb^\falpha(\pi(\uu))\zz\bb^\falpha(\pi^2(\uu))\cc(\xx_0;\pi^2(\uu))$ is a string. Let $\tilde\ww$ be the shortest right substring of $\zz\bb^\falpha(\pi^2(\uu))$ such that $\bb^\falpha(\pi(\uu))\tilde\ww\sbq(\pi(\uu))$ is a string.

If $\upsilon(\pi(\uu))=-1$ then let $\zz$ satisfy $\cc(\xx_0;\pi^2(\uu))=\zz\cc(\cc_0;\pi(\uu))$. If $\cc(\xx_0;\pi^2(\uu))$ is a substring of $\tilde\ww(\xx_0;\uu)\cc(\xx_0;\pi(\uu))$ then $\zz$ is a substring of $\tilde\ww(\xx_0;\uu)$. Let $\tilde\ww$ be the shortest right substring of $\bb^\falpha(\pi^2(\uu))$ such that $\tilde\ww\sbq(\pi(\uu))$ is a string. 

In both cases $\tilde\ww$ is a left substring of $\pi(\uu)$.

Since $\uu$ is long, Proposition \ref{longequivwithextC} gives that $(\bb,\yy):=(\sbq(\uu),\tilde\ww(\xx_0;\uu)\cc(\xx_0;\pi(\uu)))$ is long. If $\uu$ is double long, then it is either $(b,s)^{-\partial}$-long or $(s,b)^{-\partial}$-long.

The proof will be completed in four cases depending on whether $\uu$ is single/double $b/s^{\pm\partial}$-long but the general plan in each case is as follows.

\noindent{\textbf{Step 1}:} To show that for some string $\yy''$ the string $\yy''\beta(\uu)\tilde\ww(\xx_0;\uu)\tilde\ww$ is an H-reduced string.

\noindent{\textbf{Step 2}:} Since $\beta(\uu)\tilde\zz\bb^\falpha(\pi^2(\uu))$ is a string the hypotheses of the following lemma are satisfied, and it yields $\bar\uu\in\bHQ$ such that $\yy''\beta(\uu)\tilde\ww(\xx_0;\uu)\tilde\ww$ is a left substring of $\tbq(\bar\uu)\bar\uu$. Then \cite[Proposition~8.8,Theorem~8.9]{SK} yield an arch bridge $\uu_1$ and a weak arch bridge $\uu_2$ such that $\bar{\uu}=\uu_2\ch\uu_1$. Since $N(\sbq(\uu),\bar\uu)=0$ and the LVP $(\sbq(\uu),\tilde\ww(\xx_0;\uu)\cc(\xx_0;\pi(\uu)))$ is long, we get $\uu_1\neq\pi(\uu)$.

\begin{lemma}\label{existencereverse(weak)archbridge}
\begin{enumerate}
    \item If $\ww$ is an H-reduced string and $\bb$ is a band such that $\ww\bb$ is a skeletal string that is not a substring of $\bb^2$ then there is a string $\ww'$ and $\uu\in\bHQ$ such that $\ww'\ww\bb=\tbq(\uu)\uu\bb$.
    \item If $\ww$ is an H-reduced string such that $\ww\xx_0$ is a skeletal string with respect to $(i,\xx_0)$ then there is a string $\ww'$ and $\uu\in\bHQ$ such that $\ww'\ww\xx_0=\tbq(\uu)\uu\xx_0$.
\end{enumerate}
\end{lemma}
\begin{proof}
We only prove the first statement; a similar proof works for the other case.

There are two possibilities:
\begin{enumerate}
\item If $\ww\bb$ is torsion then taking $\ww'$ to be a length $0$ string, \cite[Proposition~8.8]{SK} gives that $\ww$ is a maximal reverse weak arch bridge.
\item If $\ww\bb$ is not torsion then there is a $\alpha\in Q_1$ such that $\alpha\ww\bb$ is a string. There are further two subcases.
\begin{enumerate}
    \item If $\alpha\ww$ is not H-reduced then there is a band $\bb_1$ such that $\alpha\in\bb_1$. Since $\ww\bb$ is not a substring of $\bb^2$ we get $\bb\neq\bb_1$. Let $\uu$ be the shortest left substring of $\alpha\ww\bb$ such that $\bb_1\uu\bb$ is a string. Since $\alpha\ww\bb$ is not H-reduced, $\uu$ is a proper left substring of $\alpha\ww\bb$. Hence being a left substring of the H-reduced string $\ww$ \cite[Proposition~5.11]{SK} gives that $\uu$ is also H-reduced. Thus $\uu$ is a weak arch bridge from $\bb$ to $\bb_1$ by \cite[Proposition~8.8]{SK}.
    \item If $\alpha\ww\bb$ is H-reduced then we can proceed as above by replacing $\ww\bb$ by $\alpha\ww\bb$.
\end{enumerate}
\end{enumerate}

Since there are only finitely many H-reduced strings by \cite[Corollary~5.8]{SK}, after finitely many iterations of case (2b) we must land in either case (1) or case (2a), and thus the result.
\end{proof}

In this step we also show that $\beta(\uu_1)\in\{\beta(\uu),\beta(\pi(\uu))\}$. We begin with an observation.
\begin{rmk}
Suppose $\uu\in\Vf_i(\xx_0)$, $h(\uu)>2$ and $\theta(\beta(\uu))=\theta(\beta(\pi(\uu)))=\partial$ for some $\partial\in\{+,-\}$. If $\cc(\xx_0;\pi^2(\uu))$ is a substring of $\tilde\ww(\xx_0;\uu)\cc(\xx_0;\pi(\uu))$ and $|\pi(\uu)^c|=0$ then $\upsilon(\pi(\uu))=1$.
\end{rmk}

In view of the above remark the following two lemmas are sufficient to complete the proof of this step in all four cases.

\begin{lemma}\label{normalHfactor}
Assuming the hypotheses of Lemma \ref{Forkinglocation} and using the notations above, suppose $\upsilon(\pi(\uu))=1$. If either $\pi(\uu)$ is normal or $\pi(\uu)$ is abnormal with $|\pi(\uu)|^c=0$ then
\begin{itemize}
    \item $\beta(\uu_1)=\beta(\pi(\uu))=\beta(\bar\uu)$;
    \item $\uu_1$ is long if and only if $\pi(\uu)$ is long.
\end{itemize}
\end{lemma}
\begin{proof}
Since $\pi(\uu)$ is normal, $\beta(\pi(\uu))\in\tilde{\ww}$ by the construction of $\tilde\ww$. Moreover the construction of $\bar{\uu}$ gives $\beta(\pi(\uu))=\beta(\bar{\uu})$. If $\pi(\uu)$ is abnormal with $|\pi(\uu)|^c=0$ then $|\tilde\ww(\xx_0;\uu)|>0$, and hence the same conclusion follows.

For a contradiction suppose that $\beta(\uu_1)\neq\beta(\bar{\uu})$. Then $\uu_2$ is non-trivial.

\noindent{\textbf{Claim:}} $N(\tbq(\uu_1),\bar{\uu})=0$

Suppose not. Since $\bar\uu$ is H-reduced by \cite[Theorem~8.6]{SK}, we get $N(\tbq(\uu_1),\bar{\uu})=1$. Hence $\bar\uu=\sk{\uu_2\tbq(\uu_1)\uu_1}$.

If $\uu_1$ is normal then $\beta(\uu_1)\in\uu_1$. Hence $\beta(\bar\uu)=\beta(\uu_1)$, a contradiction. On the other hand if $\uu_1$ is abnormal then $\beta(\uu_1)\uu_1^c$ is a substring of $\bar\uu$, and hence $\beta(\bar\uu)=\beta(\uu_1)$, again a contradiction. This completes the proof of the claim.

Now the claim combined with $\beta(\uu_1)\neq\beta(\bar\uu)$ gives that $\beta(\uu_1)\in\tbq(\uu_1)$. Therefore $\uu_1$ is abnormal. As a consequence there is a partition $\uu_1^e=\xx_2\xx_1$ such that $|\xx_2|>0,|\xx_1|\neq|\uu_1^c|$ and $\beta(\bar\uu)\xx_1$ is a string. Since $\beta(\pi(\uu))=\beta(\bar{\uu})$, $\xx_2\tilde\ww(\xx_0;\pi(\uu))$ is a string and hence $(\sbq(\pi(\uu)),\tilde\ww(\xx_0;\pi(\uu))\cc(\xx_0;\pi^2(\uu)))$ is $b$-long. Since $\beta(\bar\uu)=\beta(\pi(\uu))$ there are two cases.

If $|\xx_1|<|\uu_1^c|$ then $\beta(\uu_2)=\beta(\bar\uu)$. On the other hand, if $|\xx_1|>|\uu_1^c|$ then $\beta(\uu_2)$ is the first syllable of $\uu_1^\beta$. In both cases $(\sbq(\uu_2),\tilde\ww(\xx_0;\uu_2)\cc(\xx_0;\pi(\uu_2)))$ is $b$-long. Hence $N(\tbq(\uu_1),\bar\uu)=1$, a contradiction to the claim. Therefore our assumption is wrong and we conclude $\beta(\uu_1)=\beta(\bar\uu)$.

As $\pi^2(\uu)=\pi(\uu_1)$, $\cc(\xx_0;\pi^2(\uu))=\cc(\xx_0;\pi(\uu_1))$ and $\beta(\uu_1)=\beta(\pi(\uu))$ we get $\tilde{\ww}(\xx_0;\pi(\uu))=\tilde{\ww}(\xx_0;\uu_1)$. Hence from Corollary \ref{longchildlocation}, $\uu_1$ is long if and only if $\pi(\uu)$ is long.
\end{proof}

\begin{lemma}\label{abnsign}
Assuming the hypotheses of Lemma \ref{Forkinglocation} and using the notations above, if $\pi(\uu)$ is abnormal with $|\pi(\uu)^c|>0$ then $\beta(\bar{\uu})=\beta(\uu_1)=\beta(\uu)$. 
\end{lemma}

\begin{proof}
We prove the result for $\upsilon(\pi(\uu))=1$. With slight modification we can get the proof for the other two cases.

Since $\cc(\xx_0;\pi^2(\uu))$ is a proper left substring of $\cc(\xx_0;\pi(\uu))$, $\uu$ is long and $|\pi(\uu)^c|>0$, then Corollaries \ref{slongdueabnormality} and \ref{blongdueabnormality} together guarantee that $\uu$ is $b$-long. Hence the latter corollary guarantees that $\rr^b_1(\pi(\uu))=\pi(\uu)^c$ and that $\tilde\ww(\xx_0;\uu)$ is a substring of $\pi(\uu)^c$. Since $\tilde\ww(\xx_0;\uu)\tilde\ww$ is a left substring of $\sbq(\pi(\uu))$, we see that $\beta(\bar\uu)=\beta(\uu)$. Hence $\tilde{\ww}$ is a left substring of $\uu_1$ thanks to \cite[Proposition~4.3]{SK}. 

Suppose $\beta(\bar{\uu})\neq\beta(\uu_1)$. Then $\uu_2$ is non-trivial, and $\uu_1$ is abnormal with $|\tilde{\ww}(\xx_0;\uu)|<|\uu_1^e|$ and $N(\tbq(\uu_1),\bar\uu)=0$. But in this case $(\sbq(\uu_2),\tilde\ww(\xx_0;\uu_2)\cc(\xx_0;\pi(\uu_2)))$ is $b$-long which guarantees that $N(\tbq(\uu_1),\bar{\uu})=1$, a contradiction. Hence $\beta(\bar{\uu})=\beta(\uu_1)$.
\end{proof}

\noindent{\textbf{Step 3}:} The construction of $\bar\uu$ and $\uu_1$ guarantees that $\pi(\uu)\sqsf_\partial\uu_1$ if and only if $\pi(\uu)$ is short but $\uu_1$ is long.

\begin{example}
Continuing from Example \ref{unclelongpositive} we get $\pi(\uu)$ is short whereas uncle $\uu_1$ of $\uu$ is $b^{-1}$-long. Here indeed $\ff(\xx_0;\uu_1)=GBDicBa$, $\ff(\xx_0;\pi(\uu),\uu_1)=cBa= \newline \tilde\ww(\xx_0;\pi(\uu))\cc(\xx_0;\pi^2(\uu))$ and $\pi(\uu)\sqsf_1\uu_1$.
\end{example}

If $\uu_1$ and $\pi(\uu)$ are both long or short then $\uu_1\sqsf_\partial\pi(\uu)$. We then show that $\ff(\xx_0;\pi(\uu),\uu_1)=\tilde{\ww}(\xx_0;\uu)\cc(\xx_0;\pi(\uu))$.

If $\pi(\uu)$ is short but $\uu_1$ is long, and hence $\pi(\uu)\sqsf_\partial\uu_1$, then in view of Remark \ref{longshortchildflocation} we have $\ff(\xx_0;\pi(\uu),\uu_1)=\tilde{\ww}(\xx_0;\pi(\uu))\cc(\xx_0;\pi^2(\uu))$. The three subcases, namely  $\upsilon(\pi(\uu))=1$,  $\upsilon(\pi(\uu))=0$ and $\upsilon(\pi(\uu))=-1$ are described in Corollaries \ref{exceptionalpositive}, \ref{exceptionalzero} and \ref{exceptionalnegative} respectively.
    
If $\pi(\uu)$ is long but $\uu_1$ is short, and hence $\uu_1\sqsf_\partial\pi(\uu)$, then Corollary \ref{abnpropshort} gives that $\upsilon(\pi(\uu))=1$, $|\pi(\uu)^c|>0$, $\tilde\ww(\xx_0;\uu_1)=\zz$ and $|\tilde\ww(\xx_0;\uu)|=0$. Hence $\ff(\xx_0;\pi(\uu),\uu_1)=\tilde{\ww}(\xx_0;\uu)\cc(\xx_0;\pi(\uu))$ in view of Remark \ref{longshortchildflocation}.



Now we can analyse the four cases to complete the proof of Lemma \ref{Forkinglocation}.

\noindent{\textbf{Case I:}} $\uu$ is $b^\partial$-long or $(s,b)^{-\partial}$-long.

In this case, $\rr^1_b(\pi(\uu))=\tilde\ww(\xx_0;\uu)$. Remark \ref{nonpositivelong} gives that $\upsilon(\pi(\uu))=1$. Further Corollary \ref{blongdueabnormality} guarantees that $\pi(\uu)$ is normal. Moreover, $\rr^b_1(\pi(\uu))\tilde\ww$ is a substring of $\pi(\uu)$, and hence $\rr^b_1(\pi(\uu))\tilde\ww$ is an H-reduced string by \cite[Proposition~5.11]{SK}. 

\textbf{Step I.1:} the string $\rr^b_2(\pi(\uu))\rr^b_1(\pi(\uu))\tilde{\ww}$ is H-reduced. 

Suppose not. Then there is a cyclic permutation $\bb'_1$ of a band $\bb_1$ such that $\rr^b_2(\pi(\uu))\rr^b_1(\pi(\uu))\tilde{\ww}=\yy'_2\bb'_1\yy'_1$ and $\bb'_1\yy'_1\equiv_H\yy'_1$. Moreover there is a weak bridge $\sbq(\pi(\uu))\xrightarrow{\uu''}\bb_1$. Since $\delta(\rr^b_2(\pi(\uu))\rr^b_1(\pi(\uu)))\neq0$ and $\rr^b_1(\pi(\uu))\tilde{\ww}$ is H-reduced, $\bb'_1$ contains at least one syllable of $\tilde{\ww}$ and $\rr^b_2(\pi(\uu))$ each.

If $\rr^b_2(\pi(\uu))$ is not a substring of $\bb'_1$ then $\bb$ and $\bb_1$ commute contradicting the domesticity of the algebra. Hence $\rr^b_2(\pi(\uu))$ is a right substring of $\bb'_1$ and $|\yy'_2|=0$. But then there is an abnormal weak bridge $\bb_1\xrightarrow{\uu'}\sbq(\uu)$ such that ${\uu'}^c=\rr^b_1(\pi(\uu))$.

Let $\bb''_1,\bb'_2$ denote the cyclic permutations of $\bb_1$ and $\sbq(\uu)$ respectively such that $\rr^b_1(\pi(\uu))$ is a right substring of $\bb''_1$ and $\bb'_2\bb''_1\rr^b_1(\pi(\uu))\tilde\ww$ is a string. If $|\rr^b_1(\pi(\uu))|>0$ then clearly $\bb''_1\rr^b_1(\pi(\uu))\tilde\ww\equiv_H\rr^b_1(\pi(\uu))\tilde\ww$, which yields a factorization of $\pi(\uu)=\uu'\ch\uu''$, a contradiction to $\pi(\uu)\in\HQ$. Hence $|\rr^b_1(\pi(\uu))|=0$. 

In this case let $\bb'_1=\rr^b_2(\pi(\uu))\yy'_3$ so that $\tilde\ww=\yy'_3\yy'_1$. If $\delta(\yy'_3)=0$ then clearly $\bb''_1\tilde\ww\equiv_H\tilde\ww$. On the other hand if $\delta(\yy'_3)\neq0$ then \cite[Proposition~5.2]{SK} together with $\bb'_1\yy'_1\equiv_H\yy'_1$ gives that $\bb''_1\tilde\ww\equiv_H\tilde\ww$. Therefore in each case we conclude that $\pi(\uu)\notin\HQ$. This contradiction proves the claim.

\textbf{Step I.2:}
Since $\rr^b_2(\pi(\uu))\rr^b_1(\pi(\uu))\tilde{\ww}$ is H-reduced from the claim above, Lemma \ref{existencereverse(weak)archbridge} yields a string $\ww'$ and $\bar\uu\in\bHQ_i(\xx_0)$ such that $\tbq(\bar\uu)\bar{\uu}=\ww'\rr^b_2(\pi(\uu))\rr^b_1(\pi(\uu))\tilde{\ww}$.

Now \cite[Proposition~8.8, Theorem~8.9]{SK} yields an arch bridge $\uu_1$ and a possibly trivial weak arch bridge $\uu_2$ such that $\bar{\uu}=\uu_2\ch\uu_1$. Lemma \ref{normalHfactor} guarantees that $\beta(\uu_1)=\beta(\pi(\uu))$ and that $\uu_1$ is long if and only if $\pi(\uu)$ is long. Since $|\rr^b_2(\pi(\uu))|>0$ it readily follows that $\uu_1\sqsf\pi(\uu)$. 

\textbf{Step I.3:} We want to show that $\ff(\xx_0;\pi(\uu),\uu_1)=\tilde\ww(\xx_0;\uu)\cc(\xx_0;\pi(\uu))$.

Suppose not. Then $\uu_2$ is non-trivial. Let $\bb_1:=\tbq(\uu_1)$. If for any cyclic permutation $\bb'_1$ of $\tbq(\uu_1)$, $\bb'_1\ff(\xx_0;\pi(\uu),\uu_1)$ is not a word then $\uu_1$ is normal and $\beta(\uu)\rr^b_1(\pi(\uu))\pi(\uu)^o$ is a substring of $\uu_1^o$ and hence the required identity holds. Let $\bb'_1$ be its cyclic permutation such that $\bb'_1\ff(\xx_0;\pi(\uu),\uu_1)$ is a word.

There are two cases.
 
If $\rr^b_1(\pi(\uu))\cc(\xx_0;\pi(\uu))=\zz'\ff(\xx_0;\pi(\uu),\uu_1)$ for some string $\zz'$ with $|\zz'|>0$ then there are further two subcases.
\begin{itemize}
\item $\zz'=\rr^b_1(\pi(\uu))\zz'_1$ with $|\zz'_1|>0$.

If $|\rr^b_1(\pi(\uu))|>0$ then $\delta(\fgamma(\pi(\uu))\bar{\bar\gamma}^b(\pi(\uu)))=0$ can be used to show that $\pi(\uu)=\uu'_2\ch\uu_1$ for some $\uu'_2$, which is a contradiction to $\pi(\uu)\in\HQ_i(\xx_0)$.

If $|\rr^b_1(\pi(\uu))|=0$ then since $\pi(\uu)\in\HQ$, we obtain that $\bb^\falpha(\pi(\uu))\zz'_1\bb'_1$ is not a string. Then there is a right substring $\ww_1$ of $\bb'_1$ and a left substring $\ww_2$ of $\bb^\falpha(\pi(\uu))\zz'_1$ such that $\ww_2\ww_1\in\rho\cup\rho^{-1}$. Since $\rr^b_2(\pi(\uu))\zz'_1\bb'_1$ is a string, we get that $|\zz'_1|<|\ww_2|$. Similarly since $\bb^\falpha(\pi(\uu))\cc(\xx_0;\uu)$ is a string, we get that $\ww_1$ is not a right substring of $\ff(\xx_0;\pi(\uu),\uu_1)$. As a consequence $\ff(\xx_0;\pi(\uu),\uu_1)\nequiv_H\bb'_1\ff(\xx_0;\pi(\uu),\uu_1)$. This is a contradiction to $\bar\uu=\uu_2\ch\uu_1$ and $N(\bb_1,\bar\uu)=0$.

\item $\ff(\xx_0;\uu_1,\pi(\uu))=\zz'_1\cc(\xx_0;\pi(\uu))$ for $|\zz'_1|\geq0$.

Here $|\rr^b_1(\pi(\uu))|>|\zz'_1|$. Then using $\bar{\uu}=\uu_2\ch\uu_1$ we get that $\bb^\falpha(\pi(\uu))\bb'_1$ and $\bb'_1\bb^\falpha(\pi(\uu))$ are strings, a contradiction to domesticity.
\end{itemize}

If $\zz'\rr^b_1(\pi(\uu))\cc(\xx_0;\pi(\uu))=\ff(\xx_0;\pi(\uu),\uu_1)$ for some string $\zz'$ with $|\zz'|>0$ then $\gamma\in\sbq(\uu)$ is also a syllable of $\bb_1$, where $\gamma\rr^b_1(\pi(\uu))$ is a string, and thus $\theta(\gamma)=-\partial$. Suppose $\gamma'$ is the syllable of $\bb_1$ such that $\gamma\gamma'$ is a substring of a cyclic permutation of $\bb_1$.

If $\theta(\gamma')=\partial$ then $\bb'_1$ and $\bb^\falpha(\pi(\uu))$ share $\gamma\gamma'$ with $\delta(\gamma\gamma')=0$, a contradiction to domesticity. 

On the other hand if $\theta(\gamma')=\partial$ then let $\bb''_1$ be the cyclic permutation of $\bb_1$ such that $\gamma\bb''_1$ is a string. To ensure domesticity there is a right substring $\ww_1$ of $\bb''_1$ and a left substring $\ww_2$ of a cyclic permutation $\bb''$ of $\sbq(\uu)$ such that $|\ww_2||\ww_1|>0$ and $\ww_2\ww_1\in\rho\cup\rho^{-1}$. Since $\bb''\zz'\rr^b_1(\pi(\uu))\cc(\xx_0;\pi(\uu))$ is a string, $\ww_1$ is not a right substring of $\zz'\rr^b_1(\pi(\uu))\cc(\xx_0;\pi(\uu))$. As a consequence, $\bb''_1\zz'\rr^b_1(\pi(\uu))\cc(\xx_0;\pi(\uu))\nequiv_H\zz'\rr^b_1(\pi(\uu))\cc(\xx_0;\pi(\uu))$, a contradiction to $\bar\uu=\uu_2\ch\uu_1$ and $N(\bb_1,\bar\uu)=0$.

\begin{example}
In the algebra $\Gamma^{(v)}$ in Figure \ref{normaluncle} the only bands are $\bb_1:=aCB$, $\bb_2:=ihG$. Choose $\xx_0:=1_{(\vv_1,i)}$ and $\uu:=j$ so that $\pi(\uu)=iDe$ and $\pi^2(\uu)=a$. 

\begin{figure}[h]
    \centering
    \begin{tikzcd}
v_1 \arrow[r, "a"'] \arrow[rd, "c"'] & v_2 \arrow[r, "e"'] & v_4 \arrow[d, "f"] & v_6 \arrow[l, "d"] \arrow[d, "i"] & v_8 \arrow[l, "h"] \arrow[ld, "g"] & v_{10} \\
                                     & v_3 \arrow[u, "b"]  & v_5                & v_7 \arrow[r, "j"]                & v_9 \arrow[ru, "k"]                &       
\end{tikzcd}
    \caption{$\Gamma^{(v)}$ with $\rho=\{eb,fd,dh,jg,kjih\}$}
    \label{normaluncle}
\end{figure}
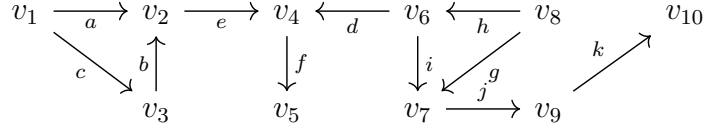

Clearly $\theta(\beta(\uu))=-1$, $\uu$ is $b^{-1}$-long with $\rr^1_b(\pi(\uu))=\tilde\ww(\xx_0;\uu)=i$, $\rr^2_b(\pi(\uu))=kj$ and $\tilde{\ww}=De$. It is readily verified that $\bar{\uu}:=kjiDe$ is H-reduced and $\uu_1:=\bar{\uu}$ is the uncle of $\uu$ such that $\ff(\xx_0;\pi(\uu),\uu_1)=iDea$.
\end{example}

\noindent{\textbf{Case II:}} $\uu$ is $b^{-\partial}$-long.

In view of Remark \ref{nonpositivelong} we have three subcases depending on the value of $\upsilon(\pi(\uu))$. Here we show the proof only for the case $\upsilon(\pi(\uu))=1$; the proofs for the remaining two cases are obvious modifications.

In this case $|\tilde{\ww}(\xx_0;\uu)|\neq|\rr^b_1(\pi(\uu))|$. If $|\tilde{\ww}(\xx_0;\uu)|>|\rr^b_1(\pi(\uu))|$ then $\yy\equiv_H\bb_\yy\yy$, and hence $(\bb,\yy)$ is short, a contradiction to our assumption. Therefore $|\tilde{\ww}(\xx_0;\uu)|<|\rr^b_1(\pi(\uu))|$. As a consequence $|\rr^b_1(\pi(\uu))|>0$. Moreover, $\delta(\fgamma(\pi(\uu))\bar{\bar\gamma}^b(\pi(\uu)))=0$. Since $\tilde\ww(\xx_0;\uu)\tilde\ww$ is a left substring of $\pi(\uu)$, it is an H-reduced string by \cite[Proposition~5.11]{SK}.

\textbf{Step II.1:} $\beta(\uu)\tilde{\ww}(\xx_0;\uu)\tilde{\ww}$ is H-reduced. 

If not then there is a cyclic permutation $\bb'_1$ of a band $\bb_1$ such that $\beta(\uu)\tilde{\ww}(\xx_0;\uu)\tilde{\ww}=\bb'_1\yy'_1$. Suppose $\bb'_1\beta(\uu)=\beta(\uu)\bb''_1$ and $\beta'$ is the last syllable of $\bb''_1$ that is not in $\tilde\ww(\xx_0;\uu)$. Then clearly $\theta(\beta')=\theta(\beta(\uu))=\partial$, and hence it is readily verified that $\bb$ and $\bb_1$ commute, a contradiction to the domesticity of the algebra. This proves the claim.

\textbf{Step II.2:}
Then we can argue as in Case I that there is $\ww'$ such that $\tbq(\bar\uu)\bar{\uu}:=\ww'\beta(\uu)\tilde{\ww}(\xx_0;\uu)\tilde{\ww}\in\bHQ$, and that $\bar\uu=\uu_2\ch\uu_1$ for some $\uu_1\in\HQ_i(\xx_0)$ and possibly trivial $\uu_2$. 

\textbf{Step II.3:} Assuming that $\uu_1$ and $\pi(\uu)$ are either both long or both short, we need to show that $\ff(\xx_0;\pi(\uu),\uu_1)=\tilde{\ww}(\xx_0;\uu)\cc(\xx_0;\pi(\uu))$.

Suppose not. Then $\uu_2$ is non-trivial. Let $\bb_1:=\tbq(\uu_1)$. First we argue that there is a cyclic permutation $\bb'_1$ of $\tbq(\uu_1)$ such that $\bb'_1\ff(\xx_0;\pi(\uu),\uu_1)$ is a string. If not then there is a substring $\bar{\yy}$ of $\bar\uu$ such that $\bb''_1\bar{\yy}\ff(\xx_0;\pi(\uu),\uu_1)$ is a string for some cyclic permutation $\bb''_1$ of $\tbq(\uu_1)$. Therefore $\bar{\yy}\tilde{\ww}(\xx_0;\uu)\tilde{\ww}$ is a string and the required identity follows.

There are two cases.

If $\ff(\xx_0;\pi(\uu),\uu_1)$ is a proper left substring of $\tilde{\ww}(\xx_0;\uu)\cc(\xx_0;\pi(\uu))$ then $\ff(\xx_0;\pi(\uu),\uu_1)$ is necessarily a proper left substring of $\cc(\xx_0;\pi(\uu))$ for otherwise $\bb_1$ commutes with $\tbq(\pi(\uu))$. However if $\ff(\xx_0;\pi(\uu),\uu_1)$ is a proper left substring of $\cc(\xx_0;\pi(\uu))$ then the proof of impossibility is similar to the proof of the first subcase of the first case of the proof of the same identity from Case I. 

If $\tilde{\ww}(\xx_0;\uu)\cc(\xx_0;\pi(\uu))$ is a proper left substring of $\ff(\xx_0;\pi(\uu),\uu_1)$. Then the first syllable of $\tbq(\pi(\uu))$ after $\tilde\ww(\xx_0;\uu)$ is also a syllable of $\bb_1$. To ensure domesticity, $\ff(\xx_0;\pi(\uu),\uu_1)$ is a left substring of $\rr^b_1(\pi(\uu))\cc(\xx_0;\pi(\uu))$. There are two subcases:
\begin{itemize}
\item If $\tilde{\rr}^b(\pi(\uu))$ is not a substring of $\bb_1$ then for the existence of an abnormal weak bridge $\bb_1\xrightarrow{\uu''}\pi(\uu)$, it is necessary to have $\ff(\xx_0;\pi(\uu),\uu_1)=\rr^b_1(\pi(\uu))\cc(\xx_0;\pi(\uu))$ and that $\rr^b_2(\pi(\uu))$ is a left substring of $\bb'_1$. Then it is readily verified that $\pi(\uu)=\uu''\ch\uu_1$, a contradiction to $\pi(\uu)\in\HQ$.

\item If $\tilde{\rr}^b(\pi(\uu))$ is a substring of $\bb_1$ then consider the cyclic permutation $\bb'''_1$ of $\bb_1$ such that $\beta(\uu_2)\bb'''_1$ is a string. To ensure $\pi(\uu)\in\HQ$ there is a right substring $\yy_1$ of $\bb'''_1$ and a substring $\yy_2$ of $\tbq(\pi(\uu))$ such that $\yy_2\yy_1\in\rho\cup\rho^{-1}$. But in this case $\uu_2$ is long and hence $N(\tbq(\uu_1),\bar{\uu})=1$ which is a contradiction to $N(\tbq(\uu_1),\tilde{\ww}(\xx_0;\uu)\tilde{\ww})=0$.  
\end{itemize}

\begin{example}
Consider the algebra $\tilde{\Gamma}^{(ii)}$ that differs in its presentation from $\Gamma^{(ii)}$ from Figure \ref{doublelongpositive} in that $dbgh$ is removed from $\rho$. Keeping $\xx_0$ and $\uu$ same as in Example \ref{unclelongpositive} we get $\theta(\beta(\uu))=1$, $\pi(\uu)$ is short and $\uu$ is $b^{-1}$-long with $\tilde\ww(\xx_0;\uu)=1_{(\vv_3,i)}$, $\rr^1_b(\pi(\uu))=c$, $\rr^2_b(\pi(\uu))=i$ and $\tilde{\ww}=BD$. It is readily verified that $\bar{\uu}:=HGBD$ is H-reduced and $\uu_1:=\bar{\uu}$ is the short uncle of $\uu$ such that $\ff(\xx_0;\pi(\uu),\uu_1)=Ba=\tilde\ww(\xx_0;\uu)\cc(\xx_0;\pi(\uu))$.     
\end{example}

\noindent{\textbf{Case III:}} $\uu$ is $s^\partial$-long or $(b,s)^{-\partial}$-long. In this case $\tilde{\ww}(\xx_0;\uu)=\rr^s_1(\pi(\uu))$.

\textbf{Step III.1:} $\beta(\uu)\tilde{\ww}(\xx_0;\uu)\tilde{\ww}$ is H-reduced.

First we argue that $\tilde{\ww}(\xx_0;\uu)\tilde{\ww}$ is band-free. Suppose not. Since $\tilde{\ww}$ is a left substring of $\pi(\uu)$, $\tilde{\ww}$ is H-reduced. Thus $|\tilde\ww||\tilde{\ww}(\xx_0;\uu)|>0$ and hence $\fgamma(\pi(\uu))\in\tilde{\ww}(\xx_0;\uu)$, and there is a cyclic permutation $\bar{\bb}'_1$ of a band $\bar{\bb}_1$ such that $\tilde{\ww}(\xx_0;\uu)\tilde{\ww}=\yy'_2\bar{\bb}'_1\yy'_1$, the first syllable of $\yy'_2$ is not a syllable of $\bar{\bb}_1$ and $\yy'_1$ is a proper left substring of $\tilde{\ww}$. Since $\fgamma(\pi(\uu))$ is a common syllable of both the bands, in view of \cite[Proposition~4.3]{SK} there is an abnormal weak bridge from $\sbq(\uu)$ to $\bar{\bb}_1$ and $|\yy'_2|>0$. On the other hand since $\bb^\falpha(\pi(\uu))\tilde\ww$ is a string and $\tilde\ww(\xx_0;\uu)$ is a left substring of $\bb^\falpha(\pi(\uu))$ there is a weak bridge from $\bar{\bb}_1$ to $\sbq(\uu)$, a contradiction to the domesticity of the algebra. This completes the proof that $\tilde{\ww}(\xx_0;\uu)\tilde{\ww}$ is band-free.

If $\beta(\uu)\tilde{\ww}(\xx_0;\uu)\tilde{\ww}$ is not band-free then in view of the above paragraph there is a cyclic permutation $\bb'_1$ of a band $\bb_1$ such that $\beta(\uu)\tilde{\ww}(\xx_0;\uu)\tilde{\ww}=\bb'_1\yy'_1$. Moreover since $\delta(\beta(\uu)\rr^s_1(\pi(\uu))\tilde{\rr}^s(\pi(\uu)))\neq0$ we conclude that $\beta(\uu)\rr^s_1(\pi(\uu))\bar\gamma^s(\pi(\uu))$ is a substring $\bb'_1$. Since $\theta(\bar{\bar\gamma}^s(\pi(\uu)))=-\theta(\fgamma(\pi(\uu)))$ it is readily verified that $\bb_1$ and $\sbq(\uu)$ commute, a contradiction to domesticity that proves the claim.

\textbf{Step III.2:}
We can argue as in Case I that there is $\ww'$ such that $\tbq(\bar\uu)\bar{\uu}:=\ww'\beta(\uu)\tilde{\ww}(\xx_0;\uu)\tilde{\ww}\in\bHQ$, and that $\bar\uu=\uu_2\ch\uu_1$ for some $\uu_1\in\HQ_i(\xx_0)$ and possibly trivial $\uu_2$.

\textbf{Step III.3:} Assuming that $\uu_1$ and $\pi(\uu)$ are either both long or both short, we need to show that $\ff(\xx_0;\pi(\uu),\uu_1)=\tilde{\ww}(\xx_0;\uu)\cc(\xx_0;\pi(\uu))$.

Suppose not. Then $\uu_2$ is non-trivial. Let $\bb_1:=\tbq(\uu_1)$. As in Case II, we can show that there is a cyclic permutation $\bb'_1$ of $\bb_1$ such that $\bb'_1\ff(\xx_0;\pi(\uu),\uu_1)$ is a string.

There are two cases.

If $\ff(\xx_0;\pi(\uu),\uu_1)$ is a proper left substring of $\tilde{\ww}(\xx_0;\uu)\cc(\xx_0;\pi(\uu))$ then there are two subcases:
\begin{itemize}
\item $\cc(\xx_0;\pi(\uu))=\zz'\ff(\xx_0;\pi(\uu),\uu_1)$ for a string $\zz'$. If $|\zz'|=0$ then our assumption gives that $|\tilde\ww(\xx_0;\uu)|>0$. Then $\rr^s_1(\pi(\uu))\zz'\bb'_1\ff(\xx_0;\pi(\uu),\uu_1)$ is a string since $\bar{\uu}=\uu_2\ch\uu_1$. Hence there is a weak bridge $\tbq(\uu_1)\xrightarrow{\uu'}\sbq(\uu)$. Since $\pi(\uu)\in\HQ$, we get that $\pi(\uu)\neq\uu'\ch\uu_1$. Hence $N(\tbq(\uu_1),\uu'\ch\uu_1)=1$. But then it is readily verified that $N(\tbq(\uu_1),\bar\uu)=1$, a contradiction to the fact that $\tilde\ww$ is band-free.

\item $\ff(\xx_0;\pi(\uu),\uu_1)=\zz'\cc(\xx_0;\pi(\uu))$ for a string $\zz'$ with $0<|\zz'|<|\tilde\ww(\xx_0;\uu)|$. Then $\fgamma(\pi(\uu))$ is the first syllable of $\zz'$. Let $\beta$ be the syllable of $\tbq(\uu_1)$ such that $\beta\ff(\xx_0;\pi(\uu),\uu_1)$ is a string. Clearly $\theta(\beta)=-\delta(\rr^s_1(\pi(\uu)))=-\theta(\fgamma(\pi(\uu)))=\theta(\bar{\bar\gamma}^s(\pi(\uu)))$ which implies that there is an abnormal weak arch bridge $\uu'$ from $\sbq(\uu)$ to $\tbq(\uu_1)$ with $|\uu'^c|>0$. Since $\theta(\beta(\uu'))=-\delta(\uu'^c)$, we get $\bub{\uu'}\uu'^c\tilde{\ww}\equiv_H\uu'^c\tilde{\ww}$ and hence $\uu_1=\uu'\ch\pi(\uu)$, a contradiction to $\uu_1\in\HQ_i(\xx_0)$.
\end{itemize}

If $\tilde{\ww}(\xx_0;\uu)\cc(\xx_0;\pi(\uu))$ is a proper left substring of $\ff(\xx_0;\pi(\uu),\uu_1)$ then to ensure domesticity no syllable of $\tilde\ww(\xx_0;\uu)\tilde\ww$ is a syllable of $\bb'_1$. On the other hand since $\bar\uu=\uu_2\ch\uu_1$, we conclude that $\bb'_1\tilde\ww(\xx_0;\uu)\tilde\ww$ is a string with at least one syllable common between $\tbq(\uu_1)$ and $\sbq(\uu)$. Then it is readily verified that $\uu_1=\uu''\ch\pi(\uu)$ for an abnormal weak bridge $\sbq(\uu)\xrightarrow{\uu''}\tbq(\uu_1)$, a contradiction to $\uu_1\in\HQ$.

\begin{example}
In the algebra $\Gamma^{(vi)}$ in Figure \ref{fundamental}, $\bb_1:=bK$, $\bb_2:=feG$ and their inverses are the only bands. Choose $\xx_0:=1_{(\vv_1,i)}$ and $\uu:=ji$ so that $\pi(\uu):=feca$ and $\pi^2(\uu):=b$. Here $\theta(\beta(\uu))=-1$, $\pi(\uu)$ is short and $\uu$ is $s^{-1}$-long with $\tilde\ww(\xx_0;\uu)=\rr^1_s(\pi(\uu))=fe$, $\rr^2_s(\pi(\uu))=ji$ and $\tilde{\ww}=ca$. It is readily verified that $\bar{\uu}:=ifeca$ is H-reduced and $\uu_1:=\bar{\uu}$ is the short uncle of $\uu$ such that $\ff(\xx_0;\pi(\uu),\uu_1)=\tilde\ww(\xx_0;\uu)\cc(\xx_0;\pi(\uu))=fecab$.
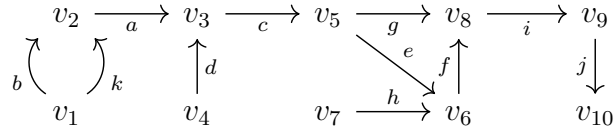
\begin{figure}[h]
    \centering
    \begin{tikzcd}
v_2 \arrow[r, "a"']                                             & v_3 \arrow[r, "c"'] & v_5 \arrow[r, "g"'] \arrow[rd, "e"] & v_8 \arrow[r, "i"'] & v_9 \arrow[d, "j"'] \\
v_1 \arrow[u, "b", bend left=49] \arrow[u, "k"', bend right=49] & v_4 \arrow[u, "d"'] & v_7 \arrow[r, "h"]                  & v_6 \arrow[u, "f"]  & v_{10}             
\end{tikzcd}
    \caption{$\Gamma^{(vi)}$ with $\rho=\{ak,cd,gc,fh,ig,jifeca\}$}
    \label{fundamental}
\end{figure}
\end{example}

\noindent{\textbf{Case IV:}} $\uu$ is $s^{-\partial}$-long.

An argument similar to Case II guarantees that $|\tilde{\ww}(\xx_0;\uu)|<|\rr^s_1(\pi(\uu))|$.

\textbf{Step IV.1:} $\beta(\uu)\tilde{\ww}(\xx_0;\uu)\tilde{\ww}$ is H-reduced.

First we show that $\tilde{\ww}(\xx_0;\uu)\tilde{\ww}$ is band-free. Recall that $\tilde{\ww}$ is a left substring of $\pi(\uu)$ and hence H-reduced by \cite[Remark~5.6]{SK}. If $\tilde{\ww}(\xx_0;\uu)\tilde{\ww}$ is not band-free then $|\tilde\ww||\tilde{\ww}(\xx_0;\uu)|>0$. In this case there is a cyclic permutation $\bar{\bb}'_1$ of a band $\bar{\bb}_1$ such that $\tilde{\ww}(\xx_0;\uu)\tilde{\ww}=\yy'_2\bar{\bb}'_1\yy'_1$. Since $\delta(\tilde\ww(\xx_0;\uu)\tilde\rr^s(\pi(\uu)))\neq0$ and $\delta(\fgamma(\pi(\uu))\bar{\bar\gamma}^s(\pi(\uu)))=0$ there is a weak bridge from $\sbq(\uu)$ to $\bar{\bb}_1$. Moreover since $\bb^\falpha(\pi(\uu))\tilde\ww$ is a string and $\tilde\ww(\xx_0;\uu)$ is a left substring of $\bb^\falpha(\pi(\uu))$ there is a weak bridge from $\bar{\bb}_1$ to $\sbq(\uu)$, a contradiction to the domesticity of the algebra. In fact the same argument also proves that $\beta(\uu)\tilde\ww(\xx_0;\uu)\tilde\ww$ is also band-free. Hence it is H-reduced.

\textbf{Step IV.2:} Lemma \ref{existencereverse(weak)archbridge} ensures that there is a $\ww'$ such that $\tbq(\bar\uu)\bar{\uu}:=\ww'\beta(\uu)\tilde{\ww}(\xx_0;\uu)\tilde{\ww}$ is either a maximal reverse arch bridge or a weak arch-bridge from $\sbq(\pi(\uu))$. Without loss we assume that $\bar{\uu}$ is weak and $\bar\uu=\uu_2\ch\uu_1$ for some $\uu_1\in\HQ_i(\xx_0)$ and possibly trivial weak arch bridge $\uu_2$.

\textbf{Step IV.3:} Assuming that $\uu_1$ and $\pi(\uu)$ are either both long or both short, we need to show that $\ff(\xx_0;\pi(\uu),\uu_1)=\tilde{\ww}(\xx_0;\uu)\cc(\xx_0;\pi(\uu))$.

Suppose not. Then $\uu_2$ is non-trivial. Let $\bb_1:=\tbq(\uu_1)$. As in Case II, we can show that there is a cyclic permutation $\bb'_1$ of $\bb_1$ such that $\bb'_1\ff(\xx_0;\pi(\uu),\uu_1)$ is a string.

There are three subcases:
\begin{itemize}
\item $\cc(\xx_0;\pi(\uu))=\zz'\ff(\xx_0;\pi(\uu),\uu_1)$ for a string $\zz'$. Then $\rr^s_1(\pi(\uu))\zz'\bb'_1\ff(\xx_0;\pi(\uu),\uu_1)$ is a string since $\bar{\uu}=\uu_2\ch\uu_1$. Hence there is a weak bridge $\tbq(\uu_1)\xrightarrow{\uu'}\sbq(\uu)$. Since $\pi(\uu)\in\HQ$, we get that $\pi(\uu)\neq\uu'\ch\uu_1$. Hence $N(\tbq(\uu_1),\uu'\ch\uu_1)=1$. But then it is readily verified that $N(\tbq(\uu_1),\bar\uu)=1$, a contradiction to the fact that $\tilde\ww$ is band-free.

\item $\ff(\xx_0;\pi(\uu),\uu_1)=\zz'\cc(\xx_0;\pi(\uu))$ for a string $\zz'$ with $0<|\zz'|<|\rr^s_1(\pi(\uu))|$ and $\zz'\neq\tilde{\ww}(\xx_0;\uu)$. Then $\fgamma(\pi(\uu))$ is the first syllable of $\zz'$. Let $\beta$ be the syllable of $\tbq(\uu_1)$ such that $\beta\ff(\xx_0;\pi(\uu),\uu_1)$ is a string. Clearly $\theta(\beta)=-\delta(\rr^s_1(\pi(\uu)))=-\theta(\fgamma(\pi(\uu)))=\theta(\bar{\bar\gamma}^s(\pi(\uu)))$ which implies that there is an abnormal weak arch bridge $\uu'$ from $\sbq(\uu)$ to $\tbq(\uu_1)$ with $|\uu'^c|>0$. Since $\theta(\beta(\uu'))=-\delta(\uu'^c)$, we get $\bub{\uu'}\uu'^c\tilde{\ww}\equiv_H\uu'^c\tilde{\ww}$ and hence $\uu_1=\uu'\ch\pi(\uu)$, a contradiction to $\uu_1\in\HQ_i(\xx_0)$. 

\item $\ff(\xx_0;\pi(\uu),\uu_1)=\zz'\cc(\xx_0;\pi(\uu))$ for a string $\zz'$ with $|\zz'|\geq|\rr^s_1(\pi(\uu))|$. To ensure domesticity we must have equality. Let $\uu''$ be the abnormal weak bridge between $\sbq(\uu)$ and $\tbq(\uu_1)$. The dual of \cite[Proposition~4.3]{SK} gives that $\sbq(\uu'')=\tbq(\uu_1)$ and that ${\uu''}^c$ is a proper right substring of $\rr^s_1(\pi(\uu))$. It is readily verified that $\pi(\uu)=\uu''\ch\uu_1$, a contradiction to $\pi(\uu)\in\HQ$.
\end{itemize}

\begin{example}\label{sopplong}
Consider the algebra $\Gamma^{(vi)}$ in Figure \ref{fundamental}. Choose $\xx_0:=1_{(\vv_1,i)}$ and $\uu:=HeG$ so that $\pi(\uu):=feca$ and $\pi^2(\uu):=b$. Here $\theta(\beta(\uu))=1$, $\pi(\uu)$ is short and $\uu$ is $s^{-1}$-long with $\tilde\ww(\xx_0;\uu)=e$, $\rr^1_s(\pi(\uu))=fe$, $\rr^2_s(\pi(\uu))=ji$ and $\tilde{\ww}=ca$. It is readily verified that $\bar{\uu}:=Heca$ is H-reduced and $\uu_1:=\bar{\uu}$ is the short uncle of $\uu$ such that $\ff(\xx_0;\pi(\uu),\uu_1)=\tilde\ww(\xx_0;\uu)\cc(\xx_0;\pi(\uu))=ecab$.   
\end{example}
Thus we have completed the proof of all cases of Lemma \ref{Forkinglocation}.

We note the significance of the hypotheses of the above lemma. If only condition $(3)$ fails then we cannot obtain the weak arch bridge $\bar\uu$ in Step 2 with $\sbq(\bar\uu)=\sbq(\pi(\uu))$. On the other hand, if only condition $(2)$ fails then we can obtain $\bar\uu\in\bHQ$ but it factors uniquely as $\bar\uu=\uu_2\ch\pi(\uu)$. Therefore the recipe in the proof of the uncle forking lemma fails to explain such a forking string. Our search leads us to the granduncle-grandchild interaction (see Lemma \ref{Forkinglocationgranduncle}).

This raises a natural question, namely ``how many distant relative pairs do we need to justify each forking string?'' Fortunately we will be able to justify that uncle-nephew and granduncle-grandchild are the only ones necessary as explained in Proposition \ref{negativegrandpapositive}.

Our next goal is to ensure the existence of an uncle for long children of elements of $\Hf_i(\xx_0)$.

\begin{rmk}\label{abnhalfbridgeupsilon}
Suppose $\uu\in\Vf_i(\xx_0)$ and $h(\uu)=2$. Then $\upsilon(\pi(\uu))=0$ if and only if $\pi(\uu)$ is an abnormal half $i$-arch bridge. When these equivalent conditions hold then $\upsilon(\uu)\geq0$ and the LVP $(\sbq(\uu),\xx_0)$ could be simultaneously $b$-long as well as $s$-long. If the LVP $(\sbq(\uu),\xx_0)$ is $(b,s)$-long (resp. $(s,b)$-long) then $|\rr_b^1(\pi(\uu))|\geq1$ (resp. $|\rr_s^1(\pi(\uu))|\geq1$), and the pair is $(b,s)^i$-long (resp. $(s,b)^i$-long).
\end{rmk}

\begin{examples}
In the algebra $\Gamma^{(iv)}$ from Figure \ref{cumulative} if $\xx_0:=1_{(\vv_4,1)}$ and $\uu:=e$ then $\pi(\uu)$ is an abnormal half $1$-arch bridge with $\cc(\xx_0;\pi(\uu))=\xx_0$ and $\cc(\xx_0;\uu)=acD$, and hence $\upsilon(\uu)>0$.

If we remove the relation $feac$ from $\rho$ while keeping $\xx_0$ and $\uu$ the same we get $\cc(\xx_0;\pi(\uu))=\xx_0=\cc(\xx_0;\uu)$, and hence $\upsilon(\uu)=0$.
\end{examples}

\begin{proposition}
Suppose $\uu\in\Vf_i(\xx_0)$, $h(\uu)=2$, $\uu$ is long, and if $\uu$ is $s$-long then $|\rr^s_2(\pi(\uu))|>1$. Then there is an uncle $\uu_1$ of $\uu$ satisfying $$\ff(\xx_0;\pi(\uu),\uu_1)=\tilde{\ww}(\xx_0;\uu)\cc(\xx_0;\pi(\uu)).$$ Furthermore,
\begin{itemize}
\item if $\pi(\uu)$ is abnormal then $\pi(\uu)\sqsf_i\uu_1$;
\item if $\pi(\uu)$ is normal then $\theta(\beta(\uu))=i$ iff $\pi(\uu)\sqsf_i\uu_1$.
\end{itemize}
\end{proposition}


Recall from Corollary \ref{longequivwithextC} that $\uu$ is long if and only if the LVP $(\sbq(\uu),\tilde\ww(\xx_0;\uu))\cc(\xx_0;\uu))$ is long. If these equivalent conditions hold then the LVP $(\sbq(\uu),\cc(\xx_0;\pi(\uu)))$ is long. In case the latter LVP is $b$-long with $|\rr^b_2(\pi(\uu))|=1$ then the proof of the uncle forking lemma still goes through where $\uu$ is an ``invisible'' long child of $\pi(\uu)$ with $\tilde\ww(\xx_0;\uu)=\rr^b_1(\pi(\uu))$ and thus explains this forking string.
\begin{proposition}\label{pseudouncle}
Suppose $\uu\in\Hf_i(\xx_0)\cup\Uf_i(\xx_0)$. If the LVP $(\bb,\yy):=(\tbq(\uu),\cc(\xx_0;\uu))$ is $b$-long with $|\rr^b_2(\uu)|=1$ then
\begin{itemize}
\item if $\uu\in\Hf_i(\xx_0)$ and if $|\rr^b_1(\uu)|>0$ whenever $\uu$ is abnormal then there exists a sibling $\uu_1$ of $\uu$ satisfying  $\uu\sqsf_i\uu_1$ iff $(\bb,\yy)$ is $b^i$-long and $$\ff(\xx_0;\uu_1,\uu)=\rr^b_1(\uu)\cc(\xx_0;\uu).$$
\item if $\uu\in\Uf_i(\xx_0)$ and there is a string $\yy$ such that $\theta(\yy\rr^b_2(\uu)\rr^b_1(\uu)\cc(\xx_0;\uu)\mid\ff(\xx_0;\pi(\uu))=\theta(\beta(\uu))$ then there exists a sibling $\uu_1$ of $\uu$ satisfying $\theta(\beta(\uu_1))=\theta(\beta(\uu))$, $\uu_1\sqsf_{\theta(\beta(\uu))}\uu$ iff $(\bb,\yy)$ is $b^{\theta(\beta(\uu))}$-long and $$\ff(\xx_0;\uu_1,\uu)=\rr^b_1(\uu)\cc(\xx_0;\uu).$$
\end{itemize}
\end{proposition}

\begin{example}
Suppose $\bar{\Gamma}^{(v)}$ is obtained by replacing the relation $kjih$ by $jih$ in $\rho$ in Figure \ref{normaluncle}. The only bands here are $\bb_1:=aCB$, $\bb_2:=ihG$ and their inverses. Choosing $\xx_0:=1_{(\vv_1,i)}$ and $\uu:=iDe$ we have $\pi(\uu)=a$, $\theta(\beta(\uu))=-1$ and $\cc(\xx_0;\uu)=Dea$. Moreover the LVP $(\tbq(\uu),\cc(\xx_0;\uu))$ is $b^{-1}$-long due to $jih\in\rho$ with $\rr^b_1(\uu)=i$ and $\rr^b_2(\uu)=j$. Here $\uu_1:=kjiDe$ is a sibling of $\uu$ such that $\ff(\xx_0;\uu_1,\uu)=iDea$ and $\uu_1\sqsf_{\theta(\beta(\uu))}\uu$.
\end{example}


The next result is essentially a mini-version of Lemma \ref{Forkinglocation}.
\begin{proposition}\label{normalintforking}
Suppose $\uu\in\Vf_i(\xx_0)$ and $\uu$ is normal if $\uu\in\Uf_i(\xx_0)\cup\Hf_i(\xx_0)$. Let $\yy$ be a proper left substring of $\uu^o$ such that $|\yy|>0$ and $\yy\ff(\xx_0;\uu,\pi(\uu))$ forks. Then the following hold:
\begin{itemize}
    \item[$h(\uu)=1$] If $\theta(\yy;\uu^o)=i$ then there is a sibling $\uu_1$ of $\uu$ such that $\uu_1\sqsf_i\uu$;
    \item[$h(\uu)>1$] If $\theta(\yy;\uu^o)=-\theta(\beta(\uu))$ then there is a sibling $\uu_1$ of $\uu$ such that $\uu_1\sqsf_{\theta(\beta(\uu))}\uu$.
\end{itemize}
In both cases we have $$\ff(\xx_0;\uu,\uu_1)=\yy\ff(\xx_0;\uu,\pi(\uu)).$$
\end{proposition}

\begin{proof}
We only prove the first statement; a similar proof works for the other case.

Let $\partial:=\theta(\beta(\uu))$. Let $\tilde\ww$ be the left substring of $\uu$ such that $\uu^o\tilde\ww$ is a string. Let $\alpha$ be the syllable with $\theta(\alpha)=\delta$ such that $\alpha\yy$ is a string. 

Since $\alpha\yy\tilde\ww$ is a string, an argument as in Step 1 of the proof of the uncle forking lemma (Lemma \ref{Forkinglocation}), we can show that $\alpha\yy\tilde\ww$ is a band-free string. Then as in Step 2 there, we extend $\alpha\yy\tilde\ww$ to $\tbq(\bar\uu)\bar{\uu}$ for some $\bar\uu\in\bHQ$ using Lemma \ref{existencereverse(weak)archbridge}. Then an argument similar to Step 3 shows that $\alpha\yy\tilde\ww\in\tbq(\uu_1)\uu_1$, where $\bar{\uu}=\uu_2\ch\uu_1$ is the canonical factorization such that $\uu_2$ is a possibly trivial weak arch bridge and $\uu_1\in\HQ$. Since $\beta(\uu_1)=\beta(\uu)$ we get $\ff(\xx_0;\uu,\uu_1)=\yy\ff(\xx_0;\uu,\pi(\uu))$ by Corollary \ref{longchildlocation} and hence $\uu_1\sqsf_\partial\uu$. In both Step 1 and 3 the proof uses the fact that if $N(\tbq(\uu_1),\bar\uu)=0$ then $\uu=\uu'\ch\uu_1$ for some $\uu'\in\bHQ$. 
\end{proof}

\begin{example}
Consider $\xx_0:=1_{(\vv_1,i)}$ and $\uu:=iDe$ in the algebra $\Gamma^{(v)}$ from Figure \ref{normaluncle} so that $\pi(\uu):=a$ and $\yy:=e$. Clearly $h(\uu)>1$, $\uu$ is normal with $\uu^o=De$ and $\yy\ff(\xx_0;\uu,\pi(\uu))=ea$ forks. Here $\uu_1:=fe$ is a sibling of $\uu$ with $\uu_1\sqsf_{\theta(\beta(\uu))}\uu$ and $\ff(\xx_0;\uu,\uu_1)=ea$.
\end{example}

Below we note without proof the consequences of the failure of condition $(3)$ of the hypotheses of Lemma \ref{Forkinglocation}.
\begin{proposition}\label{granduncleslong}
Suppose all hypotheses of Lemma \ref{Forkinglocation} hold except for $(3)$. Then the following are equivalent:
\begin{enumerate}
    \item $\uu$ is not $s^\partial$-long;
    \item $\tilde\ww(\xx_0;\uu)\cc(\xx_0;\pi(\uu))$ is a forking string;
    \item $\pi(\uu)$ is abnormal with $|\pi(\uu)^c|>0$, the LVP $(\sbq(\pi(\uu)),\cc(\xx_0;\pi^2(\uu)))$ is $b^\partial$-long with $\rr^b_2(\pi^2(\uu))=\beta(\uu)$ and one of the following holds:
    \begin{enumerate}
        \item $\upsilon(\pi(\uu))=-1$ and $|\rr^b_1(\pi^2(\uu))|=0$;
        \item $\upsilon(\pi(\uu))=1$ and $\rr^b_1(\pi^2(\uu))\cc(\xx_0;\pi^2(\uu))=\cc(\xx_0;\pi(\uu))$.
    \end{enumerate}
\end{enumerate}
When these equivalent conditions hold then $\pi(\uu)$ is short by Corollary \ref{abnpropshort} and Proposition \ref{abnequalshort}.
\end{proposition}

It is worth noting that either condition $(2)$ or $(3)$ of the hypotheses of Lemma \ref{Forkinglocation} fails if and only if exactly one of those conditions fails. When this happens, there are yet unexplained forking strings then we get the existence of a granduncle, which is the content of Lemma \ref{Forkinglocationgranduncle}. However this raises a natural question whether we can guarantee the existence of a sibling of $\pi^n(\uu)$ for arbitrary $n>0$. Since abnormality is the source of anomalies, the following result states that only $n=1,2$ are possible cases.

\begin{proposition}\label{negativegrandpapositive} 
Suppose for $\uu\in\Vf_i(\xx_0)$ with $h(\uu)\geq3$ that all hypotheses of Lemma \ref{Forkinglocation} fold except that exactly one of $(2)$ or $(3)$ fails. 

If $h(\uu)>3$ then $\ff(\xx_0;\pi^3(\uu),\pi(\uu))$ is a proper left substring of $\cc(\xx_0;\pi(\uu))$. As a consequence $\cc(\xx_0;\pi^3(\uu))$ is a proper substring of $\cc(\xx_0;\pi(\uu))$, and hence $\upsilon(\pi^2(\uu))=1$. Moreover $\pi^2(\uu)$ and $\pi(\uu)\ch\pi^2(\uu)$ are normal.

If $h(\uu)=3$ then $\ff(\xx_0;\pi^3(\uu),\pi(\uu))$ is a proper left substring of $\tilde\ww(\xx_0;\uu)\cc(\xx_0;\pi(\uu))$. 
\end{proposition}

We do not include a proof but only the first statement needs one, where if we suppose that the first conclusion fails then the impossibility can be verified in several long but straightforward cases using Propositions \ref{CopposeCharacter}, \ref{CequalCharacter}, Corollary \ref{CopposeCharactercor} and \cite[Proposition~4.3]{SK}.

\begin{lemma}(Granduncle forking lemma)\label{Forkinglocationgranduncle}
Suppose $\uu\in\Vf_i(\xx_0)$ with $h(\uu)\geq3$. Assume all the hypotheses of Lemma \ref{Forkinglocation} hold but exactly one of $(2)$ or $(3)$ fails. Furthermore assume that $\tilde\ww(\xx_0;\uu)\cc(\xx_0;\pi(\uu))$ is a forking string. Then $\pi(\uu)$ is short and there is a normal granduncle $\uu_1$ of $\uu$ satisfying $$\ff(\xx_0;\pi(\uu),\uu_1)=\ff(\xx_0;\pi^2(\uu),\uu_1)=\tilde{\ww}(\xx_0;\uu)\cc(\xx_0;\pi(\uu)).$$

Moreover
\begin{itemize}
\item[$(h(\uu)=3)$] $\uu_1\sqsf_i\pi^2(\uu)$ if and only if $\delta(\pi(\uu)^e)=i$;
\item[$(h(\uu)>3)$] $\beta(\uu_1)=\beta(\pi^2(\uu))$ and $\pi^2(\uu)\sqsf_{\theta(\beta(\pi^2(\uu)))}\uu_1$ if and only if $\theta(\beta(\pi^2(\uu)))=-\partial$.
\end{itemize}
\end{lemma}

\begin{proof}
If condition $(2)$ in the hypothesis of Lemma \ref{Forkinglocation} fails then $\tilde\ww(\xx_0;\uu)\cc(\xx_0;\pi(\uu))$ is a proper substring of $\cc(\xx_0;\pi^2(\uu))$. In particular, $\upsilon(\pi(\uu))=-1$. Then Proposition \ref{negativegrandpapositive} guarantees that $\bb^\falpha(\pi(\uu))\bar{\zz}\bb^\falpha(\pi^3(\uu))\cc(\xx_0;\pi^3(\uu))$ is a string for some string $\bar{\zz}$ of positive length. Let $\tilde\ww$ be the shortest right substring of $\bar{\zz}\bb^\falpha(\pi^3(\uu))$ such that $\bb^\falpha(\pi(\uu))\tilde\ww\sbq(\pi^2(\uu))$ is a string. Clearly $\tilde\ww$ is a left substring of $\pi^2(\uu)$. In fact $\tilde\ww(\xx_0;\uu)\tilde\ww$ is also a left substring of $\pi^2(\uu)$. Then the hypotheses imply that $\beta(\uu)\tilde\ww(\xx_0;\uu)\tilde\ww$ is a band-free string, which can be extended using Lemma \ref{existencereverse(weak)archbridge} to a weak arch bridge $\bar\uu=\uu_2\ch\uu_1$ with $\uu_1\in\HQ$. Using Propositions \ref{CopposeCharacter}, \ref{negativegrandpapositive} and Corollaries \ref{CopposeCharactercor}, \ref{exceptionalnegative} as tools while following the idea of the proof of uncle forking lemma the proof can be completed.

If condition $(3)$ from the hypothesis of Lemma \ref{Forkinglocation} fails then Proposition \ref{granduncleslong} gives a characterization of when $\tilde\ww(\xx_0;\uu)\cc(\xx_0;\pi(\uu))$ is a forking string. Then Corollary \ref{abnpropshort} and Proposition \ref{abnequalshort} guarantee that $\pi(\uu)$ is short. The idea of the proof is a combination of Lemma \ref{Forkinglocation} and Proposition \ref{pseudouncle}. The proof of the former fails marginally as $\beta(\uu)\tilde\zz\bb^\falpha(\pi^2(\uu))$ is not a string. However $\pi^2(\uu)$ satisfies the hypotheses of the latter, and this explains the forking string $\tilde\ww(\xx_0;\uu)\cc(\xx_0;\pi(\uu))$.
\end{proof}

\begin{examples}
Suppose the algebra $\bar\Gamma^{(ii)}$ (resp. $\bar\Gamma^{(iii)}$) is obtained by replacing $dbgh$ by $dbg$ (resp. $def$ by $de$) in Figure \ref{doublelongpositive} (resp. Figure \ref{doublelongnegative}). Choose $\xx_0:=1_{(\vv_1,i)}$ and $\uu:=HGfE$ (resp. $\uu:=FEahG$) so that $\pi(\uu):=cBD$ (resp. $cD$) and $\pi^2(\uu):=icBa$ (resp. $icaB$). Clearly $\upsilon(\uu)=1$ (resp. $\upsilon(\uu)=-1$), $\uu$ is $b^{-1}$-long but not $s$-long and all the hypotheses of Lemma \ref{Forkinglocation} except $(3)$ hold. It is readily verified that $\uu$ satisfies the equivalent conclusions of Proposition \ref{granduncleslong}. Moreover there is a grand-uncle $\uu_1:=HGBa$ (resp. $FEaB$) of $\uu$ such that $\ff(\xx_0;\uu,\uu_1)=\ff(\xx_0;\pi(\uu),\uu_1)=\ff(\xx_0;\pi^2(\uu),\uu_1)=Ba$ (resp. $aB$). 

In the algebra $\Gamma^{(iii)}$ in Figure \ref{doublelongnegative} choosing $\xx_0:=1_{(\vv_1,i)}$ and $\uu:=JhG$ we have $\pi(\uu)=cD$ and $\pi^2(\uu)=icaB$ with $\upsilon(\uu)=-1$. Here $\uu$ is $b^{-1}$-long and all the hypotheses of Lemma \ref{Forkinglocation} except $(2)$ hold. It is easy to check that there is a grand-uncle $\uu_1:=JB$ of $\uu$ such that $\ff(\xx_0;\uu,\uu_1)=\ff(\xx_0;\pi(\uu),\uu_1)=\ff(\xx_0;\pi^2(\uu),\uu_1)=\tilde{\ww}(\xx_0;\uu)\cc(\xx_0;\pi(\uu))=B$.
\end{examples}

There is yet another variation of Lemma \ref{Forkinglocation} where $\theta(\beta(\pi(\uu)))=-\theta(\beta(\uu))$. The proof techniques are similar, and thus the very long proof is omitted. Before we state the variation, we give a supporting result, again without proof.
\begin{proposition}\label{characteruncleopparity}
Suppose $\uu\in\Vf_i(\xx_0)$, $h(\uu)>2$, $\theta(\beta(\pi(\uu)))=\partial=-\theta(\beta(\uu))$ for some $\partial\in\{+,-\}$ and $\uu$ is long.

If $\pi(\uu)$ is abnormal with $|\pi(\uu)^c|>0$ then $\uu$ is $b^{-\partial}$-long, $\beta(\uu)$ is the first syllable of $\pi(\uu)^\beta$ and $\tilde\ww(\xx_0;\uu)=\rr^b_1(\pi(\uu))$.

If $\pi(\uu)$ is abnormal with $|\pi(\uu)^c|=0$ then one of the following happens:
\begin{itemize}
    \item If $\uu$ is $b^{-\partial}$-long, $\beta(\uu)$ is the first syllable of $\pi(\uu)^\beta$ and $\tilde\ww(\xx_0;\uu)=\rr^b_1(\pi(\uu))$;
    \item If $\uu$ is $s^\partial$-long then $0<|\tilde\ww(\xx_0;\uu)|<|\rr^s_1(\pi(\uu))|$.
\end{itemize}

Irrespective of whether $\pi(\uu)$ is normal or abnormal, $\tilde\ww(\xx_0;\uu)\cc(\xx_0;\pi(\uu))=\tilde\zz\cc(\xx_0;\pi^2(\uu))$ for some string $\tilde\zz$ and $\beta(\uu)\tilde\zz\bb^\falpha(\pi^2(\uu))$ is a string.
\end{proposition}

\begin{lemma}(Uncle forking lemma--opposite parity)\label{Forkinglocationopposite}
Suppose $\uu\in\Vf_i(\xx_0)$, $h(\uu)>2$, $\theta(\beta(\pi(\uu)))=\partial=-\theta(\beta(\uu))$ for some $\partial\in\{+,-\}$ and $\uu$ is long. Further suppose that if $\pi(\uu)$ is abnormal and $\uu$ is $b^{-\partial}$-long then $\tbq(\uu)\uu$ is not a substring of $\pi(\uu)^\beta\blb{\pi(\uu)}$. Then there is an uncle $\uu_1$ of $\uu$ constructed as in Lemma \ref{Forkinglocation} that satisfies $\theta(\beta(\uu_1))=\theta(\beta(\pi(\uu)))$. Such an uncle is short if $\upsilon(\pi(\uu))<1$. Moreover one of the following holds:
\begin{enumerate}
\item If either of the following sets of conditions hold:
\begin{itemize}
    \item $\pi(\uu)$ is normal;
    \item $\pi(\uu)$ is abnormal with $|\pi(\uu)^c|=0$, $\uu$ is $s^\partial$-long and $\beta(\uu)$ is not the first syllable of $\pi(\uu)^\beta$,
\end{itemize}
then $\beta(\uu_1)=\beta(\pi(\uu))$, and hence $\pi(\uu)\sqsf_\partial\uu_1$ and
$$\ff(\xx_0;\pi(\uu),\uu_1)=\tilde{\ww}(\xx_0;\uu)\cc(\xx_0;\pi(\uu)).$$
\item If $\pi(\uu)$ is abnormal and $\beta(\uu)$ is the first syllable of $\pi(\uu)^\beta$ then the hypothesis ensures that there is a partition $\pi(\uu)^\beta=\yy_2\yy_1$ with $|\yy_1||\yy_2|>0$ such that $\beta(\uu_1)\yy_1$ is a string and $\beta(\uu_1)$ is not the first syllable of $\yy_2$. Moreover $$\ff(\xx_0;\pi(\uu),\uu_1)=\begin{cases}\HRed{\sbq(\pi(\uu))}(\yy_1\tilde{\ww}(\xx_0;\uu)\cc(\xx_0;\pi(\uu)))&\mbox{if }\uu_1\sqsf_\partial\pi(\uu);\\\tilde\ww(\xx_0;\pi(\uu))\cc(\xx_0;\pi^2(\uu))&\mbox{if }\pi(\uu)\sqsf_\partial\uu_1.\end{cases}$$
\end{enumerate}
\end{lemma}

Since we omit the proof of the above lemma we indicate some examples below to illustrate some cases of its proof.
\begin{examples}
\begin{enumerate}
\item In the algebra $\Gamma^{(vii)}$ in Figure \ref{C-positiveopplong} the only bands are $\bb_1:=hbdC$, $\bb_2:=dfE$ and their inverses. Choosing $\xx_0:=1_{(\vv_1,i)}$ and $\uu:=Gb$ we get $\pi(\uu)=dC$ and $\pi^2(\uu)=hA$. Clearly $\upsilon(\uu)=1$, $\theta(\beta(\pi(\uu)))=-\theta(\beta(\uu))=1$, $\yy_1=b$ and $\uu$ is $b^{-1}$-long. Moreover both $\pi(\uu)$ and the uncle $\uu_1:=GbdC$ of $\uu$ are short, $\uu_1\sqsf_{1}\pi(\uu)$ and $\ff(\xx_0;\pi(\uu),\uu_1)=\HRed{\sbq(\pi(\uu))}(\yy_1\tilde{\ww}(\xx_0;\uu)\cc(\xx_0;\pi(\uu)))=A$.

\begin{figure}[h]
    \centering
    \begin{tikzcd}
v_7 \arrow[r, "g"'] & v_2 \arrow[r, "h"'] \arrow[ld, "a"'] & v_3                                \\
v_1                 & v_5 \arrow[u, "b"]                   & v_4 \arrow[u, "c"'] \arrow[l, "d"] \\
                    & v_6 \arrow[u, "e"] \arrow[ru, "f"']  &                                   
\end{tikzcd}
    \caption{$\Gamma^{(vii)}$ with $\rho=\{be,ab,hg,cf,hbdf\}$}
    \label{C-positiveopplong}
\end{figure}
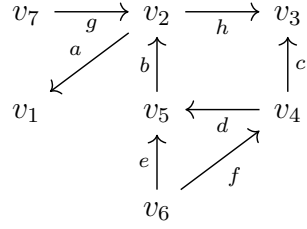

However if we add $ag$ to $\rho$ for the same quiver then choosing the same $\xx_0$ and $\uu$ we still have that $\pi(\uu)$ is short but its uncle $\uu_1:=GbdC$ is $s$-long, $\pi(\uu)\sqsf_{1}\uu_1$ and $\ff(\xx_0;\pi(\uu),\uu_1)=\tilde\ww(\xx_0;\pi(\uu))\cc(\xx_0;\pi^2(\uu))=dChA$.

\item Consider the algebra $\Gamma^{(i)}$ in Figure \ref{C-equalopplong}. Choosing $\xx_0:=1_{(\vv_1,i)}$ and $\uu:=Gb$ we get $\pi(\uu)=dC$ and $\pi^2(\uu)=hbA$. Clearly $\upsilon(\uu)=1$, $\theta(\beta(\pi(\uu)))=-\theta(\beta(\uu))=1$, $\yy_1=b$ and $\uu$ is $b^{-1}$-long. Moreover $\pi(\uu)$ is $s$-long, the uncle $\uu_1:=GbdC$ of $\uu$ is short, $\uu_1\sqsf_{1}\pi(\uu)$ and $\ff(\xx_0;\pi(\uu),\uu_1)=\HRed{\sbq(\pi(\uu))}(\yy_1\tilde{\ww}(\xx_0;\uu)\cc(\xx_0;\pi(\uu)))=bA$.

If we remove $ae$ from $\rho$ for the same quiver then choosing the same $\xx_0$ and $\uu$ we have $\upsilon(\uu)=0$, both $\pi(\uu)$ and $\uu_1:=GbdC$ are short, $\pi(\uu)\sqsf_{1}\uu_1$ and $\ff(\xx_0;\pi(\uu),\uu_1)=\tilde\ww(\xx_0;\pi(\uu))\cc(\xx_0;\pi^2(\uu))=A$.

\item In the algebra $\Gamma^{(viii)}$ in Figure \ref{C-positivezeropplong} the only bands are $\bb_1:=bED$, $\bb_2:=fcH$ and their inverses. Choosing $\xx_0:=1_{(\vv_1,i)}$ and $\uu:=jDH$ we get $\pi(\uu)=fc$ and $\pi^2(\uu)=ba$. Clearly $\upsilon(\uu)=1$, $\theta(\beta(\pi(\uu)))=-\theta(\beta(\uu))=-1$ and $\uu$ is $b^{1}$-long. Moreover both $\pi(\uu)$ and the uncle $\uu_1:=jD$ of $\uu$ are short, $\pi(\uu)\sqsf_{1}\uu_1$ and $\ff(\xx_0;\pi(\uu),\uu_1)=\tilde\ww(\xx_0;\pi(\uu))\cc(\xx_0;\pi^2(\uu))=ba$.

\begin{figure}[h]
    \centering
    \begin{tikzcd}
v_1 \arrow[r, "a"'] & v_2 \arrow[r, "b"'] \arrow[d, "e"]   & v_3 \arrow[r] \arrow[r, "c"'] \arrow[rd, "h"'] & v_4 \arrow[d, "f"'] & v_5 \arrow[l, "g"] \\
v_9                 & v_8 \arrow[ru, "d"'] \arrow[l, "j"'] &                                                & v_6 \arrow[r, "i"]  & v_7               
\end{tikzcd}
    \caption{$\Gamma^{(viii)}$ with $\rho=\{ea,hb,cd,fg,ih,je,hde\}$}
    \label{C-positivezeropplong}
\end{figure}
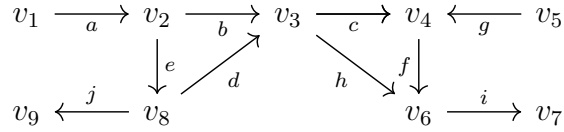

However if we add $ifcba$ to $\rho$ for the same quiver then choosing the same $\xx_0$ and $\uu$ we get that $\yy_1=D$, $\pi(\uu)$ is $s$-long and the uncle $\uu_1:=jD$ of $\uu$ is short, $\uu_1\sqsf_{1}\pi(\uu)$ and $\ff(\xx_0;\pi(\uu),\uu_1)=\HRed{\sbq(\pi(\uu))}(\yy_1\tilde{\ww}(\xx_0;\uu)\cc(\xx_0;\pi(\uu)))=Dba$.

\item  In the algebra $\Gamma^{(vi)}$ in Figure \ref{fundamental} choosing $\xx_0$, $\uu$ same as in Example \ref{sopplong} we get $\upsilon(\uu)=1$, $\theta(\beta(\pi(\uu)))=-\theta(\beta(\uu))=-1$ and both $\pi(\uu)$ and the uncle $\uu_1$ of $\uu$ are short. Moreover $\ff(\xx_0;\pi(\uu),\uu_1)=\tilde\ww(\xx_0;\uu)\cc(\xx_0;\pi(\uu))=ecab$.
\end{enumerate}
\end{examples}

\begin{rmk}
The last conclusion of Proposition \ref{characteruncleopparity} guarantees that there is no variation of the granduncle forking lemma when $\theta(\beta(\pi(\uu)))=-\theta(\beta(\uu))$.
\end{rmk}

\section{Building H-reduced strings}\label{secbuild}
Associate to each $\uu\in\Vf_i(\xx_0)$ and $j\in\{1,-1\}$ the maximal path $$\mathcal P_j(\uu):=(\xx_0,\uu_0,\uu_1,\hdots,\uu_k=\uu,\uu_{k+1},\hdots,\uu_{k[j]})$$ in $\Tf_i(\xx_0)$ such that $\phi(\uu_{p+1})=-j$ and $\uu_{p+1}^{f-}$ does not exist for each $k\leq p<k[j]$. In particular, if $\uu_k\in\RR^f_i(\xx_0)\cup\ZZ^f_i(\xx_0)$ then $\uu_{k[\pm 1]}=\uu_k$.

Set $\uu_{\mathrm{min}}:=\fmin_i(\xx_0)_{[i]}$ and $\uu_{\mathrm{max}}:=\fmax_i(\xx_0)_{[-i]}$.

\begin{rmk}\label{intermediatepoints}
Suppose $\uu\in\Vf_i(\xx_0)$ and $j\in\{1,-1\}$. If $\uu'\in\mathcal P_j(\uu)$ and $h(\uu')>h(\uu)$ then $\mathcal P_j(\uu)=\mathcal P_j(\uu')$, and hence $\uu_{[j]}=\uu'_{[j]}$. 
\end{rmk}

\begin{proposition}\label{parentchildcomparison}
Suppose $\uu\in\Vf_i(\xx_0)$ and $h(\uu)\geq2$. Then $\ff(\xx_0;\pi(\uu),\pi^2(\uu))$ is a proper left substring of $\ff(\xx_0;\uu,\pi(\uu))$.
\end{proposition}
\begin{proof}
Suppose $h(\uu)=2$. Then $\pi^2(\uu)=\xx_0$ and hence $\ff(\xx_0;\pi(\uu),\pi^2(\uu))=\xx_0$. Let $\alpha_i$ be the syllable with $\theta(\alpha_i)=i$ such that $\alpha_i\xx_0$ is a string. Since $\pi(\uu)\in\Hf_i(\xx_0)$, $\alpha_i$ is a syllable of $\ff(\xx_0;\pi(\uu))$. Moreover $\uu\ch\pi(\uu)$ is a weak half $i$-arch bridge and hence $\alpha_i$ is also a syllable of $\ff(\xx_0;\uu)$ and hence the conclusion.

Suppose $h(\uu)>2$. Both $\ff(\xx_0;\pi(\uu),\pi^2(\uu))$ and $\ff(\xx_0;\uu,\pi(\uu))$ are left substrings of $\ff(\xx_0;\pi(\uu))$ and hence they are comparable. If $\ff(\xx_0;\pi(\uu),\pi^2(\uu))$ is a proper substring of $\cc(\xx_0;\pi(\uu))$ then the conclusion is obvious. Thus we assume that $\ff(\xx_0;\pi(\uu),\pi^2(\uu))=\zz\cc(\xx_0;\pi(\uu))$ for some string $\zz$. Then Proposition \ref{Cequivnormality} gives that $\pi(\uu)$ is abnormal. Then Corollary \ref{abnrfrelation} gives that the LVP $(\tbq(\pi(\uu)),\cc(\xx_0,\pi(\uu)))$ is $b$-long and hence $\rr^b_1(\pi(\uu))$ is a right substring of $\ff(\xx_0;\pi(\uu),\pi^2(\uu))$. 

If $\uu$ is short then $\rr^b_1(\pi(\uu))$ is a proper left substring of $\tilde{\ww}(\xx_0;\uu)$. On the other hand if $\uu$ is long then $\ww(\xx_0;\uu)$ contains a cyclic permutation of $\bb^\falpha(\pi(\uu))$ as a left substring. Therefore the conclusion follows in both the cases.
\end{proof}

Propositions \ref{parentchildcomparison} and \ref{fisinjective} together guarantee, for each $\uu\in\Vf_i(\xx_0)\setminus\{\xx_0\}$ with $\uu=\fmin_\partial(\pi(\uu))$ for some $\partial\in\{+,-\}$, the existence of a string $\qq_0(\xx_0;\uu)$ of positive length satisfying
$$\qq_0(\xx_0;\uu)\ff(\xx_0;\uu,\pi(\uu)):=\begin{cases} \ff(\xx_0;\uu,\fmin
_{\phi(\uu)}(\uu))&\mbox{if }\xif_{\phi(\uu)}(\uu)\neq\emptyset; \\\ff(\xx_0;\uu)&\mbox{if }\xif_{\phi(\uu)}(\uu)=\emptyset.\end{cases}$$

\begin{rmk}\label{wr1comparison}
Suppose $\uu\in\Vf_i(\xx_0)$, $h(\uu)>2$ and $\pi(\uu)$ is abnormal. In view of Proposition \ref{parentchildcomparison} there is a positive length string $\bar\zz$ such that $\ff(\xx_0;\uu,\pi(\uu))=\bar\zz\ff(\xx_0;\pi(\uu),\pi^2(\uu))$. Since $\ff(\xx_0;\pi(\uu),\pi^2(\uu))=\rr^b_1(\pi(\uu))\cc(\xx_0;\pi(\uu))$ and $\ff(\xx_0;\uu,\pi(\uu))=\ww(\xx_0,\uu)\cc(\xx_0;\pi(\uu))$, we get $\ww(\xx_0;\uu)=\bar{\zz}\rr^b_1(\pi(\uu))$.
\end{rmk}

\begin{rmk}\label{granduncleindminmin}
Suppose $\uu\in\Vf_i(\xx_0)$ with $h(\uu)>2$ satisfies the hypotheses of the granduncle forking lemma (Lemma \ref{Forkinglocationgranduncle}). If $\pi(\uu)=\fmin_\partial(\pi^2(\uu))$ and $\uu=\fmin_\partial(\pi(\uu))$ then $\tilde\ww(\xx_0;\uu)\cc(\xx_0;\pi(\uu))$ is a proper left substring of $\ff(\xx_0;\pi(\uu),\pi^2(\uu))$.  
\end{rmk}

\begin{rmk}
Suppose $\uu\in\Vf_i(\xx_0),\ \uu\neq\xx_0$ and $\xif_\partial(\uu)\neq\emptyset$. If $\uu\neq\fmax_{\theta(\beta(\uu))}(\pi(\uu))$ and $\beta(\uu)\neq\beta(\uu^{f+})$ then $\ff(\xx_0;\uu,\uu^{f+})=\ff(\xx_0;\uu,\pi(\uu))$.
\end{rmk}

\begin{proposition}\label{siblingcomparison}
Suppose $\uu\in\Vf_i(\xx_0)$ and $h(\uu)>2$. If $\theta(\beta(\pi(\uu)))=\partial$, $\uu=\fmin_{\partial}(\pi(\uu))$ and $\pi(\uu)^{f-}$ exists then $\ff(\xx_0;\pi(\uu),\pi(\uu)^{f-})$ is a proper left substring of $\ff(\xx_0;\pi(\uu),\uu)$.
\end{proposition}
\begin{proof}
Both $\ff(\xx_0;\pi(\uu),\pi(\uu)^{f-})$ and $\ff(\xx_0;\pi(\uu),\uu)$ are left substring of $\ff(\xx_0;\pi(\uu))$ and hence are comparable. Corollary \ref{Comparableparentchildc} states that $\cc(\xx_0;\pi(\uu))$ is a left substring of $\hh(\xx_0;\Pp(\uu))$ and hence $\cc(\xx_0;\pi(\uu))$ is a substring of $\ff(\xx_0;\pi(\uu),\uu)$. The conclusion is clear if $\ff(\xx_0;\pi(\uu),\pi(\uu)^{f-})$ is a proper left substring of $\cc(\xx_0;\pi(\uu))$. Hence we assume otherwise.

If the LVP $(\bb,\yy):=(\tbq(\pi(\uu)),\ff(\xx_0;\pi(\uu),\pi(\uu)^{f-}))$ is short then it can be readily seen that for some $\uu'\in\bHQ$ we have $\pi(\uu)^{f-}=\uu'\ch\pi(\uu)$, a contradiction to $\pi(\uu)^{f-}\in\HQ$. Hence the LVP $(\bb,\yy)$ is long. Since $\theta(\ff(\xx_0;\pi(\uu)^{f-}\mid\pi(\uu)))=\partial$, the LVP $(\bb,\yy)$ can not be $(s,b)^{-\partial}$-long and $(b,s)^\partial$-long.

Now if $\ff(\xx_0;\pi(\uu),\pi(\uu)^{f-})=\zz\cc(\xx_0;\pi(\uu))$ then $\zz$ is a proper left substring of $\bb^{\falpha}(\pi(\uu))$. If $\uu$ is long then $\ff(\xx_0;\pi(\uu),\uu)=\tilde{\ww}(\xx_0;\uu)(\pi(\uu))\bb^\falpha{(\pi(\uu))}\cc(\xx_0;\pi(\uu))$ and hence the conclusion follows. Thus it remains to consider the case when $\uu$ is short.

There are four cases:
\begin{itemize}
\item $(\bb,\yy)$ is $b^\partial$-long. Here $\zz=\rr^b_1(\pi(\uu))$ but since $\uu$ is short, $\rr^b_1(\pi(\uu))$ is a proper left substring of $\tilde{\ww}(\xx_0;\uu)$.

\item $(\bb,\yy)$ is $b^{-\partial}$-long or $(b,s)^{-\partial}$-long. Since $\uu$ is short, $\rr^b_1(\pi(\uu))$ is a proper left substring of $\tilde{\ww}(\xx_0;\uu)$ while $\zz$ is a proper left substring of $\rr^b_1(\pi(\uu))$.

\item $(\bb,\yy)$ is $s^\partial$-long. If $|\rr^s_2(\pi(\uu))|=1$ then $\ff(\xx_0;\pi(\uu)^{f-})$ is a left substring of $\ff(\xx_0;\pi(\uu))$ and hence $\pi(\uu)=\pi(\uu)^{f-}$ by Proposition \ref{fisinjective}, which is a contradiction. Hence $|\rr^s_2(\pi(\uu))|>1$ and $\zz=\rr^s_1(\pi(\uu))$ while $\rr^s_1(\pi(\uu))$ is a proper substring of $\tilde{\ww}(\xx_0;\uu)$ since $\uu$ is short.

\item $(\bb,\yy)$ is $s^{-\partial}$-long or $(s,b)^\partial$-long. The proof is similar to the second case above after replacing $\rr^b_1$ by $\rr^s_1$.
\end{itemize}
This completes the proof.
\end{proof}

Propositions \ref{siblingcomparison} and \ref{fisinjective} together guarantee, for each $\uu\in\Vf_i(\xx_0)\setminus\{\xx_0\}$ for which $\uu^{f-}$ exists, the existence of a string $\qq_{\phi(\uu)}(\xx_0;\uu)$ of positive length satisfying
$$\qq_{\phi(\uu)}(\xx_0;\uu)\ff(\xx_0;\uu,\uu^{f-}):=\begin{cases} \ff(\xx_0;\uu,\fmin
_{\phi(\uu)}(\uu))&\mbox{if }\xif_{\phi(\uu)}(\uu)\neq\emptyset; \\\ff(\xx_0;\uu)&\mbox{if }\xif_{\phi(\uu)}(\uu)=\emptyset.\end{cases}$$

Using similar arguments as in the proofs in Propositions \ref{parentchildcomparison} and \ref{siblingcomparison} we get the following result.
\begin{proposition}\label{siblingcomparisonsuccessor}
Suppose $\uu\in\Vf_i(\xx_0)$ and $h(\uu)\geq2$. If $\uu=\fmin_{-\theta(\beta(\pi(\uu)))}(\pi(\uu))$ and $\pi(\uu)^{f+}$ exists then $\ff(\xx_0;\pi(\uu),\pi(\uu)^{f+})$ is a proper left substring of $\ff(\xx_0;\pi(\uu),\uu)$.
\end{proposition}

Propositions \ref{siblingcomparisonsuccessor} and \ref{fisinjective} together guarantee, for each $\uu\in\Vf_i(\xx_0)\setminus\{\xx_0\}$ for which $\uu^{f+}$ exists, the existence of a string $\qq_{-\phi(\uu)}(\xx_0;\uu)$ of positive length satisfying
$$\qq_{-\phi(\uu)}(\xx_0;\uu)\ff(\xx_0;\uu,\uu^{f+}):=\begin{cases} \ff(\xx_0;\uu,\fmin
_{-\phi(\uu)}(\uu))&\mbox{if }\xif_{-\phi(\uu)}(\uu)\neq\emptyset; \\\ff(\xx_0;\uu)&\mbox{if }\xif_{-\phi(\uu)}(\uu)=\emptyset.\end{cases}$$

Combining Propositions \ref{parentchildcomparison}, \ref{siblingcomparison} and \ref{siblingcomparisonsuccessor} we get the following.

\begin{corollary}\label{breakdownHreduced}
Suppose $|\Vf_i(\xx_0)|>1$, $\uu\in\Vf_i(\xx_0)$ and $h(\uu)\geq1$. Then
\begin{itemize}
    \item $\ff(\xx_0;\uu_{\mathrm{min}})=\qq_0(\xx_0;\uu_{\mathrm{min}})\qq_0(\xx_0;\pi(\uu_{\mathrm{min}}))\hdots\qq_0(\xx_0;\pi^{h(\uu_{\mathrm{min}})-1}(\uu_{\mathrm{min}}))\xx_0$;
    \item $\ff(\xx_0;\uu_{[-\phi(\uu)]}\mid\uu^{f-})=\qq_0(\xx_0;\uu_{[-\phi(\uu)]})\hdots\qq_0(\xx_0;\fmin_{\phi(\uu)}(\uu))\qq_{\phi(\uu)}(\xx_0;\uu)\ff(\xx_0;\uu,\uu^{f-})$\\ if $\uu^{f-}$ exists;
    \item $\ff(\xx_0;\uu_{[\phi(\uu)]}\mid\uu^{f+})=\qq_0(\xx_0;\uu_{[\phi(\uu)]})\hdots\qq_0(\xx_0;\fmin_{-\phi(\uu)}(\uu))\qq_{-\phi(\uu)}(\xx_0;\uu)\ff(\xx_0;\uu,\uu^{f+})$\\ if $\uu^{f+}$ exists.
\end{itemize}
\end{corollary}

For each $n\geq0$, let $\Vf_i(\xx_0)_n:=\{\uu\in\Vf_i(\xx_0)\mid h(\uu)=n\}$. Note that one can identify $\uu\in\Vf_i(\xx_0)_n$ with the path $(\xx_0,\pi^{n-1}(\uu),\hdots,\uu)$.
\begin{definition}\label{newchildren}
Let $$\bVf_i(\xx_0):=\{(\uu,(s_1,s_2,\hdots,s_{h(\uu)-1}))\mid\uu\in\Vf_i(\xx_0),(s_1,s_2,\hdots,s_{h(\uu)-1})\in\N^{h(\uu)-1}\}.$$

Define $\ff(\xx_0;(\uu,(s_1,s_2,\hdots,s_{h(\uu)-1})))$ by the identity $$\Red{\sbq(\uu)}^{s_{h(\uu)-1}}\hdots\Red{\sbq(\pi^{h(\uu)-2}(\uu))}^{s_1}\ff(\xx_0;(\uu,(s_1,s_2,\hdots,s_{h(\uu)-1})))=\ff(\xx_0;\uu).$$
Define $\pi((\uu,(s_1,\hdots,s_{h(\uu)-1}))):=(\pi(\uu),(s_1,\hdots,s_{h(\uu)-2})),\phi((\uu,(s_1,\hdots,s_{h(\uu)-1}))):=\phi(\uu)$.

Given distinct $(\uu,(s_1,s_2,\hdots,s_{h(\uu)-2}),s),(\uu',(s_1,s_2,\hdots,s_{h(\uu)-2},t))$ with $\pi(\uu)=\pi(\uu')$ and $\uu,\uu'\in\xif_\partial(\pi(\uu))$, say that $(\uu,(s_1,s_2,\hdots,s_{h(\uu)-2},s))\sqsf_\partial(\uu',(s_1,s_2,\hdots,s_{h(\uu)-2},t))$ if either $\uu\sqsf_\partial\uu'$ or $\uu=\uu'$ but $s<t$. Thus we have completed the description of a decorated tree $\bTf_i(\xx_0)$ whose vertex set is $\bVf_i(\xx_0)$.
\end{definition}
There is a natural embedding $\Vf_i(\xx_0)\to\bVf_i(\xx_0)$ defined by $\uu\mapsto(\uu,(0,0,\hdots,0))$. Note that $\bar{\uu}^{f+}$ exists for each $\uu\in\bVf_i(\xx_0)\setminus\{\xx_0,\fmax_i(\xx_0)\}$.


Equip $\Vf_i(\xx_0)$ with \emph{inorder}, denoted $\prec_i$, where for $\uu\in\Vf_i(\xx_0)\setminus\{\xx_0\}$ we have
\begin{itemize}
    \item $\xx_0\prec_i\uu$;
    \item $\uu'\prec_i\uu''$ if $\uu'\sqsf_i\uu''$ in $\xif_i(\xx_0)$;
    \item $\uu'\prec_i\uu$ if $\uu'\in\xif_{-i}(\uu)$;
    \item $\uu\prec_i\uu'$ if $\uu'\in\xif_i(\uu)$;
    \item $\uu'\prec_i\uu''$ if $\uu'\sqsf_{-i}\uu''$ in $\xif_{-i}(\uu)$ for some $\uu\neq\xx_0$;
    \item $\uu''\prec_i\uu'$ if $\uu'\sqsf_i\uu''$ in $\xif_i(\uu)$ for some $\uu\neq\xx_0$.
\end{itemize}
Then $(\Vf_i(\xx_0),\prec_i)$ is a total order with $\xx_0$ as the minimal element, $\uu_{\mathrm{min}}$ its successor and $\uu_{\mathrm{max}}$ as the maximal element.
\begin{rmk}
Suppose $\uu_1,\uu_2\in\Vf_i(\xx_0)$. Then $\uu_1\prec_i\uu_2$ if and only if one of the following happens:
\begin{itemize}
    \item $i=1$ and $\ff(\xx_0;\uu_1)<_l\ff(\xx_0;\uu_2)$ in $\Hamm_l(\xx_0)$;
    \item $i=-1$ and $\ff(\xx_0;\uu_2)<_l\ff(\xx_0;\uu_1)$ in $\Hamm_l(\xx_0)$.
\end{itemize}
\end{rmk}

In fact the set $\bVf_i(\xx_0)$ can also be equipped with inorder in a similar manner.

\section{Signature types of H-reduced strings}\label{signtype}

\begin{proposition}\label{sigtypeparentchild}
Suppose $\uu\in\Vf_i(\xx_0)$, $h(\uu)>1$ and $\uu=\fmin_\partial(\pi(\uu))$. Then $$\Theta(\qq_0(\xx_0;\uu))=C(\partial).$$

\end{proposition}
\begin{proof}
We only prove the result when $\xif_\partial(\uu)\neq\emptyset$, and indicate the changes for the other case. Recall from Proposition \ref{parentchildcomparison} that $|\qq_0(\xx_0;\uu)|>0$ and $\ff(\xx_0;\uu,\fmin_\partial(\uu))=\qq_0(\xx_0;\uu)\ff(\xx_0;\uu,\pi(\uu))$. Consider the following cases:

\textbf{Case 1:} $\uu$ is normal.

Then Proposition \ref{Cequivnormality} gives that $\cc(\xx_0;\uu)=\zz\ff(\xx_0;\uu,\pi(\uu))$ for some string $\zz$ of positive length. Note that $\cc(\xx_0;\uu)=\uu^o\ww(\xx_0;\uu)\cc(\xx_0;\pi(\uu))$ and $\theta(\qq_0(\xx_0;\uu))=\partial$.

Proposition \ref{normalintforking} gives that $\Theta(\ff(\xx_0;\uu,\pi(\uu));\uu^o\ff(\xx_0;\uu,\pi(\uu)))=C(\partial)$.

If all elements of $\xif_\partial(\uu)$ are short then using $\uu=\fmin_\partial(\pi(\uu))$ Proposition \ref{pseudouncle} gives that if the LVP $(\tbq(\uu),\cc(\xx_0;\uu))$ is $b^\partial$-long then $|\rr^b_2(\uu)|>1$, and hence the conclusion is readily verified. Therefore suppose that some element of $\xif_\partial(\uu)$ is long. Since $\uu$ is normal, any long $\uu'\in\xif_\partial(\uu)$ satisfies condition $(2)$ of the hypotheses of Lemma \ref{Forkinglocation}.

\textbf{Claim:} $\tilde\ww(\xx_0;\uu')$ is not a forking string for any $\uu'\in\xif_\partial(\uu)$.

If not then since $\uu$ is normal, condition $(3)$ of Proposition \ref{granduncleslong} fails, and hence $\uu'$ satisfies condition $(3)$ of the hypotheses of Lemma \ref{Forkinglocation}. Now that all the hypotheses of that lemma are true, it ensures the existence of a sibling $\uu_1$ of $\uu$. Since $\uu=\fmin_\partial(\pi(\uu))$ we must have $\uu\sqsf_\partial\uu_1$ but this is impossible when $\uu$ is normal (see Step 3 of the proof of that lemma). Thus the claim.

It follows from the claim that all long $\uu'\in\xif_\partial(\uu)$ share the same exit syllable. Then it is readily verified that the conclusion follows.


When $\xif_\partial(\uu)=\emptyset$ we use Proposition \ref{pseudouncle} but not Lemma \ref{Forkinglocation} in the argument.

\textbf{Case 2:} $\uu$ is abnormal.

Then Proposition \ref{Cequivnormality} gives that $\ff(\xx_0;\uu,\pi(\uu))=\zz'\cc(\xx_0;\uu)$ for some string $\zz'$. Further Remark \ref{wr1comparison} gives that $\ww(\xx_0;\fmin_\partial(\uu))=\qq_0(\xx_0;\uu)\rr^b_1(\uu)$. Since $\beta(\uu)$ is the first syllable of $\qq_0(\xx_0;\uu)$, we have $\theta(\bar{\zz})=\partial$.

If $\fmin_\partial(\uu)$ is short then the conclusion is readily verified (one may need to use Proposition \ref{pseudouncle} when $|\uu^c|=0$).

On the other hand, if $\fmin_\partial(\uu)$ is long then all elements of $\xif_\partial(\uu)$ are long. Let $\bb^\falpha(\uu)=\bar\yy\rr^b_1(\uu)$. An argument similar to the above paragraph shows that $\Theta(\bar\yy)=C(\partial)$. Since $\delta(\rr^b_1(\uu)\tilde\rr^b(\uu))=-\partial$ and $\tilde\rr^b(\uu)$ is a right substring of $\bar\yy$, we can easily conclude that $\Theta(\qq_0(\xx_0;\uu))=C(\partial)$.
\end{proof}

With minor modifications in the proof above, we can also show the following two results.
\begin{proposition}\label{sigtypeparentchildh1}
$\Theta(\qq_0(\xx_0;\fmin_i(\xx_0)))=F(i)$.
\end{proposition}


\begin{proposition}\label{sigtypeparentmaxchildh1}
Suppose $\uu=\fmax_i(\xx_0)$ then the following hold.
\begin{itemize}
\item $\Theta(\xx_0;\ff(\xx_0;\uu))=C(i)$ if $\xif_{i}(\uu)=\emptyset$;
\item $\Theta(\xx_0;\ff(\xx_0;\uu,\fmin_{i}(\uu)))=C(i)$ if $\xif_{i}(\uu)\neq\emptyset$.
\end{itemize}
\end{proposition}

\begin{proposition}\label{sigtype-sibling}
Suppose $\uu\in\Vf_i(\xx_0)$ and $h(\uu)>1$. If $\partial:=\theta(\beta(\uu))$ and $\uu\neq\fmin_\partial(\pi(\uu))$ then 
$\Theta(\qq_\partial(\xx_0;\uu))=F(-\partial)$.
\end{proposition}
\begin{proof}
We only prove the result when $\xif_\partial(\uu)\neq\emptyset$; the proof of the other case is omitted as it is very similar with only minor modifications. Recall from Proposition \ref{siblingcomparison} that $|\qq_\partial(\xx_0;\uu)|>0$ and $\ff(\xx_0;\uu,\fmin_\partial(\uu))=\qq_\partial(\xx_0;\uu)\ff(\xx_0;\uu,\uu^{f-})$. Clearly $\theta(\qq_\partial(\xx_0;\uu))=-\partial$ and $\ff(\xx_0;\uu,\uu^{f-})\neq\ff(\xx_0;\uu,\pi(\uu))$. If $\Theta(\qq_\partial(\xx_0;\uu))\neq F(-\partial)$ then there is a string $\yy$ of positive length such that $\yy\ff(\xx_0;\uu,\uu^{f-})$ is a proper left forking substring of $\ff(\xx_0;\uu,\fmin_\partial(\uu))$ and $\theta(\yy\ff(\xx_0;\uu,\uu^{f-});\ff(\xx_0;\uu,\pi(\uu)))=-\partial$. Let $\alpha$ be the syllable such that $\theta(\alpha)=\partial$ and $\alpha\yy\ff(\xx_0;\uu,\uu^{f-})$ is a string.

There are two main cases and several subcases:

\noindent{\textbf{Case 1:}} $\uu$ is normal.

Here Proposition \ref{Cequivnormality} gives that $\cc(\xx_0;\uu)=\zz\ff(\xx_0;\uu,\pi(\uu))$ for some string $\zz$ of positive length. We also have $\cc(\xx_0;\uu)=\uu^o\ww(\xx_0;\uu)\cc(\xx_0;\pi(\uu))$.

\textbf{Case 1A:} $\ff(\xx_0;\uu,\pi(\uu))=\zz\ff(\xx_0;\uu,\uu^{f-})$ for some string $\zz$ of positive length.

Suppose $\yy\ff(\xx_0;\uu,\uu^{f-})$ is a proper left substring of $\ff(\xx_0;\uu,\pi(\uu))$. Since $\theta(\yy)=-\partial$ we have $\delta(\alpha\yy)=0$. Moreover since $\alpha\yy$ is a string we conclude that $\alpha\in\mathcal E(\sbq(\uu))$. Now \cite[Proposition~8.5]{GKS} guarantees that there exists a semi-bridge or a torsion reverse semi-bridge, say $\uu_1$, such that $\sbq(\uu_1)=\sbq(\uu)$ and $\beta(\uu_1)=\alpha$ while \cite[Proposition~8.7]{GKS} ensures that $\uu_1\in\HQ$. Then we have $\uu^{f-}\sqsf_\partial\uu'\sqsf_\partial\uu$, a contradiction to the definition of $\uu^{f-}$.

Proposition \ref{normalintforking} gives that if $\ff(\xx_0;\uu,\pi(\uu))$ is a proper left substring of $\yy\ff(\xx_0;\uu,\uu^{f-})$ and the latter is a substring of $\cc(\xx_0;\uu)$ then there is a sibling $\uu_1$ of $\uu$ with $\beta(\uu_1)=\beta(\uu)$ and $\ff(\xx_0;\uu,\uu_1)=\yy\ff(\xx_0;\uu,\uu^{f-})$. As a consequence we get $\uu^{f-}\sqsf_\partial\uu_1\sqsf_\partial\uu$, a contradiction to the definition of $\uu^{f-}$.

Finally an argument similar to the proof of Case 1 of Proposition \ref{sigtypeparentchild} using Lemma \ref{Forkinglocation} and Proposition \ref{pseudouncle} shows that $\cc(\xx_0;\uu)$ is not a substring of $\yy\ff(\xx_0;\uu,\uu^{f-})$.

\textbf{Case 1B:} $\ff(\xx_0;\uu,\uu^{f-})=\zz\ff(\xx_0;\uu,\pi(\uu))$ is a proper left substring of $\cc(\xx_0;\uu)$ for some string $\zz$ of positive length.

The proof of this subcase is similar to the above two paragraphs.

\textbf{Case 1C:} $\ff(\xx_0;\uu,\uu^{f-})=\zz\cc(\xx_0;\uu)$ for some string $\zz$.

Under this assumption we showed in the proof of Proposition \ref{siblingcomparison} that the LVP $(\tbq(\uu),\ff(\xx_0;\uu,\uu^{f-}))$ is long. Let $k:=\max \{|\rr^b_1(\uu)|,|\rr^s_1(\uu)|\}$. Then $|\zz|\leq k$.

\textbf{Claim:} There is no $\uu'\in\xif_\partial(\uu)$ satisfying $|\zz|<|\tilde\ww(\xx_0;\uu')|\leq k$.

Suppose not. Then the LVP $(\tbq(\uu),\ff(\xx_0;\uu,\uu^{f-}))$ is necessarily either $b^{-\partial}$-long or $s^{-\partial}$-long. Then the uncle forking lemma yields a sibling $\uu_1$ of $\uu$ satisfying $\ff(\xx_0;\uu,\uu_1)=\tilde\ww(\xx_0;\uu')\cc(\xx_0;\uu)$. The existence of such $\uu_1$ will contradict the definition of $\uu^{f-}$, and hence the claim.

Irrespective of whether $\fmin_\partial(\uu)$ is short or long, the conclusion follows using the claim.

\noindent{\textbf{Case 2:}} $\uu$ is abnormal.

Here Proposition \ref{Cequivnormality} gives that $\ff(\xx_0;\uu,\pi(\uu))=\zz\cc(\xx_0;\uu)$ for some string $\zz$.

\textbf{Case 2A:} $\cc(\xx_0;\uu)=\zz_1\ff(\xx_0;\uu,\uu^{f-})$ for some string $\zz_1$ of positive length.

Since $\cc(\xx_0;\pi(\uu))$ is a substring of $\ff(\xx_0;\uu,\uu^{f-})$, it immediately follows that $\upsilon(\uu)=1$. Corollary \ref{abnrfrelation} gives that the LVP $(\tbq(\uu),\cc(\xx_0;\uu))$ is $b^{-\partial}$-long. Furthermore Corollary \ref{blongdueabnormality} gives that $\rr^b_1(\uu)=\uu^c$. Let $\qq_\partial(\xx_0;\uu)=\zz_2\uu^c\zz_1$, where $|\zz_1||\zz_2|>0$. 

\begin{enumerate}
\item[$|\uu^c|>0$] As in the first paragraph of Case 1A above we can show that $\Theta(\zz_1)=F(-\partial)$. Since $\delta(\beta(\uu)\uu^c)=0$ we clearly get $\Theta(\zz_2)=C(\partial)$. It remains to show that $\Theta(\uu^c\zz_1)=F(-\partial)$. Let $\uu'\in\xif_\partial(\uu)$. We show that $\beta(\uu')\tilde\ww(\xx_0;\uu')\cc(\xx_0;\uu)$ is not a string.

If $\uu$ is long then Corollary \ref{abnpropshort} states that $\uu$ is necessarily $s$-long and $\ww(\xx_0;\uu)=\bb^\falpha{(\pi(\uu))}$. Since $\beta(\uu)\neq\beta(\uu^{f-})$, we can conclude that $\uu^{f-}$ is always short. Hence $\ff(\xx_0;\uu^{f-},\pi(\uu))=\tilde\ww(\xx_0;\uu^{f-})\cc(\xx_0;\pi(\uu))$.

If $|\tilde\ww(\xx_0;\uu')|>0$ then the uncle forking lemma produces a sibling $\uu_1$ of $\uu$ such that $\ff(\xx_0;\uu,\uu_1)=\tilde\ww(\xx_0;\uu')\cc(\xx_0;\uu)$. Given the assumption of Case 2A, we get $\uu^{f-}\sqsf_\partial\uu_1\sqsf_\partial\uu$, a contradiction to the definition of $\uu^{f-}$.

On the other hand if $|\tilde\ww(\xx_0;\uu')|=0$ then $\delta(\beta(\uu')\zz_1)=0$. If $\beta(\uu)\tilde\ww(\xx_0;\uu')\cc(\xx_0;\uu)$ is a string then there is some $\uu''\in\xif_\partial(\pi(\uu))$ satisfying $\beta(\uu'')=\beta(\uu'')$ and $\uu^{f-}\sqsf_{\partial}\uu''\sqsf_{\partial}\uu$, a contradiction. Thus $\beta(\uu)\tilde\ww(\xx_0;\uu')\cc(\xx_0;\uu)$ is not a string. As a consequence $\tilde\ww(\xx_0;\uu')\cc(\xx_0;\pi(\uu))$ is not a forking string, and hence we have completed the proof that $\theta(\uu^c\zz_1)=F(-\partial)$.

\item[$|\uu^c|=0$] As in the first paragraph of Case 1A above we can show that $\Theta(\zz_1)=F(-\partial)$. 

If $\fmin_\partial(\uu)$ is short then $\Theta(\zz_2)=C(\partial)$ follows from the fact $\delta(\beta(\uu)\bar{\gamma}^b(\uu))=0$. 

If $\fmin_\partial(\uu)$ is long then the LVP $(\tbq(\uu), \cc(\xx_0;\uu))$ is $s^\partial$-long with $\rr^s_1(\uu)=\tilde{\ww}(\xx_0;\fmin_\partial(\uu))$ and $|\rr^s_1(\uu)|>0$.

If $\fmin_\partial(\uu)$ satisfies the hypotheses of the uncle forking lemma then the lemma will produce a sibling $\uu_1$ of $\uu$ satisfying $\uu^{f-}\sqsf_\partial\uu_1\sqsf_\partial\uu$, a contradiction. Hence condition $(3)$ fails. Since $\delta(\beta(\uu)\tilde\ww(\xx_0;\fmin_\partial(\uu))\zz_1)=0$ we conclude that $|\rr^s_2(\uu)|=1$ and hence that $\tilde\ww(\xx_0;\fmin_\partial(\uu))\cc(\xx_0;\uu)$ is not a forking string. Then it immediately follows that $\Theta(\zz_2)=C(\partial)$.
\end{enumerate}

\textbf{Case 2B:} $\cc(\xx_0;\uu)=\ff(\xx_0;\uu,\uu^{f-})$.

If $\upsilon(\uu)=-1$ then $\bb^\falpha(\uu)=\tilde{\zz}\rr^b_1(\uu)$ for some string $\tilde{\zz}$ of positive length with $\delta(\tilde{\zz})=0$ by Corollary \ref{CopposeCharactercor}. Then it is readily verified that $\uu^{f-}=\uu''\ch\uu$ for some $\uu''$, a contradiction. Hence $\upsilon(\uu)\neq-1$. Let $\qq_\partial(\xx_0;\uu)=\zz_2\rr^b_1(\uu)$.

Note that $\beta(\uu)\neq\beta(\uu^{f-})$ for otherwise $\cc(\xx_0;\uu)$ is a proper left substring of $\ff(\xx_0;\uu,\uu^{f-})$, a contradiction to the assumption of this case. Hence $|\rr^b_1(\uu)|>0$.

An easy impossibility argument using the uncle forking lemma shows that there is no $\uu'\in\xif_\partial(\uu)$ such that $0<|\tilde\ww(\xx_0;\uu')|<|\rr^b_1(\uu)|$. Hence $\Theta(\rr^b_1(\uu))=F(-\partial)$.

As $|\rr^b_1(\uu)|>0$ we have $\delta(\beta(\uu)\rr^b_1(\uu))=0$. Hence it is readily seen that $\Theta(\zz_2)=C(\partial)$.

\textbf{Case 2C:} $\ff(\xx_0;\uu,\uu^{f-})=\zz_1\cc(\xx_0;\uu)$ is a proper left substring of $\ff(\xx_0;\uu,\pi(\uu))$ for some string $\zz_1$ of positive length. 

Since $\ff(\xx_0;\uu,\uu^{f-})$ is a proper left substring of $\ff(\xx_0;\uu,\pi(\uu))$, we have $\beta(\uu)\neq\beta(\uu^{f-})$.

If $\upsilon(\uu)=-1$ then $\cc(\xx_0;\pi(\uu))$ is a left substring of $\ff(\xx_0;\uu,\uu^{f-})$ for otherwise $\beta(\uu)=\beta(\uu^{f-})$ and hence $\uu^{f-}$ factors through $\uu$ as explained in the first paragraph of Case 2B above.

Let $\qq_\partial(\xx_0;\uu)=\zz_2\zz_1$ with $|\zz_2||\zz_1|>0$ and $\zz_1\tilde{\ww}(\xx_0;\uu^{f-})\cc(\xx_0;\pi(\uu))=\rr^b_1(\uu)\cc(\xx_0;\uu)$. The argument that $\Theta(\zz_1)=F(-\partial)$ and $\Theta(\zz_2)=C(\partial)$ is similar to the proof of Case 2B.

\textbf{Case 2D:} $\ff(\xx_0;\uu,\uu^{f-})=\zz_3\ff(\xx_0;\uu,\pi(\uu))$ for some string $\zz_3$ of positive length.

Clearly $\beta(\uu)=\beta(\uu^{f-})$, $\theta(\zz_3)=\partial$ and $|\zz_3|>0$. We also have $\delta(\zz_3)=\partial$ for otherwise $\uu^{f-}$ factors through $\uu$.

As argued in the second paragraph of the proof of Proposition \ref{siblingcomparison}, we see that the LVP $(\bb,\yy):=(\tbq(\uu),\ff(\xx_0;\uu,\uu^{f-}))$ is long.

If $\upsilon(\uu)=-1$ then Corollary \ref{abnrfrelation} gives that $\ff(\xx_0;\uu,\pi(\uu))=\uu^c\cc(\xx_0;\pi(\uu))=\rr^b_1(\uu)\cc(\xx_0;\uu)$. If $|\uu^c|>0$ then $\delta(\beta(\uu)\uu^c)=0$. These two statements together contradict the fact that $(\bb,\yy)$ is long. Therefore $|\uu^c|=0$.

Further since $(\bb,\yy)$ is long we get that $\uu$ is $b^\partial$-long with $|\rr^b_1(\pi(\uu))|=0$ and $\zz_3$ a proper left substring of $\rr^b_2(\pi(\uu))$, a contradiction by Proposition \ref{abnequalshort}. Thus $\upsilon(\uu)>-1$.

Now if $|\rr^b_1(\uu)|>0$ then $\delta(\beta(\uu)\rr^b_1(\uu))=0$. This together with the identity $\ff(\xx_0;\uu,\pi(\uu))=\rr^b_1(\uu)\cc(\xx_0;\uu)$ guaranteed by Corollary \ref{abnrfrelation} give a contradiction the fact that $(\bb,\yy)$ is long. Therefore $|\rr^b_1(\uu)|=0$. If $\upsilon(\uu)=0$ then $(\bb,\yy)$ is long gives that $\uu$ is $s$-long, again a contradiction by Proposition \ref{abnequalshort}, and hence $\upsilon(\uu)=1$.

In this case it is readily verified that $(\bb,\yy)$ is $s^\partial$-long with $\rr^s_1(\uu)\ff(\xx_0;\uu,\pi(\uu))=\ff(\xx_0;\uu,\uu^{f-})$ and $|\rr^s_1(\uu)|>0$. Then the conclusion easily follows.
\end{proof}

The following is the half $i$-arch bridge version of the above whose proof uses similar techniques.
\begin{proposition}\label{sigtype-siblingh1}
Suppose $\uu\in\Vf_i(\xx_0)$ and $h(\uu)=1$. If $\uu\neq\fmin_i(\xx_0)$ then $$\Theta(\qq_{-i}(\xx_0;\uu))=F(i).$$
\end{proposition}

The following two results are the opposite parity versions of Propositions \ref{sigtype-sibling} and \ref{sigtype-siblingh1} whose proof uses similar techniques but Proposition \ref{Forkinglocationopposite} as the main tool instead of Proposition \ref{Forkinglocation}.
\begin{proposition}\label{sigtype-sibling-opp}
Suppose $\uu\in\Vf_i(\xx_0)$ and $h(\uu)>1$. If $\theta(\beta(\uu))=\partial$, $\uu\neq\fmax_\partial(\pi(\uu))$ then $\Theta(\qq_{-\partial}(\xx_0;\uu))=F(\partial)$.
\end{proposition}

\begin{proposition}\label{sigtype-sibling-opph1}
Suppose $\uu\in\Vf_i(\xx_0)$ and $h(\uu)=1$. If $\uu\neq\fmax_i(\xx_0)$ then $$\Theta(\qq_i(\xx_0;\uu))=F(-i).$$
\end{proposition}

Combining Propositions \ref{sigtypeparentchild}-\ref{sigtype-sibling-opph1} with Corollary \ref{breakdownHreduced} we get the following result.

\begin{corollary}\label{sigtypemaxpaths}
Suppose $|\Vf_i(\xx_0)|>1$, $\uu\in\Vf_i(\xx_0)$ and $h(\uu)\geq1$. Then the following hold.
\begin{enumerate}
\item $\Theta(\xx_0;\ff(\xx_0;\uu_{\mathrm{min}}))=F(i)$;
\item $\Theta(\ff(\xx_0;\uu,\uu^{f-});\ff(\xx_0;\uu_{[-\phi(\uu)]}))=F(-\phi(\uu))$ if $\uu^{f-}$ exists;
\item $\Theta(\ff(\xx_0;\uu,\uu^{f+});\ff(\xx_0;\uu_{[\phi(\uu)]}))=F(\phi(\uu))$ if $\uu^{f+}$ exists;
\item $\Theta(\mm_i(\xx_0);\ff(\xx_0;\uu_{\mathrm{max}}))=F(-i)$.
\end{enumerate}
\end{corollary}

The last equality follows from Proposition \ref{sigtypeparentmaxchildh1} once we recall that $\delta(\xx_0;\mm_i(\xx_0))=i$ and that there is no syllable $\alpha$ with $\theta(\alpha)=i$ such that $\alpha\mm_i(\xx_0)$ is a string.

In view of Definition \ref{newchildren}, there is an obvious generalization of Corollary \ref{sigtypemaxpaths} to $\bVf_i(\xx_0)$.

\section{Terms from bridge quivers}\label{brquivtoterm}
Now we use the notation introduced in \S \ref{motivaterms} to get useful consequences of the results in the above sections.

Given $\yy\in H_l^i(\xx_0)$, $\yy\neq\xx_0$, say that $\nu\in\Ml$ is a \emph{label} of the path from $\xx_0$ to $\yy$ in $H_l^i(\xx_0)$, written $\nu(\xx_0)=\yy$, if one of the following holds:
\begin{itemize}
    \item $i=1$ and $\OT(\nu)$ is isomorphic to the interval $(\xx_0,\yy]$ in $H_l^1(\xx_0)$;
    \item $i=-1$ and $\OT(\nu)$ is isomorphic to the interval $[\yy,\xx_0)$ in $H_l^{-1}(\xx_0)$.
\end{itemize}

Here is an observation.
\begin{rmk}\label{signdeterminedbyoperator}
Suppose $(\la\ast j)(\xx)$ exists for some string $\xx$ for some $j\in\{1,-1\}$. Then $(\la\ast-j)((\la\ast j)(\xx))$ does not exist. Hence $(\la\ast j)(\xx)$ is not a forking string.

Moreover if $(\la\ast j)^n(\xx)=\zz\yy$ for some $n\geq1$, forking string $\yy$ with $(|\yy|-|\xx|)|\zz|>0$ then $\theta(\zz)=-j$. As a consequence, $\Theta(\xx;\brac1{\la\ast j}(\xx))=F(j)$ if $\brac1{\la\ast j}(\xx)$ exists, otherwise $\Theta(\xx;[1,\la\ast j](\xx))=F(j)$.
\end{rmk}

\begin{proposition}\label{operatordeterminedbysign}
Suppose $\xx$ is a proper left substring of $\yy$ and $\Theta(\xx;\yy)=F(j)$ for some $j\in\{1,-1\}$.
\begin{enumerate}
    \item If $\yy$ is a left $\N$-string then $\yy=\brac1{\la\ast j}(\xx)$.
    \item If $\yy$ is a left torsion string then $\yy=[1,\la\ast j](\xx)$.
\end{enumerate}

In addition, if $\xx$ is a forking string then $\xx$ is a fundamental solution of either $(1)$ or $(2)$ above.
\end{proposition}

\begin{proof}
We only prove the first statement; the proof of the second is similar.

Without loss assume that $j=1$. Note that $\la(\xx)$ exists since $\theta(\xx;\yy)=1$.

Suppose $\la^n(\xx)$ exists for some $n\geq1$ and $\la^{n-1}(\xx)$ is a left substring of $\yy$. If $\la^n(\xx)$ is not a left substring of $\yy$ then let $\ww$ be the maximal common forking left substring of $\la^n(\xx)$ and $\yy$. Remark \ref{signdeterminedbyoperator} yields that $|\ww|\geq|\la^{n-1}(\xx)|+1$. Thus by the definition of $\la$ we have $\theta(\ww;\la^n(\xx))=-1$, and hence $\theta(\ww;\yy)=1$, a contradiction to $\Theta(\xx;\yy)=F(1)$. Thus we have concluded that $\la^n(\xx)$ is a left substring of $\yy$. Since $\yy$ is infinite and $\la^n(\xx)$ is not a forking string by Remark \ref{signdeterminedbyoperator}, it is readily seen that $\la^{n+1}(\xx)$ exists. Thus the argument is complete using induction.

The last statement follows from Remark \ref{signdeterminedbyoperator}.
\end{proof}

The next result follows by combining Proposition \ref{operatordeterminedbysign} and Corollary \ref{sigtypemaxpaths}.
\begin{corollary}\label{fundamentalsol}
Suppose $|\Vf_i(\xx_0)|>1$, $\uu\in\Vf_i(\xx_0)$ and $h(\uu)\geq1$. Then the following hold.
\begin{enumerate}
\item $\ff(\xx_0;\uu_{\mathrm{min}})=\begin{cases}\brac1{\la\ast i}(\xx_0)&\mbox{if }\uu_{\mathrm{min}}\in\Uf_i(\xx_0)\cup\Hf_i(\xx_0);\\ [1,\la\ast i](\xx_0)&\mbox{if }\uu_{\mathrm{min}}\in\Rf_i(\xx_0)\cup\Zf_i(\xx_0).\end{cases}$
\item $\ff(\xx_0;\uu_{[-\phi(\uu)]})=\begin{cases}\brac1{\la\ast(-\phi(\uu))}(\ff(\xx_0;\uu,\uu^{f-}))&\mbox{if }\uu_{[-\phi(\uu)]}\in\Uf_i(\xx_0)\cup\Hf_i(\xx_0);\\ [1,\la\ast(-\phi(\uu))](\ff(\xx_0;\uu,\uu^{f-}))&\mbox{if }\uu_{[-\phi(\uu)]}\in\Rf_i(\xx_0)\cup\Zf_i(\xx_0),\end{cases}$\\
if $\uu^{f-}$ exists.
\item $\ff(\xx_0;\uu_{[\phi(\uu)]})=\begin{cases}\brac1{\la\ast\phi(\uu)}(\ff(\xx_0;\uu,\uu^{f+}))&\mbox{if }\uu_{[\phi(\uu)]}\in\Uf_i(\xx_0)\cup\Hf_i(\xx_0);\\ [1,\la\ast\phi(\uu)](\ff(\xx_0;\uu,\uu^{f+}))&\mbox{if }\uu_{[\phi(\uu)]}\in\Rf_i(\xx_0)\cup\Zf_i(\xx_0),\end{cases}$\\
if $\uu^{f+}$ exists.
\item $\ff(\xx_0;\uu_{\mathrm{max}})=\begin{cases}\brac1{\la\ast-i}(\mm_i(\xx_0))&\mbox{if }\uu_{\mathrm{max}}\in\Uf_i(\xx_0)\cup\Hf_i(\xx_0);\\ [1,\la\ast-i](\mm_i(\xx_0))&\mbox{if }\uu_{\mathrm{max}}\in\Rf_i(\xx_0)\cup\Zf_i(\xx_0).\end{cases}$
\end{enumerate}
\end{corollary}

We only justify the last identity. Using the fact that $\delta(\xx_0;\mm_i(\xx_0))=i$ and the definition of $\fmax_i(\xx_0)^{\la\ast-i}$ we clearly have $|\fmax_i(\xx_0)^{\la\ast-i}|\geq|\mm_i(\xx_0)|$. The reverse inequality follows by combining Proposition \ref{sigtypeparentmaxchildh1} with Proposition \ref{operatordeterminedbysign}.

\begin{definition}
Given $\uu\in\Vf_i(\xx_0)$ and $j\in\{1,-1\}$, define $\kappa_j(\uu)$ to be the minimal height element of $\mathcal P_j(\uu)$ such that $\mathcal P_j(\kappa_j(\uu))=\mathcal P_j(\uu)$.
\end{definition}

\begin{rmk}\label{kappapredecessor}
Suppose $\uu\in\Vf_i(\xx_0)$, $\phi(\uu)\neq0$ and $j\in\{1,-1\}$. Then the following hold.
\begin{enumerate}
\item $\uu=\kappa_{\phi(\uu)}(\uu)$;
\item $\uu^{f-}$ exists if and only if $\uu=\kappa_{-\phi(\uu)}(\uu)$;
\item If $\phi(\kappa_j(\uu))=-j$ then $\kappa_j(\uu)^{f-}$ exists;
\item If $\phi(\kappa_{-\phi(\uu)}(\uu))=\phi(\uu)$ then $\kappa_{-\phi(\uu)}(\uu)^{f-}$ exists.
\end{enumerate}
\end{rmk}

Using Corollary \ref{fundamentalsol} and Remark \ref{kappapredecessor} we get the following.
\begin{corollary}\label{linearsol}
Suppose $\uu\in\bVf_i(\xx_0)$ and $j\in\{1,-1\}$. Then the fundamental solution, say $\uu^{\la\ast j}$, of
\begin{equation}\label{basicsol}
\ff(\xx_0;\uu_{[j]})=\begin{cases}\brac1{\la\ast j}(\xx)&\mbox{if }\uu_{[j]}\in\Uf_i(\xx_0)\cup\Hf_i(\xx_0);\\ [1,\la\ast j](\xx)&\mbox{if }\uu_{[j]}\in\Rf_i(\xx_0)\cup\Zf_i(\xx_0),\end{cases}
\end{equation}
is given by
$\uu^{\la\ast j}=\begin{cases}\xx_0&\mbox{if } \phi(\kappa_j(\uu))=0;\\\ff(\xx_0;\kappa_j(\uu),\kappa_j(\uu)^{f-})&\mbox{if }\phi(\kappa_j(\uu))=-j;\\\ff(\xx_0;\kappa_j(\uu),\kappa_j(\uu)^{f+})&\mbox{if }\phi(\kappa_j(\uu))=j\mbox{ and }(j,\kappa_j(\uu))\neq(-i,\fmax_i(\xx_0));\\\mm_i(\xx_0)&\mbox{if }(j,\kappa_j(\uu))=(-i,\fmax_i(\xx_0)).\end{cases}$
\end{corollary}

\begin{lemma}\label{forkcharacter}
Suppose $F_l^i(\xx_0):=\{\xx\in H_l^i(\xx_0)\mid\xx\mbox{ is a forking string}\}\cup\{\xx_0,\mm_i(\xx_0)\}$. Then $F_l^i(\xx_0)=\{(\uu,(s_1,s_2,\hdots,s_{h(\uu)-1}))^{\la\ast\phi(\uu)}\mid(\uu,(s_1,s_2,\hdots,s_{h(\uu)-1}))\in\bVf_i(\xx_0)\}$.
\end{lemma}

\begin{proof}
The inclusion of the right hand side into the left hand side follows from Corollary \ref{linearsol}. The other inclusion follows by Proposition \ref{operatordeterminedbysign}. 
\end{proof}

The assignment of $\ff(\xx_0;\bar\uu)$ for $\bar\uu\in\bVf_i(\xx_0)$ as in Definition \ref{newchildren} has a simple consequence.
\begin{proposition}\label{iterationsol}
Suppose $(\uu,(s_1,\hdots,s_{h(\uu)-1}))\in\bVf_i(\xx_0),h(\uu)\geq1$ and $\bar\uu\in\xif(\uu)$. Then
\begin{equation}
\Red{\sbq(\uu)}((\bar\uu,(s_1,\hdots,s_{h(\uu)-1},s+1))^{\la\ast\phi(\uu)})=(\bar\uu,(s_1,\hdots,s_{h(\uu)-1},s))^{\la\ast\phi(\uu)},
\end{equation}
and hence
\begin{equation}\label{limitsolgen}
\lim_{s\to\infty}(\bar\uu,(s_1,\hdots,s_{h(\uu)-1},s))^{\la\ast\phi(\uu)}=\ff(\xx_0;(\uu,(s_1,\hdots,s_{h(\uu)-1}))).
\end{equation}
\end{proposition}

Our goal is to associate a term $\mu_i(\xx_0;\uu)$ to each non-root $\uu\in\Vf_i(\xx_0)$ using induction on subtrees, such that
\begin{equation}\label{mudefinition}
\mu_i(\xx_0;\uu)(\xx_0)=\uu^{\la\ast\phi(\uu)}.
\end{equation}

\begin{algo}\label{termformation}
Suppose $\uu\in\Vf_i(\xx_0)$, $h(\uu)\geq1$ and $\uu_{1,\partial}\sqsf_\partial\hdots\sqsf_\partial\uu_{{m_\partial},\partial}$ are all the elements of $\xif_\partial(\uu)$ for $m_\partial\in\mathbb{N}$ and $\partial\in\{+,-\}$. Suppose for each $\uu_{k,\partial}\in\xif_\partial(\uu)$ we already have inductively defined terms $\tau^j(\uu_{k,\partial})$, for $j\in\{+,-\}$, satisfying
\begin{equation}
\ff(\xx_0;\uu_{k,\partial})=\begin{cases}\brac1{\tau^j(\uu_{k,\partial})}(\uu_{k,\partial}^{\la\ast j})&\mbox{if }\uu_{k,\partial}\in\Uf_i(\xx_0)\cup\Hf_i(\xx_0);\\ [1,\tau^j(\uu_{k,\partial})](\uu_{k,\partial}^{\la\ast j})&\mbox{if }\uu_{k,\partial}\in\Rf_i(\xx_0)\cup\Zf_i(\xx_0).\end{cases}
\end{equation}

We abbreviate the above equations as follows:
\begin{equation}
\ff(\xx_0;\uu_{k,\partial})=\{1,\tau^j(\uu_{k,\partial})\}(\uu_{k,\partial}^{\la\ast j}).
\end{equation}

Define $\tau^j(\uu):=\begin{cases}\{\tau^{-j}(\uu_{m_{-j},-j}),\tau^j(\uu_{m_{-j},-j})\}\hdots\{\tau^{-j}(\uu_{1,-j}),\tau^j(\uu_{1,-j})\}&\mbox{if }m_{-j}>0;\\\la\ast j&\mbox{otherwise}.\end{cases}$
\end{algo}

\begin{proposition}\label{localsol}
With the above hypotheses and notations, we have
\begin{equation}\label{indsol}
\ff(\xx_0;\uu)=\{1,\tau^j(\uu)\}(\uu^{\la\ast j})
\end{equation}
\end{proposition}

\begin{proof}
Without loss we assume that $\uu\in\Uf_i(\xx_0)\cup\Hf_i(\xx_0)$ and that $m_{-j}>0$. The proofs in the other cases are simpler, and hence omitted.

First note by definition of $\uu_{1,-j}$ that $\mathcal P_j(\uu_{1,-j})=\mathcal P_j(\uu)$. Hence $\kappa_j(\uu_{1,-j})=\kappa_j(\uu)$. Hence by Corollary \ref{linearsol} we obtain 
\begin{equation}\label{startsol}
\uu_{1,-j}^{\la\ast j}=\uu^{\la\ast j}.    
\end{equation}
If $1\leq k<m_{-j}$ then the same corollary gives that
\begin{equation}\label{neighboursol}
\uu_{k,-j}^{\la\ast-j}=\ff(\xx_0;\kappa_{-j}(\uu_{k,-j}),\kappa_{-j}(\uu_{k,-j})^{f+})=\ff(\xx_0;\kappa_{-j}(\uu_{k+1,-j}),\kappa_{-j}(\uu_{k+1,-j})^{f-})=\uu_{k+1,-j}^{\la\ast j}.
\end{equation}
In fact if we extend the list $\xif_{-j}(\uu)$ by adding $(\uu_{k,-j},(0,\hdots,0,s))$ for $s\geq0$ and $1\leq k\leq m_{-j}$ as in Definition \ref{newchildren} then also Equation \eqref{neighboursol} holds true. Hence
\begin{equation}\label{succesorsol}
\{\tau^{-j}(\uu_{k,-j}),\tau^j(\uu_{k,-j})\}(\uu_{k,-j},(0,\hdots,0,s))^{\la\ast j}=((\uu_{k,-j},(0,\hdots,0,s))^{f+})^{\la\ast j}.
\end{equation}

Finally for each $1\leq k\leq m_{-j}$ Proposition \ref{iterationsol} yields
\begin{equation}\label{limitsol}
\lim_{s\to\infty}(\uu_{k,-j},(0,\hdots,0,s))^{\la\ast j}= \lim_{s\to\infty}(\uu_{k,-j},(0,\hdots,0,s))^{\la\ast-j}=\ff(\xx_0;\uu).
\end{equation}

The required equation in the statement now follows from Equations \eqref{startsol}, \eqref{neighboursol},\eqref{succesorsol} and \eqref{limitsol}.
\end{proof}

There is an obvious generalization of the above when $\uu\in\bVf_i(\xx_0)$.

The string $\uu^{\la\ast j}$ may not be the fundamental solution of Equation \eqref{indsol} as the following example demonstrates.

\begin{example}
In the algebra $\Gamma^{(vi)}$ in Figure \ref{fundamental} consider $\xx_0:=1_{(\vv_1,i)}$, $\uu:=feca$ and $j=-1$. It is readily verified that $\uu^{\la\ast j}=ab$ and $\tau^j(\uu)=\brac{\la}{\lb}$ by Algorithm \ref{termformation} but the fundamental solution of Equation \eqref{indsol} is $b$.
\end{example}

\begin{lemma}\label{kappasolution}
Suppose $\uu\in\Vf_i(\xx_0)$, $\phi(\uu)\neq0$ and $j\in\{1,-1\}$.
Then
\begin{equation}\label{kappasolutioneqn}
\uu^{\la\ast-\phi(\uu)}=\begin{cases}\kappa_{-\phi(\uu)}(\uu)^{\la\ast-\phi(\uu)}&\mbox{if }\phi(\kappa_{-\phi(\uu)}(\uu))\neq\phi(\uu);\\
(\kappa_{-\phi(\uu)}(\uu)^{f-})^{\la\ast\phi(\uu)}=\kappa_{-\phi(\uu)}(\uu)^{\la\ast-\phi(\uu)}&\mbox{if }\phi(\kappa_{-\phi(\uu)}(\uu))=\phi(\uu).\end{cases}
\end{equation}
\end{lemma}

\begin{proof}
Since $\mathcal P_{-\phi(\uu)}(\kappa_{-\phi(\uu)}(\uu))=\mathcal P_{-\phi(\uu)}(\uu)$, we get $\uu^{\la\ast-\phi(\uu)}=\kappa_{-\phi(\uu)}(\uu)^{\la\ast-\phi(\uu)}$. 

If $\phi(\kappa_{-\phi(\uu)}(\uu))=\phi(\uu)$ then Remark \ref{kappapredecessor}(1) ensures the existence of $\kappa_{-\phi(\uu)}(\uu)^{f-}$. Finally the proof of $(\kappa_{-\phi(\uu)}(\uu)^{f-})^{\la\ast\phi(\uu)}=\kappa_{-\phi(\uu)}(\uu)^{\la\ast-\phi(\uu)}$ is similar to the proof of Equation \eqref{neighboursol} using Corollary \ref{linearsol}.
\end{proof}

Now we are ready to inductively define the term $\mu_i(\xx_0;\uu)$ for $\uu\in\Vf_i(\xx_0)$.

\begin{equation}\label{muexpression}
\mu_i(\xx_0;\uu):=\begin{cases}1&\mbox{if }\phi(\uu)=0;

\\\{\tau^{\phi(\uu)}(\uu),\tau^{-\phi(\uu)}(\uu)\}\mu_i(\xx_0;\kappa_{-\phi(\uu)}(\uu)^{f-})&\mbox{if }\phi(\uu)=\phi(\kappa_{-\phi(\uu)}(\uu));

\\\{\tau^{\phi(\uu)}(\uu),\tau^{-\phi(\uu)}(\uu)\}\mu_i(\xx_0;\kappa_{-\phi(\uu)}(\uu))&\mbox{ if }\phi(\uu)\neq\phi(\kappa_{-\phi(\uu)}(\uu)).
\end{cases}
\end{equation}

Combining Corollary \ref{linearsol}, Proposition \ref{localsol} and Lemma \ref{kappasolution} we get the main result of this section. 
\begin{theorem}
For each $\uu\in\Vf_i(\xx_0)$, Equation \eqref{mudefinition} indeed holds. In particular,
$$\mu_i(\xx_0;\fmax_i(\xx_0))(\xx_0)=\mm_i(\xx_0).$$
\end{theorem}

\begin{corollary}\label{hammocksaredlofpb}
If $\Lambda$ is a domestic string algebra, $\xx_0$ is a string for $\Lambda$ and $i\in\{1,-1\}$ then $(H_l^i(\xx_0),<_l)\in\dLOfpb$, where $\dLOfpb$ is the class of finitely presented bounded discrete linear orders.
\end{corollary}

\begin{example}\label{maxtermcompute}
Continuing from Example \ref{Ex}, choosing $\xx_0:=1_{(\vv_1,i)}$ the only non-root vertices are $\uu_0:=a$ and $\uu_1:=ec$, where $\pi(\uu_1)=\uu_0$ and $\pi(\uu_0)=\xx_0$. Here $\mm_{-1}(\xx_0)=eca$. Since $\uu_1$ is a leaf we get $\tau^1(\uu_1)=\la$ and $\tau^{-1}(\uu_1)=\lb$ by Algorithm \ref{termformation}. Moreover $\xif_{-1}(\uu_0)=\{\uu_1\}$, $\xif_1(\uu_0)=\emptyset$ and hence the same algorithm gives $\tau^1(\uu_0)=\brac{\tau^{-1}(\uu_1)}{\tau^1(\uu_1)}=\brac{\lb}{\la}$ and $\tau^{-1}(\uu_0)=\lb$. Since $\kappa_{-1}(\uu_0)=\xx_0$  and $\phi(\xx_0)=0$, Corollary \ref{linearsol} gives $\uu_0^{\lb}=\xx_0$. The same corollary gives $\uu_0^\la=\mm_{-1}(\xx_0)$. Therefore we get $$\mu_{-1}(\xx_0;\fmax_{-1}(\xx_0))=\brac{\brac{\lb}{\la}}{\lb}.$$
\end{example}

The definition of $\mu_i(\xx_0;\mbox{-})$ naturally extends to $\bVf_i(\xx_0)$. Indeed the computation of $\mu_i(\xx_0;\bar\uu)$ for $\bar\uu:=(\uu,(s_1,s_2,\hdots,s_{h(\uu)-1}))\in\bVf_i(\xx_0)$ needs at most $2\prod_{j=1}^{s_{h(\uu)}-1}(s_j+1)$ many iterations of Algorithm \ref{termformation}.

For $\yy\in H_l^i(\xx_0)$ the following result gives the recipe to find a term $\tau_i(\xx_0;\yy)$ such that $$\tau_i(\xx_0;\yy)(\xx_0)=\yy.$$
\begin{lemma}\label{findingexactterm}
Suppose $\yy\in H_l^i(\xx_0)$. Then $$\tau_i(\xx_0;\yy)=(\la\ast j)^n\mu_i(\xx_0;(\uu,(s_1,\hdots,s_{h(\uu)-1})))$$ for some 
$(\uu,(s_1,\hdots,s_{h(\uu)-1}))\in\bVf_i(\xx_0)$, $n\in\N$ and $j\in\{1,-1\}$.
\end{lemma}

\begin{proof}
Recall from Lemma \ref{HBDLO} that $H_l^i(\xx_0)$ is a bounded discrete linear order. If $\xx\in H_l^i(\xx_0)$ is a forking string then both $\la(\xx)$ and $\lb(\xx)$ exist. Since $|\xx|<|\la(\xx)|$ and $|\xx|<|\lb(\xx)|$, we get that the sequence of lengths has a local minimum at $\xx$.

Hence if $\yy\neq(\uu,(s_1,\hdots,s_{h(\uu)-1}))^{\la\ast\phi(\uu)}$ for any $(\uu,(s_1,\hdots,s_{h(\uu)-1}))\in\bVf_i(\xx_0)$ then $\yy$ is not a forking string and the sequence of lengths of strings does not have a local minimum at $\yy$. Hence $\yy=(\la\ast j)^n(\yy')$ for some $n\geq 1$, $j\in\{1,-1\}$ and $\yy'\in F_l^i(\xx_0)$. The conclusion then follows from Lemma \ref{forkcharacter}.
\end{proof}

This defines a map $\mu_i(\xx_0;\mbox{-}):H_l^i(\xx_0)\to\Ml_{\la\ast i}$.

\begin{example}\label{intermediatetermcompute}
Continuing from Example \ref{maxtermcompute}, we compute $\tau_{-1}(\xx_0;ca)$. 

Since $\kappa_1(\uu_1)=\uu_0$, the last clause of Corollary \ref{linearsol} gives that $\uu_1^{\la}=\mm_{-1}(\xx_0)$. Since $\kappa_{-1}(\uu_1)=\uu_1$, we get $\uu_1^{\lb}=a$ by the third clause of the same corollary. Therefore we get
$$a=\mu_{-1}(\xx_0;\uu_1)(\xx_0)$$ by Equation \eqref{mudefinition}. Now using the computation in Example \ref{maxtermcompute} together with Equation \eqref{muexpression} we have
$$\mu_{-1}(\xx_0;\uu_1)=\brac{\tau^{-1}(\uu_1)}{\tau^1(\uu_1)}\mu_{-1}(\xx_0;\uu_0)=\brac{\lb}{\la}\brac{\brac{\lb}{\la}}{\lb}.$$ Since $ca=\lb(a)$ we get $$\tau_{-1}(\xx_0;ca)=\lb\brac{\lb}{\la}\brac{\brac{\lb}{\la}}{\lb}$$.
\end{example}

The following result is neither surprising nor difficult to prove using the theory of decorated trees and terms that we have developed so far.
\begin{proposition}\label{termreverse}
For $i\in\{1,-1\}$ and $\nu\in\Ml_{\la\ast i}$, there is a domestic string algebra $\Lambda$, a string $\xx_0$ and $\yy(\neq\xx_0)\in H_l^i(\xx_0)$ such that $\nu=\tau_i(\xx_0;\yy)$.
\end{proposition}

\vspace{0.2in}
\noindent{}Shantanu Sardar\\
Indian Institute of Technology Kanpur\\
Uttar Pradesh, India\\
Email: \texttt{shantanusardar17@gmail.com}

\vspace{0.2in}
\noindent{}Corresponding Author: Amit Kuber\\
Indian Institute of Technology Kanpur\\
Uttar Pradesh, India\\
Email: \texttt{askuber@iitk.ac.in}\\
Phone: (+91) 512 259 6721\\
Fax: (+91) 512 259 7500

\begin{thebibliography}{9999}
\bibitem{AKG} Agrawal, S., Kuber, A., Gupta, E.: Euclidean algorithm for a class of linear orders, preprint (2022) available at https://arxiv.org/abs/2202.04282

\bibitem{BR} Butler, M.C., Ringel, C.M.: Auslander-reiten sequences with few middle terms and applications to string algebras, Communications in Algebra, 15(1\&2), 145-179 (1987)


\bibitem{GP} 
Gel’fand, I.M., Ponomarev, V.A.: Indecomposable representations of the Lorentz group.
Russ. Math. Surv. 23(2), 1–58 (1968)

\bibitem{GKS} 
Gupta, E., Kuber, A., Sardar, S.: On the stable radical of some non-domestic string algebras. Algebras and Representation Theory, 25, 1207-1230 (2022)







\bibitem{PS}
Prest, M., Schr\"{o}er, J.: Serial functors, Jacobson radical and representation type. J. Pure
Appl. Algebra, 170(2-3), 295–307 (2002)

\bibitem{R}
Ringel, C.M.: Infinite length modules. Some examples as introduction, Birkhäuser, Basel, 1-73 (2000)

\bibitem{SchroerThesis}
Schr\"{o}er, J.: Hammocks for string algebras. PhD thesis, Universit\"{a}t Bielefeld (1997)

\bibitem{Sch}
Schr\"{o}er, J.: On the infinite radical of a module category. Proc. Lond. Math. Soc. 81(3), 651–674 (2000)

\bibitem{SK} Sardar, S., Kuber, A.: Variations of the bridge quiver for domestic string algebras, preprint (2021) available at https://arxiv.org/abs/2109.06592



















\end{thebibliography}
\end{document}